\documentclass[a4paper,reqno,12pt]{amsart}
\usepackage[margin=1.2in]{geometry}
\usepackage{indentfirst}
\usepackage{amsfonts}
\usepackage{amscd}
\usepackage{amsthm}
\usepackage{amssymb}
\usepackage{amsmath}
\usepackage{thmtools}
\usepackage[all,cmtip]{xy}
\usepackage{tikz}
\usepackage{tikz-cd}
\usepackage{todonotes}
\usepackage{extpfeil}
\usetikzlibrary{decorations.markings,shapes.geometric}
\usepackage{enumitem}
\usepackage{setspace}
\usepackage[hidelinks]{hyperref}
\usetikzlibrary{matrix,arrows,positioning}
\usepackage{upgreek}
\usepackage[final]{microtype}
\def\MR#1{} 

\def\itemNum$#1${\item $\displaystyle#1$
   \hfill\refstepcounter{equation}(\theequation)}

\makeatletter
\providecommand\@dotsep{5}
\renewcommand{\listoftodos}[1][\@todonotes@todolistname]{%
  \@starttoc{tdo}{#1}}
\makeatother

\tikzset{pullback/.style={minimum size=1.2ex,path picture={
\draw[opacity=1,black,-,#1] (-0.5ex,-0.5ex) -- (0.5ex,-0.5ex) -- (0.5ex,0.5ex);%
}}}

\newsavebox{\pullback}
\sbox\pullback{%
\begin{tikzpicture}%
\draw (0,0) -- (1ex,0ex);%
\draw (1ex,0ex) -- (1ex,1ex);%
\end{tikzpicture}}

\newsavebox{\pushout}
\sbox\pushout{%
\begin{tikzpicture}%
\draw (0,1ex) -- (0,0);%
\draw (0,0) -- (1ex,0);%
\end{tikzpicture}}

\newtheorem{Lem}{Lemma}[section]
\newtheorem{Prop}[Lem]{Proposition}

\newtheorem{Def}[Lem]{Definition}

\theoremstyle{plain}
\newtheorem{Thm}[Lem]{Theorem}

\newtheorem{Cor}[Lem]{Corollary}

\theoremstyle{definition}
\newtheorem{Rem}[Lem]{Remark}

\theoremstyle{definition}
\newtheorem{examplex}[Lem]{Example}
\newenvironment{Ex}
  {\pushQED{\qed}\examplex}
  {\popQED\endexamplex}
  
\theoremstyle{plain}
\newtheorem{thmIntro}{Theorem}

\def\pser#1{[\![#1]\!]} 
\def\Lser#1{(\!(#1)\!)} 

\newcommand{\Hom}{\text{\textnormal{Hom}}}
\newcommand{\End}{\text{\textnormal{End}}}

\newcommand{\Aut}{\text{\textnormal{Aut}}}
\newcommand{\Out}{\text{\textnormal{Out}}}

\newcommand{\Spec}{\text{\textnormal{Spec}}\,}

\DeclareMathOperator{\im}{im}
\mathchardef\mhyphen="2D
\newcommand{\git}{/\!\!/}

\newcommand{\ev}{\text{\textnormal{ev}}}

\newcommand{\id}{\text{\textnormal{id}}}
\newcommand{\Vect}{\text{\textnormal{Vect}}}

\newcommand{\Rep}{\text{\textnormal{Rep}}}
\newcommand{\RRep}{\text{\textnormal{RRep}}}

\newcommand{\QR}{QEllR}
\newcommand{\Q}{QEll}
\newcommand{\pt}{\text{\textnormal{pt}}}
\newcommand{\G}{\mathsf{G}}
\newcommand{\Gh}{\hat{\mathsf{G}}}
\newcommand{\Hh}{\mathsf{H}}
\newcommand{\Hhh}{\hat{\mathsf{H}}}
\newcommand{\A}{\mathsf{A}}

\newcommand{\ch}{\text{\textnormal{ch}}}
\newcommand{\ph}{\text{\textnormal{ph}}}
\newcommand{\Ind}{\text{\textnormal{Ind}}}
\newcommand{\RInd}{\text{\textnormal{RInd}}}

\newcommand{\Res}{\text{\textnormal{Res}}}
\newcommand{\orb}{\text{\textnormal{orb}}}
\newcommand{\refl}{\text{\textnormal{ref}}}
\newcommand{\var}{\text{\textnormal{var}}}
\newcommand{\ext}{\text{\textnormal{ext}}}
\newcommand{\tor}{\text{\textnormal{tor}}}
\newcommand{\Loop}{\text{\textnormal{Loop}}}
\newcommand{\Tate}{\text{\textnormal{Tate}}}
\newcommand{\str}{\text{\textnormal{string}}}

\newcommand{\R}{\mathbb{R}}
\newcommand{\Z}{\mathbb{Z}}
\newcommand{\pr}{\text{\textnormal{pr}}}
\newcommand{\Stab}{\text{\textnormal{Stab}}}
\newcommand{\Irr}{\text{\textnormal{Irr}}}

\newcommand{\Bun}{\text{\textnormal{Bun}}}
\newcommand{\Bibun}{\text{\textnormal{Bibun}}}
\newcommand{\Ad}{\text{\textnormal{Ad}}}

\newcommand{\cpt}{\text{\textnormal{cpt}}}
\newcommand{\Grp}{\text{\textnormal{Grp}}}
\newcommand{\AbGrp}{\text{\textnormal{AbGrp}}}

\newcommand{\scale}[1]{(\;\,)_{#1}}
\newcommand{\scaleLam}[1]{(\;\,)_{#1}^{\Lambda}}

\newcommand{\Cyc}{\text{\textnormal{Cyc}}}

\def\Lser#1{(\!(#1)\!)} 

\makeatletter
\newbox\xrat@below
\newbox\xrat@above
\newcommand{\xrightarrowtail}[2][]{%
  \setbox\xrat@below=\hbox{\ensuremath{\scriptstyle #1}}%
  \setbox\xrat@above=\hbox{\ensuremath{\scriptstyle #2}}%
  \pgfmathsetlengthmacro{\xrat@len}{max(\wd\xrat@below,\wd\xrat@above)+.6em}%
  \mathrel{\tikz [>->,baseline=-.75ex]
                 \draw (0,0) -- node[below=-2pt] {\box\xrat@below}
                                node[above=-2pt] {\box\xrat@above}
                       (\xrat@len,0) ;}}
\makeatother

\begin{document}
\title[Real quasi-elliptic cohomology]{Twisted Real quasi-elliptic cohomology}

\author[Z. Huan]{Zhen Huan}
\address{Zhen Huan, Center for Mathematical Sciences,
Huazhong University of Science and Technology, Hubei 430074, China} \curraddr{}
\email{2019010151@hust.edu.cn}

\author[M.\,B. Young]{Matthew B. Young}
\address{Department of Mathematics and Statistics \\ Utah State University\\
Logan, Utah 84322 \\ USA}
\email{matthew.young@usu.edu}

\date{\today}


\begin{abstract}
In this paper we construct twisted Real quasi-elliptic cohomology as the twisted $KR$-theory of loop groupoids. The theory systematically incorporates loop rotation and reflection. After establishing basic properties of the theory, we construct twisted elliptic Pontryagin characters and, without twists, Real analogues of the string power operation of quasi-elliptic cohomology. We also explore the relation of the theory to the Tate curve.
\end{abstract}

\maketitle

\setcounter{tocdepth}{1}

\tableofcontents

\setcounter{footnote}{0}

\section*{Introduction}
\addtocontents{toc}{\protect\setcounter{tocdepth}{1}}

\subsection*{Background and motivation}

An elliptic cohomology theory is an even periodic multiplicative generalized cohomology theory whose associated formal group is a formal completion of an elliptic curve. An idea of Witten, cited in \cite{landweber1988}, is that the elliptic cohomology of a space $X$ is closely related to the $\mathbb{T}$-equivariant $K$-theory of the free loop space $LX = C^{\infty}(S^1, X)$, where the circle $\mathbb{T}$ acts on $LX$ by rotation of loops \cite{witten1988}. Witten argued that there is a close connection between elliptic cohomology and a particular class of two dimensional conformal field theories, showing that the torus partition function in such theories recovers the elliptic genus \cite{witten1987}. Building on this idea, Segal suggested that conformal field theories can be used to construct a cocycle model for elliptic cohomology \cite{segal1988,segal2007}. This idea continues to guide the development of elliptic cohomology and is fundamental to the Stolz--Teichner program \cite{stolz2004,stolz2011}. More generally, given a compact Lie group $\G$ acting on $X$ and a cohomology class $\theta \in H^4_{\G}(X;\Z)$, twisted equviariant versions of elliptic cohomology are expected and, in various settings, have been constructed \cite{ando2010,devoto1996,devoto1998,grojnowski2007,lurie2019,berwick2021,berwick2022}.

One case in which elliptic cohomology is particularly well-understood is at the (singular) Tate curve over $\Spec \Z \Lser{q}$, where the theory reduces to Tate $K$-theory \cite{ando2001}. Motivated by Witten's idea, 
Ganter gave a careful interpretation of the free loop space of a groupoid and used this to construction equivariant Tate $K$-theory as the orbifold $K$-theory of the loop space of a global quotient orbifold \cite[\S 2]{ganter2007}. A twisted generalization of this theory was introduced and studied by Dove \cite{dove2019}.

Motivated by this, the first author introduced quasi-elliptic cohomology, a generalized cohomology theory $\Q^{\bullet}$ which is intermediate to $K$-theory and elliptic cohomology and recovers upon completion Tate $K$-theory \cite{huan2018}. This theory has a natural equivariant generalization: given a compact Lie group $\G$ acting on a space $X$, the $\G$-equivariant quasi-elliptic cohomology of $X$ is
\[
\Q_{\G}^{\bullet}(X) = K^{\bullet}(\Lambda (X \git \G)),
\]
where $\Lambda (X \git \G)$ is the ghost loop space of the groupoid $X \git \G$. Since loop rotation is incorporated into the morphisms of $\Lambda (X \git \G)$, quasi-elliptic cohomology directly incorporates Witten's idea. Via this theory, $\G$-equivariant Tate $K$-theory for any compact Lie group $\G$ can be defined as the $\mathbb{T}$-equivariant $K$-theory of the ghost loops. The construction of quasi-elliptic cohomology in terms of equivariant $K$-theory gives it many computational and conceptual advantages over Tate $K$-theory. Some formulations can be generalized to other equivariant cohomology theories. Other significant features of quasi-elliptic cohomology include naturally defined power operations \cite{huan2018b} and, when $\G$ is finite, twists of $\Q_{\G}^{\bullet}(X)$ by $H^4(B\G ; \Z)$ \cite{huan2020}.

In a different direction, a classical generalization of complex $K$-theory is Atiyah's $KR$-theory \cite{atiyah1966}. Given a space with involution $(X,\sigma_X)$, elements of the Real $K$-theory group $KR(X)$ are constructed from Real vector bundles, that is, complex vector bundles $V \rightarrow X$ together with a lift of $\sigma_X$ to a complex antilinear involution of $V$. Like complex $K$-theory, there are twisted equivariant generalizations of $KR$-theory \cite{atiyah1969,karoubi1970}. These generalizations, together with their complex counterparts, are unified by Freed--Moore $K$-theory \cite{freed2013c,gomi2017}. In particular, given a $\Z_2$-graded group $\pi: \Gh \rightarrow \Z_2$ which acts on a space $X$ and a twisted cohomology class $\hat{\theta} \in H^{3+\pi}(B \Gh; \mathsf{U}(1))$, the associated Freed--Moore $K$-theory is denoted ${^{\pi}}K^{\bullet + \hat{\theta}}_{\Gh}(X)$.

Twisted equivariant $KR$-theory appears naturally in physical contexts which are unoriented, including orientifold string theory \cite{witten1998,gukov2000,braun2002,distler2011} and topological phases of matter with time reversal symmetry \cite{freed2013c}. Roughly speaking, in such settings orientation reversal corresponds to the involution required to formulate $KR$-theory. In view of the relationship between elliptic cohomology and conformal field theory, it is natural to expect a relationship between unoriented conformal field theories and any Real generalization of elliptic cohomology. Since reflections of the circle are symmetries of unoriented theories in two dimensions, the group of rotations $\mathsf{SO}_2(\R) \simeq \mathbb{T}$ of $S^1$ will be enhanced to an action of the orthogonal group $\mathsf{O}_2(\R) = \mathsf{SO}_2(\R) \rtimes \Z_2$. The $\mathsf{O}_2$-equivariant topology and geometry of loop spaces has been studied by a number of authors with connections to dihedral and quaternionic Hochschild homologies \cite{dunn1989,lodder1990,frauenfelder2013,ungheretti2016}. In a different direction, the $\mathsf{O}_2$-equivariant $KR$-theory of loop groups was used by Fok to formulate and prove a Real analogue of the Freed--Hopkins--Teleman theorem \cite{fok2018}, relating twisted equivariant $KR$-theory of compact Lie groups to Real positive energy representations of loop groups.

\subsection*{Main results}

For ease of exposition, in the introduction we restrict attention to finite groups. In the body of the paper, we also treat the case of compact Lie groups under the assumption of trivial twist.

In Section \ref{sec:prelim} we collect required background material on groupoids and their twisted cocycles, Real representation theory and Freed--Moore $K$-theory. One result of independent interest, giving a Mackey-type decomposition of twisted equivariant $KR$-theory in the presence of a subgroup which acts trivially. The proof builds directly on earlier examples of Mackey-type decompositions in the literature  \cite{freed2011,freed2013c,angel2018,gomez2021}.

\begin{thmIntro}[{Theorem \ref{thm:KRMackeyDecomp}}]
\label{thm:KRMackeyDecompIntro}
Let $1 \rightarrow \Hh \rightarrow \Gh \rightarrow \hat{\mathsf{Q}} \rightarrow 1$ be an exact sequence of $\Z_2$-graded finite groups with $\hat{\mathsf{Q}}$ non-trivially graded and $\hat{\theta}$ a Real central extension of $\Gh$. Let $\Gh$ act on a compact Hausdorff space $X$ such that $\Hh$ acts trivially. Then there is an isomorphism
\[
{^{\pi}}K^{\bullet+\hat{\theta}}_{\Gh}(X)
\simeq
{^{\pi}}K^{\bullet+\hat{\nu}}_{\hat{\mathsf{Q}}, \cpt}(X \times \Irr^{\theta}(\Hh))
\]
for an explicitly constructed twist $\hat{\nu}$.
\end{thmIntro}

Motivated by the central role of loop groupoids in quasi-elliptic cohomology, in Section \ref{sec:Realoopgrpd} we present a detailed study of loop groupoids which are enhanced by extra data, including equivariance for loop rotation and reflection and cocycle twists. Let $X$ be a space acted on by a finite $\Z_2$-graded group $\Gh$. Write $\G \leq \Gh$ for $\ker \pi$. Using this data, we define a categorical $\Z_2$-action on the rotation enhanced loop groupoid $\Lambda (X \git \G)$ obtained by reflecting loops and acting on $X \git \G$ by $\Gh \slash \G \simeq \Z_2$. The \emph{unoriented loop groupoid} is defined to be the base of the resulting groupoid double cover $\Lambda (X \git \G) \rightarrow \Lambda_{\pi}^{\refl}(X \git \Gh)$. The natural morphism $\Lambda_{\pi}^{\refl}(X \git \Gh) \rightarrow B \mathsf{O}_2$ which records loop rotation and reflection. We use twisted loop transgression $\uptau_{\pi}^{\refl}(\hat{\alpha}) \in \Z^{2+\pi}(\G \git_R \Gh)$, defined in \cite{noohiyoung2022}, to construct from a cocycle $\hat{\alpha} \in Z^3(B\Gh)$ a Jandl gerbe on $\Lambda_{\pi}^{\refl}(X \git \Gh)$, in the sense of \cite{schreiber2007}.

In Section \ref{sec:QR} we introduce the Real quasi-elliptic cohomology of a $\Gh$-space $X$ twisted by $\hat{\alpha} \in Z^3(B \Gh)$ as
\[
\QR^{\bullet + \hat{\alpha}}_{\Gh}(X) = {^{\pi}}K^{\bullet + \uptau_{\pi}^{\refl}(\hat{\alpha})}(\Lambda_{\pi}^{\refl}(X \git \Gh).
\]
The Freed--Moore twist $\pi$ on the right hand side is induced by the composition $\Lambda_{\pi}^{\refl}(X \git \Gh) \rightarrow B \mathsf{O}_2 \rightarrow B \Z_2$. In Section \ref{QR:prop} we prove that $\QR^{\bullet}$ has many parallels with Freed--Moore $K$-theory. For example, there is a forgetful map $\QR^{\bullet} \rightarrow \Q^{\bullet}$ and $\QR^{\bullet}$ has K\"{u}nneth maps and change-of-group isomorphisms. These constructions are based on corresponding statements for Freed--Moore $K$-theory. We prove in Theorem \ref{thm:QEllRChern} that there are naturally defined elliptic Pontryagin characters, generalizing the elliptic Chern characters of \cite{huan2020}. We also prove the following result, which relates $\QR^{\bullet}$ to a Real version of Tate $K$-theory (see Definition \ref{def:twistedRealTateK}).

\begin{thmIntro}[{Theorem \ref{thm:TateVsQR}}]
\label{thm:TateVsQRIntro}
There is an isomorphism
\[
KR_{\Tate}^{\bullet + \hat{\alpha}}(X \git \G)
\simeq
\QR^{\bullet + \hat{\alpha} }(X \git \G) \otimes_{KR^{\bullet}(\pt)[q^{\pm 1}]} KR^{\bullet}(\pt)\Lser{q}.
\]
\end{thmIntro}

In particular, when $\Gh = \Z_2$ and $\hat{\alpha}$ is trivial, $KR_{\Tate}^{\bullet}(X) \simeq KO^{\bullet}(X) \Lser{q}$ is well-known from topology and appears in relation to the Witten genus \cite{ando2001}.

In the final two sections of the paper, we pursue a relation between Real quasi-elliptic cohomology and the Tate curve. It is shown in \cite[Proposition 4.1.1]{luecke2019} that twisted quasi-elliptic cohomology, realized as $S^1$ completed $K$-theory, is an equivariant elliptic cohomology theory associated to the Tate curve, in the sense of the definition of equivariant elliptic cohomology given in \cite[Section 4]{luecke2019}. In Section \ref{sec:TateReal} we study the Tate curve as a Real space, with involution given by group inverse. As shown in \cite[Remark 3.13]{huan2018}, the $N$-torsion points of the Tate curve are classified by the $\Z_N$-equivariant quasi-elliptic cohomology of a point. In Section \ref{r:QEllRTatecurve} we prove a Real generalization: the $N$-torsion points of the Real Tate curve are classified by the $\Z_2$ homotopy fixed points of the $\Z_N$-equivariant Real quasi-elliptic cohomology of a point. In Section \ref{sec:powerOp} we construct a Real generalization of the power operations for $\Q^{\bullet}$ \cite{huan2018b}. These power operations mix power operations in $KR$-theory with the dilation, rotation and reflection of loops.

\begin{thmIntro}[{Theorem \ref{thm:powOpQR}}]
\label{thm:powOpQRIntro}
There are operations
\[
\{\mathbb{P}^R_N : \QR^{\bullet}_{\G}(X) \rightarrow \QR^{\bullet}_{\G \wr \Sigma_N}(X^N)\}_{N \in \Z_{\geq 0}}
\]
which satisfy Real analogues of the power operation axioms of \cite{ganter2006,huan2018b}, including coherence rules for compositions and preservation of external products.
\end{thmIntro}

We expect that Theorem \ref{thm:powOpQRIntro} can be extended to include twists. Unfortunately, we have been unable to check compatibility of the power operations with cocycle twists, owing to the difficult combinatorics involved in the candidate power operations and twisted transgression maps. However, with this in mind and motivated by the definition of equivariant power operations \cite[Definition 4.3]{ganter2006}, we define a wreath product of twists which, by Lemma \ref{twistprop} and Proposition \ref{FMKPower}, is compatible with equivariant power operations for $KR$-theory.

\subsubsection*{Acknowledgements}
The first author is supported by the Young Scientists Fund of the National Natural Science Foundation of China (Grant No. 11901591) for the project ``Quasi-elliptic cohomology and its application in geometry and topology", and the research funding from Huazhong University of Science and Technology. The second author is partially supported by a Simons Foundation Collaboration Grant for Mathematicians (Award ID 853541) and thanks Behrang Noohi for conversations. This work was initiated while the second author was visiting the Center for Mathematical Sciences at Huazhong University of Science and Technology. Both authors thank the Max Planck Institute for Mathematics for hospitality and support.

\section{Preliminary material}
\label{sec:prelim}

Let $\Z_2$ be the multiplicative group $\{\pm 1\}$ and $\mathbb{T}$ the compact Lie group of unit norm complex numbers. We identify $\mathbb{R} \slash \Z$ with $\mathbb{T}$ by $t \mapsto e^{2 \pi i t}$.

All topological spaces are assumed to be paracompact, locally contractible and completely regular.

\subsection{Groupoids}

We collect basic results about groupoids. See \cite[Appendix A]{freed2011} and \cite[\S 2]{gomi2017} for detailed discussions.

A \emph{groupoid} is a small category all of whose morphisms are isomorphisms. All groupoids in this paper are assumed to be topological. Hence, if $\mathfrak{X}$ is a groupoid, then its sets of objects $\mathfrak{X}_0$ and morphisms $\mathfrak{X}_1$ are topological spaces and all structure maps, such as the source and target maps $\partial_0, \partial_1 : \mathfrak{X}_1 \rightarrow \mathfrak{X}_0$, are continuous.

\begin{Ex}
Let a topological group $\G$ act on a space $X$. The \emph{quotient groupoid} $X \git \G$ has objects $X$ and morphisms $X \times \G$ with $\partial_0(x,g)=x$ and $\partial_1(x,g)=gx$. When $X$ consists of a single point, $X \git \G = B\G$ is the classifying groupoid of $\G$.
\end{Ex}

We work with groupoids up to local equivalence \cite[Definition A.4]{freed2011}. Groupoids $\mathfrak{X}$ and $\mathfrak{Y}$ are called \emph{weak equivalent} if there exists a diagram of local equivalences $\mathfrak{X} \leftarrow \mathfrak{Z}\rightarrow \mathfrak{Y}$. A groupoid $\mathfrak{X}$ is called a \emph{local quotient groupoid} if its coarse moduli space admits a countable open cover $\{U_{\alpha}\}_{\alpha}$ such that each full subgroupoid $\mathfrak{X}_{U_{\alpha}} \subset \mathfrak{X}$ on objects with isomorphism class in $U_{\alpha}$ is weak equivalent to a quotient groupoid $\tilde{U}_{\alpha} \git \G_{\alpha}$ for a Hausdorff space $\tilde{U}_{\alpha}$ and compact Lie group $\G_{\alpha}$

A groupoid $\mathfrak{X}$ is called \emph{essentially finite} if it has only finitely many isomorphism classes of objects and finitely many morphisms between any two objects, in which case it is equivalent to a finite disjoint union of classifying groupoids of finite groups. The simplicial chain complex $C_{\bullet}(\mathfrak{X}; \Z)$ of an essentially finite groupoid $\mathfrak{X}$ is the free abelian group on symbols $[g_n \vert \cdots \vert g_1]x_1$, with $x_1 \xrightarrow[]{g_1} x_2 \xrightarrow[]{g_2} \cdots \xrightarrow[]{g_n} x_{n+1}$ a diagram in $\mathfrak{X}$. When $\mathfrak{X}$ has a single object, we omit it from the notation. Write $C^{\bullet}(\mathfrak{X} ; \mathsf{A}) \subset \Hom_{\AbGrp}(C_{\bullet}(\mathfrak{X}; \Z), \mathsf{A})$ for the complex of normalized simplicial cochains on $\mathfrak{X}$ with coefficients in a multiplicative abelian group $\mathsf{A}$ and $Z^{\bullet}(\mathfrak{X};\mathsf{A})$ and $H^{\bullet}(\mathfrak{X};\mathsf{A})$ for its cocycles and cohomology, respectively. When $\mathsf{A} = \mathbb{T}$, we omit it from the notation, so that, for example, $C^{\bullet}(\mathfrak{X}) = C^{\bullet}(\mathfrak{X} ; \mathbb{T})$.

\begin{Def}\label{defbz2grpd}
A \emph{$B \Z_2$-graded groupoid} is a morphism of groupoids $\pi: \hat{\mathfrak{X}} \rightarrow B\Z_2$. A \emph{morphism} of $B \Z_2$-graded groupoids is a morphism of groupoids which respects the $B \Z_2$-grading.
\end{Def}

Explicitly, $\pi$ is the data of a continuous map $\pi: \hat{\mathfrak{X}}_1 \rightarrow \Z_2$ which satisfies $\pi(\varsigma_2 \circ \varsigma_1) = \pi(\varsigma_2) \pi(\varsigma_1)$ for composable morphisms $\varsigma_1, \varsigma_2 \in \hat{\mathfrak{X}}_1$. The functor $\pi$ determines an equivalence class of groupoid double covers $\mathfrak{X} \rightarrow \hat{\mathfrak{X}}$, for which we fix a representative.

Let $\hat{\mathfrak{X}}$ be a $B\Z_2$-graded essentially finite groupoid. Denote by $C^{\bullet}(\hat{\mathfrak{X}}; \mathsf{A}_{\pi})$ the complex of normalized cochains on $\hat{\mathfrak{X}}$ with coefficients in the local system $\mathsf{A}_{\pi} = \mathfrak{X} \times_{\Z_2} \mathsf{A}$, where $\Z_2$ acts on $\mathsf{A}$ by inversion. Explicitly, the differential of $\hat{\lambda} \in C^{n-1}(\hat{\mathfrak{X}};\mathsf{A}_{\pi})$ is defined by
\begin{multline*}
d \hat{\lambda} ([\varsigma_n \vert \cdots \vert \varsigma_1]x)
=
\hat{\lambda}([\varsigma_{n-1} \vert \cdots \vert \varsigma_1]x)^{\pi(\varsigma_n)}
\prod_{j=1}^{n-1} \hat{\lambda}([\varsigma_n \vert \cdots \vert \varsigma_{j+1} \varsigma_j\vert \cdots \vert \varsigma_1]x)^{(-1)^{n-j}} \times \\
\hat{\lambda}([\varsigma_n \vert \cdots \vert \varsigma_2]\varsigma_1 \cdot x)^{(-1)^n}.
\end{multline*}
As above, we write, for example, $C^{\bullet+\pi}(\mathfrak{X})$ for $C^{\bullet}(\mathfrak{X} ; \mathbb{T}_{\pi})$.

\subsection{$\Z_2$-graded groups}
\label{sec:Z2GrGrps}


A \emph{$\Z_2$-graded group} is a group homomorphism $\pi: \Gh \rightarrow \Z_2$. The \emph{ungraded group} of $\Gh$ is $\G = \ker \pi$. When $\pi$ is non-trivial, $\Gh$ is called a \emph{Real structure} on $\G$. The group $\Gh$ acts on $\G$ by Real conjugation:
\[
\varsigma \cdot g = \varsigma g^{\pi(\varsigma)} \varsigma^{-1},
\qquad
g \in \G, \;\varsigma \in \Gh.
\]
The Real centralizer of $g\in \G$ is the stabilizer subgroup
\[
C^R_{\Gh}(g) = \{ \varsigma \in \Gh \mid \varsigma g ^{\pi(\varsigma)}\varsigma^{-1} = g \} \leq \Gh.
\]
The group $C^R_{\Gh}(g)$ is $\Z_2$-graded with ungraded group the centralizer $C_{\G}(g)$. The Real centre of $\G$ is $\{g \in \G \mid C_{\Gh}^R(g) = \Gh\}$.

Denote by $\G \git \G$ the quotient groupoid resulting from $\G$ acting on itself by conjugation. Its set $\pi_0(\G \git \G)$ of connected components is the set of conjugacy classes of $\G$. A Real structure $\Gh$ on $\G$ determines an involution of $\pi_0(\G \git \G)$ by sending a conjugacy class $\mathcal{O}$ to $\omega^{-1} \mathcal{O}^{-1} \omega$, where $\omega \in \Gh \setminus \G$ is any element. This defines a partition
\begin{equation}
\label{eq:conjClassesDecomp}
\pi_0(\G \git \G)
=
\pi_0(\G \git \G)_{-1} \sqcup \pi_0(\G \git \G)_{+1},
\end{equation}
where $\pi_0(\G \git \G)_{-1}$ is the fixed point set of the involution. Note that the conjugacy class of $g \in \G$ is fixed by the involution if and only if $C_{\Gh}^R(g) \setminus C_{\G}(g) \neq \varnothing$. The set $\pi_0(\G \git_R \Gh)$ of Real conjugacy classes of $\G$ inherits from \eqref{eq:conjClassesDecomp} a partition
\begin{equation}
\label{eq:RealConjClasses}
\pi_0(\G \git_R \Gh)
=
\pi_0(\G \git \G)_{-1} \sqcup \pi_0(\G \git \G)_{+1} \slash \Z_2.
\end{equation}

\begin{Ex}
The terminal $\Z_2$-graded group is $\pi=\id: \Z_2 \rightarrow \Z_2$ and is denoted by $\Z_2$. If $\Z_2$ acts on a group $\Hh$, then so too does any $\Z_2$-graded group $\Gh$ and the resulting semi-direct product $\Hh \rtimes_{\pi} \Gh$ is $\Z_2$-graded.
\end{Ex}

\begin{Ex}
The dihedral group $D_{2n} = \langle r,s \mid r^n =e, \; s^2 =e, \; srs=r^{-1} \rangle$ is a Real structure on the cyclic group $\Z_n$ of order $n$.
\end{Ex}

\begin{Ex}
Let $\mathsf{O}_2 = \mathsf{O}_2(\mathbb{R})$ be the orthogonal group of Euclidean $\mathbb{R}^2$. The determinant $\mathsf{O}_2 \rightarrow \Z_2$ makes $\mathsf{O}_2$ into a Real structure on $\mathsf{SO}_2 \simeq \mathbb{T}$. The exact sequence of groups
\begin{equation}
\label{eq:univCovO2}
1 \rightarrow \Z \xrightarrow[]{n \mapsto (n,1)} \mathbb{R} \rtimes_{\pi} \Z_2 \rightarrow \mathsf{O}_2 \rightarrow 1
\end{equation}
is the universal cover of $\mathsf{O}_2$.
\end{Ex}

\begin{Lem}
\label{lem:invTorsor}
Let a $\Z_2$-graded compact Lie group $\Gh$ act from the right on a topological space $X$. Each element $\omega \in \Gh \setminus \G$ determines an involution $(\iota_{\omega}, \Theta_{\omega})$ of the groupoid $X \git \G$. Moreover, the set $\{(\iota_{\omega}, \Theta_{\omega})\}_{\omega \in \Gh \setminus \G}$ admits the structure of a $\G$-torsor of involutions.
\end{Lem}

\begin{proof}
Let $\iota_{\omega}: X \git \G \rightarrow X \git \G$ be the functor given on objects and morphisms by $\iota_{\omega}(x) = x \omega$ and $\iota_{\omega}(g) = \omega^{-1} g \omega$, respectively. The natural isomorphism $\Theta_{\omega}: \id_{X \git \G} \Rightarrow \iota_{\omega}^2$ with components $\Theta_{\omega,x} = x \omega^2$, $x \in X$, satisfies $\iota_{\omega}(\Theta_{\omega,x}) = \Theta_{\omega, \iota_{\omega}(x)}$, whence $(\iota_{\omega},\Theta_{\omega})$ is an involution of $X \git \G$.

Given $\omega_1, \omega_2 \in \Gh \setminus \G$, let $\varphi_{\omega_2, \omega_1}: \iota_{\omega_1} \Rightarrow \iota_{\omega_2}$ be the natural transformation with components $\varphi_{\omega_2,\omega_1,x} : x \omega_1 \xrightarrow[]{\omega_1^{-1} \omega_2} x \omega_2$, $x \in X$. The diagram
\[
\begin{tikzpicture}[baseline= (a).base]
\node[scale=1.0] (a) at (0,0){
\begin{tikzcd}[column sep=6.5em,row sep=1.5em]
x \arrow{r}[above]{\Theta_{\omega_1,x}} \arrow{d}[left]{\Theta_{\omega_2,x}}& x \omega_1^2 \arrow{d}[right]{\varphi_{\omega_2,\omega_1,\iota_{\omega_1}(x)}} \\
x \omega_2^2 & \arrow{l}[below]{\iota_{\omega_2}(\varphi_{\omega_2,\omega_1,x})} x \omega_1 \omega_2
\end{tikzcd}
};
\end{tikzpicture}
\]
commutes, proving that $(\id_{X \git \G}, \varphi_{\omega, \omega^{\prime}}) : (X \git \G, \iota_{\omega}, \Theta_{\omega}) \rightarrow (X \git \G, \iota_{\omega^{\prime}}, \Theta_{\omega^{\prime}})$ is an equivalence of groupoids with involution. The coherence condition $\varphi_{\omega_3, \omega_2} \circ \varphi_{\omega_2, \omega_1} = \varphi_{\omega_3,\omega_1}$ is a direct computation. The torsor structure is obtained by letting $g \in \G$ act by $\varphi_{g\omega,\omega}: \iota_{\omega} \Rightarrow \iota_{g \omega}$, $\omega \in \Gh \setminus \G$.
\end{proof}

The set of commuting pairs in $\G$ is $\G^{(2)} = \{
(g,h) \in \G \times \G \mid gh=hg\}$. A Real structure $\Gh$ acts on $\G^{(2)}$ by simultaneous Real conjugation:
\[
\omega \cdot (g,h)
=
(\omega g^{\pi(\omega)} \omega^{-1}, \omega h^{\pi(\omega)} \omega^{-1}),
\qquad
\omega \in \Gh.
\]
Let $C^R_{\Gh}(g,h)$ be the $\Gh$-stabilizer of $(g,h) \in \G^{(2)}$ and $C_{\G}(g,h)$ its ungraded group. The Real structure determines a $\Z_2$-action on $\pi_0(\G^{(2)} \git \G)$ by $\omega \cdot (g,h) =(\omega g^{-1} \omega^{-1}, \omega h^{-1} \omega^{-1})$, where $\omega \in \Gh \setminus \G$ is any element. Analogously to the partition \eqref{eq:conjClassesDecomp}, we have
\begin{equation*}
\pi_0(\G^{(2)} \git \G)
=
\pi_0(\G^{(2)} \git \G)_{-1}
\sqcup
\pi_0(\G^{(2)} \git \G)_{+1}
\end{equation*}
in which $(g,h)$ lies in the $-1$ summand if and only if $C_{\Gh}^R(g,h) \setminus C_{\G}(g,h) \neq \varnothing$.

\subsection{Real representation theory}
\label{sec:RealRepThy}

We recall basic aspects of the Real representation theory of compact Lie groups, as introduced by Wigner, \cite[\S 26]{wigner1959}, Atiyah--Segal \cite[\S 6]{atiyah1969} and Karoubi \cite[\S I]{karoubi1970}. For recent treatments, see \cite[\S 2.3]{fok2014}, \cite[\S 3.1]{mbyoung2021b}. Given $\epsilon \in \{ \pm 1\}$ and a complex vector space $V$, write
\begin{equation}\label{eq:conjNotation}
{^{\epsilon}}V
=
\begin{cases}
V & \mbox{if } \epsilon =1, \\
\overline{V} & \mbox{if } \epsilon = -1.
\end{cases}
\end{equation}

Let $\G$ be a compact Lie group with fixed Real structure $\Gh$. A \emph{Real representation} of $\G$ (with respect to $\Gh$) is a complex vector space $V$ together with $\mathbb{C}$-linear maps $\rho_V(\omega) : {^{\pi(\omega)}}V \rightarrow V$ which satisfy $\rho_V(e)=\id_V$ and $\rho_V(\omega_2 \omega_1) = \rho_V(\omega_2) \rho_V(\omega_1)$, $\omega_1, \omega_2 \in \Gh$. The category $\RRep(\G)$ of Real representations of $\G$ and their $\Gh$-equivariant $\mathbb{C}$-linear maps is $\mathbb{R}$-linear abelian.

\begin{Ex}
If $\Gh = \G \times \Z_2$ with $\pi$ the projection to $\Z_2$, then $\RRep(\G) \simeq \Rep_{\mathbb{R}}(\G)$. More generally, if $\Gh \simeq \G \rtimes \Z_2$ is a split Real structure on $\G$, then $\RRep(\G)$ is the category of Real representations of $\G$ in the sense of Atiyah--Segal \cite[\S 6]{atiyah1969}.
\end{Ex}

The \emph{Real representation ring} $RR(\G)$ is the Grothendieck group of $\RRep(\G)$ with ring structure induced by $\otimes_{\mathbb{C}}$. There is an isomorphism of abelian groups
\begin{equation}
\label{eq:RRDecomp}
RR(\G)
\simeq
RR(\G,\mathbb{R}) \oplus RR(\G,\mathbb{C}) \oplus RR(\G,\mathbb{H}),
\end{equation}
where $RR(\G,\mathbb{F})$ is the free abelian group on isomorphism classes of irreducible Real representations with commuting field $\End_{\RRep(\G)}(V,\psi_V) \simeq \mathbb{F}$. The decomposition \eqref{eq:RRDecomp} is proved in \cite[Proposition 8.1]{atiyah1969} under the assumption that $\Gh$ is split, but the proof holds in general.
Note that $RR(\G,\mathbb{R})$ is a subring of $RR(\G)$.

Two examples play a central role in the remainder of the paper.

For the first example, let $\G = \Z_n$. The irreducible complex representations of $\Z_n$ are $U_k = \mathbb{C}$, $k=0, \dots, n-1$, with a chosen generator of $\Z_n$ acting by multiplication by $e^{\frac{2\pi i k}{n}}$. Writing $\zeta$ for the class of $U_1$, the complex representation ring is $R(\Z_n) \simeq \Z[\zeta] \slash \langle \zeta^n-1\rangle$.

\begin{Ex}
Let $\G = \Z_n$ with dihedral Real structure $\Gh = D_{2n}$. Each irreducible complex representation of $\Z_n$ admits a Real structure with commuting field $\mathbb{R}$. It follows that there is a ring isomorphism
\[
RR(\Z_n) = RR(\Z_n,\mathbb{R}) \simeq \Z[\zeta] \slash \langle \zeta^n-1\rangle
\]
where $\zeta$ is the class of $U_1$ with its unique (up to equivalence) Real structure.
\end{Ex}

For the second example, let $\G = \mathbb{T}$. The irreducible complex representations of $\mathbb{T}$ are $U_k = \mathbb{C}$, $k \in \Z$, with $z \in \mathbb{T}$ acting by multiplication by $z^k$ and $R(\mathbb{T}) \simeq \Z[q^{\pm 1}]$, where $q^{\pm 1}$ is the class of $U_{\pm 1}$.

\begin{Ex}[{\cite[Proposition 6.1(i)]{atiyah1969}}] \label{TorthoRR}
Let $\G = \mathbb{T}$ with orthogonal Real structure $\Gh = \mathsf{O}_2$. Each irreducible complex representation admits a Real structure with commuting field $\mathbb{R}$ and
\[
RR(\mathbb{T}) = RR(\mathbb{T},\mathbb{R}) \simeq \Z[q^{\pm 1}].\qedhere
\]
\end{Ex}

Unless mentioned otherwise, we consider only the dihedral Real structure on $\Z_n$ and orthogonal Real structure on $\mathbb{T}$. Note that $D_{2n}$ is naturally a $\Z_2$-graded subgroup of $\mathsf{O}_2$.

Let $\Gh$ be a non-trivially $\Z_2$-graded group. A \emph{Real central extension} of $\Gh$ is an exact sequence of $\Z_2$-graded groups
\[
1 \rightarrow \A \rightarrow \Hhh \rightarrow \Gh \rightarrow 1
\]
in which $\A$ lies in the Real centre of $\Hhh$. In particular, $\Hhh$ is non-trivially $\Z_2$-graded and $\A$ is trivially $\Z_2$-graded and abelian. When $\Hhh \simeq \A \rtimes_{\pi} \Gh$ as exact sequences of $\Z_2$-graded groups, then $\Hhh$ is called a trivial Real central extension. Of particular importance are Real central extensions with\footnote{This is called a $\pi$-twisted central extension in \cite[\S 1.1]{freed2013b}.} $\mathsf{A} = \mathbb{T}$. Given such an extension, a \emph{projective Real representation} of $\G$ is a Real representation of $\Hhh$ in which $\mathbb{T}$ acts by multiplication. If $\Gh$ is finite, then $\hat{\theta} \in Z^{2+\pi}(B\Gh)$ determines a Real central extension ${^{\hat{\theta}}}\Gh = \Gh \times \mathbb{T}$ with multiplication
\[
(\omega_2, z_2) \cdot (\omega_1, z_1)
=
(\omega_2 \omega_1, \hat{\theta}([\omega_2 \vert \omega_1]) z_2 z_1^{\pi(\omega_2)}).
\]
This construction leads to the classification of equivalences classes of Real central extensions of $\Gh$ by $H^{2+\pi}(B\Gh)$.

\subsection{Twisted $K$-theory of groupoids}
\label{sec:equivKRThy}

We recall basic aspects of the twisted $K$-theory of groupoids \cite{freed2013b,gomi2017}, a flexible framework which allows for the simultaneous treatment of complex and Real $K$-theories and their twisted equivariant generalizations. Extend the notation \eqref{eq:conjNotation} in the obvious way to the allow for $V$ to be a complex vector bundle over a space $X$ and $\epsilon: X \rightarrow \Z_2$ a continuous function.

Let $\hat{\mathfrak{X}}$ be a $B\Z_2$-graded groupoid. Write $\hat{\mathfrak{X}}_2$ for the space of composable pairs of morphisms of $\hat{\mathfrak{X}}$ and define $\partial_0, \partial_1, \partial_2: \hat{\mathfrak{X}}_2 \rightarrow \hat{\mathfrak{X}}_1$ by
\[
\partial_0(f_2,f_1) = f_1,
\qquad
\partial_1(f_2,f_1)=f_2 \circ f_1,
\qquad
\partial_2(f_2,f_1)=f_2.
\]
With this notation, a \emph{$\pi$-twisted extension}\footnote{This is an ungraded extension, in the sense of \cite{freed2013b}.} of $\hat{\mathfrak{X}}$ is a pair $(L, \hat{\theta})$ consisting of a hermitian line bundle $L \rightarrow \hat{\mathfrak{X}}_1$ and an isometry $\hat{\theta}: \partial_2^* L \otimes {^{\pi \circ \partial_2}}\partial_0^* L \rightarrow \partial_1^* L$ of hermitian line bundles on $\hat{\mathfrak{X}}_2$. The map $\hat{\theta}$ is required to satisfy an obvious twisted $2$-cocycle condition \cite[Diagram 7.11]{freed2013b}. 

A $\pi$-twisted extension of $\hat{\mathfrak{X}}$ is an example of a $\pi$-twist of $\hat{\mathfrak{X}}$ \cite[Definition 2.3]{gomi2017}. Since we will not encounter general $\pi$-twists, we do not define them here. Instead, we remark that $\pi$-twists form a symmetric monoidal category ${^{\pi}}\mathfrak{Twist}^+(\hat{\mathfrak{X}})$ whose group of equivalence classes of objects is $\pi_0({^{\pi}}\mathfrak{Twist}^+(\hat{\mathfrak{X}})) \simeq H^3(\hat{\mathfrak{X}}; \Z_{\pi})$. Morphisms of $B \Z_2$-graded groupoids induce pullback functors on categories of $\pi$-twists. In particular, if a $\Z_2$-graded finite group $\Gh$ acts on a space $X$ determines a $\pi$-twist of $X \git \Gh$.

Let $\hat{\mathfrak{X}}$ be a $B \Z_2$-graded local quotient groupoid and $\hat{\theta} \in {^{\pi}}\mathfrak{Twist}^+(\hat{\mathfrak{X}})$. The twisted $K$-theory group ${^{\pi}}K^{\bullet + \hat{\theta}}(\hat{\mathfrak{X}})$ is defined in terms of homotopy classes of sections of a bundle of Fredholm operators on a universal $\hat{\theta}$-twisted vector bundle on $\hat{\mathfrak{X}}$ \cite[\S 3]{gomi2017}. The canonical map $\mathfrak{X} \rightarrow \hat{\mathfrak{X}}$ induces a forgetful homomorphism $c: {^{\pi}}K^{\bullet+\hat{\theta}}(\hat{\mathfrak{X}}) \rightarrow K^{\bullet+\theta}(\mathfrak{X})$. The groups ${^{\pi}}K^{\bullet + \hat{\theta}}(\hat{\mathfrak{X}})$ recover various well-known $K$-theories \cite[\S 3.4]{gomi2017}:
\begin{enumerate}[wide,labelwidth=!, labelindent=0pt,label=(\roman*)]
\item If the $B \Z_2$-grading of $\hat{\mathfrak{X}}$ is trivial, so that $\hat{\mathfrak{X}}=\mathfrak{X}$, then ${^{\pi}}K^{\bullet + \hat{\theta}}(\hat{\mathfrak{X}}) = K^{\bullet + \theta}(\mathfrak{X})$ is twisted $K$-theory \cite{freed2011}. In particular, if a compact Lie group $\G$ acts on a space $X$, then $K^{\bullet}(X \git \G) \simeq K^{\bullet}_{\G}(X)$ is $\G$-equivariant $K$-theory \cite{segal1968b}.

\item Let $\Gh = \Z_2$ act on a space $X$. By a slight abuse of notation, denote by $\pi$ the induced map $X \git \Gh \rightarrow B \Z_2$. Then ${^{\pi}}K^{\bullet}(X \git \Gh) \simeq KR^{\bullet}(X)$ is $KR$-theory \cite{atiyah1966}. More generally, if $\Gh = \G \rtimes \Z_2$ is split, then ${^{\pi}}K^{\bullet}(X \git \Gh) \simeq KR^{\bullet}_{\G}(X)$ is $\G$-equivariant $KR$-theory \cite{atiyah1969}. For general $\Gh$, we continue to denote ${^{\pi}}K^{\bullet}(X \git \Gh)$ by $KR^{\bullet}_{\G}(X)$.
\end{enumerate}

We briefly recall vector bundle constructions of ${^{\pi}}K^{\bullet + \hat{\theta}}(X \git \Gh)$ under certain assumptions.

Let $\pi: \Gh \rightarrow \Z_2$ be a $\Z_2$-graded compact Lie group acting on a topological space $X$. A $\pi$-twisted $\Gh$-equivariant vector bundle on $X$ is a complex vector bundle $V \rightarrow X$ with a continuous lift of the $\Gh$-action on $X$ to $V$ given by $\mathbb{C}$-linear isomorphisms
\[
\rho_{x,\omega}: {^{\pi(\omega)}}V_x \rightarrow V_{x \omega},
\qquad
x \in X, \; \omega \in \Gh.
\]
Isomorphism classes of $\pi$-twisted $\Gh$-equivariant vector bundles form a monoid ${^{\pi}}\Vect_{\Gh}(X)$ with respect to direct sum whose Grothendieck group is ${^{\pi}}K_{\G}^0(X)$. The groups ${^{\pi}}K_{\Gh}^n(X)$, $n \in \Z$, are defined similarly using suspensions and Bott periodicity.

Assume now that $\Gh$ is finite and let $\hat{\theta} \in H^{2+\pi}(B \Gh) \simeq H^3(B \Gh; \Z_{\pi})$. Fix a representative $\hat{\theta} \in Z^{2 + \pi}(B \Gh)$, interpreted as a (Real, if $\pi$ is non-trivial) central extension of $\Gh$. Then ${^{\pi}}K^{0 + \hat{\theta}}(X \git \Gh)$ is the Grothendieck group of $(\pi,\hat{\theta})$-twisted $\Gh$-equivariant vector bundles on $X$. The twist $\hat{\theta}$ indicates that the $\Gh$-action on $X$ lifts to a $\hat{\theta}$-projective action on the total space of the bundle; see Section \ref{sec:RealRepThy}. When $\pi$ is non-trivial, a $(\pi,\hat{\theta})$-twisted $\Gh$-equivariant vector bundle is called a $\hat{\theta}$-twisted Real $\G$-equivariant vector bundle. In particular, $KR^{0 + \hat{\theta}}_{\G} (\pt)$ is the free abelian group on isomorphism classes of irreducible $\hat{\theta}$-twisted Real representations of $\G$. The latter isomorphism continues to hold for $\Gh$ compact.

Finally, recall that there is a graded ring isomorphism
\[
KR^{\bullet}(\pt)
\simeq
\Z[\eta, \mu] \slash \langle 2 \eta, \eta^3, \eta \mu, \mu^2 -4 \rangle,
\]
where $\eta$ and $\mu$ have degrees $-1$ and $-4$ and are the classes of the reduced Hopf bundles of $\mathbb{RP}^1$ and $\mathbb{HP}^1$, respectively \cite[\S 8]{atiyah1969}.

\subsection{Mackey-type decomposition of twisted $K$-theory and Atiyah--Segal maps}
\label{sec:mackeyAS}

In this section, we establish some results about twisted $KR$-theory which are used to compare Real quasi-elliptic cohomology with Tate $KR$-theory in Section \ref{sec:RealTateKThy} and construct the elliptic Pontryagin character in Section \ref{sec:ellPhChar}.

Let $
1 \rightarrow \Hh \rightarrow \Gh \rightarrow \hat{\mathsf{Q}} \rightarrow 1
$
be an exact sequence of $\Z_2$-graded compact Lie groups with $\hat{\mathsf{Q}}$ non-trivially graded. Let ${^{\hat{\theta}}}\Gh$ be a Real central extension of $\Gh$ by $\mathbb{T}$ and $\theta$ its restriction to $\Hh$, which is a central extension. There is an exact sequence
\begin{equation}
\label{eq:gradedSESLift}
1 \rightarrow {^{\theta}}\Hh \rightarrow {^{\hat{\theta}}}\Gh \xrightarrow[]{p} \hat{\mathsf{Q}} \rightarrow 1.
\end{equation}

The group $\Gh$ acts on the set $\Irr^{\theta}(\Hh)$ of isomorphism classes of irreducible $\theta$-twisted unitary representations of $\Hh$: for an irreducible $\theta$-twisted representation $\rho_V$ of $\Hh$ and $\omega \in \Gh$, choose a lift $\tilde{\omega} \in {^{\hat{\theta}}}\Gh$ of $\omega$ and define $\omega \cdot \rho_V$ by
\[
(\omega \cdot \rho_V)(\tilde{h}) = \rho_{{^{\pi(\omega)}}V}(\tilde{\omega}^{-1} \tilde{h} \tilde{\omega}),
\qquad
\tilde{h} \in {^{\theta}}\Hh.
\]
If $x \in \Hh$ with lift $\tilde{x} \in {^{\theta}}\Hh$, then $\rho_V(\tilde{x}^{-1}): \rho_V \rightarrow x \cdot \rho_V$ is a ${^{\theta}}\Hh$-equivariant isometry. In particular, $\Hh$ acts trivially on $\Irr^{\theta}(\Hh)$ and there is an induced action of $\hat{\mathsf{Q}}$ on $\Irr^{\theta}(\Hh)$. Let ${^{\hat{\theta}}}\Gh$ act on $\Irr^{\theta}(\Hh)$ through $\Gh$. 

Fix a representative $V$ of each $[V] \in \Irr^{\theta}(\Hh)$. Write $\tilde{\omega} \in {^{\hat{\theta}}}\Gh$ for any lift of $\omega \in \Gh$. By Schur's Lemma, $\tilde{L}_{[V],\tilde{\omega}} = \Hom_{{^{\theta}}\Hh}(W, \omega \cdot V)$ is a hermitian line, where $W$ is the chosen representative of $\omega \cdot [V]$. Following \cite[\S 9.4]{freed2013b}, the associative composition maps
\begin{equation}
\label{eq:nuComposition}
\tilde{L}_{\tilde{\omega}_1 \cdot [V],\tilde{\omega}_2} \otimes {^{\pi(\omega_2)}}\tilde{L}_{[V],\tilde{\omega}_1}
\rightarrow
\tilde{L}_{[V],\tilde{\omega}_2 \tilde{\omega}_1},
\qquad
f_2 \otimes f_1
\mapsto
 \omega_2 \cdot f_1 \circ f_2
\end{equation}
define a $\pi$-twisted extension of $\Irr^{\theta}(\Hh) \git {^{\hat{\theta}}}\Gh$. For $q \in \hat{\mathsf{Q}}$, let $L_{[V],q}$ be the set of all sections $s$ of
\[
\bigcup_{\tilde{\omega} \in p^{-1}(q)} \tilde{L}_{[V],\tilde{\omega}} \rightarrow p^{-1}(q)
\]
with the property that the image of $\rho_W(\tilde{h}) \otimes s(\tilde{\omega})$ under \eqref{eq:nuComposition} is $s(\tilde{h}\tilde{\omega})$ for all $\tilde{h} \in {^{\theta}}\Hh$, where now $W$ is the representative of $q \cdot V$. Exactness of the sequence \eqref{eq:gradedSESLift} implies that $L_{[V],q}$ is one dimensional. The maps \eqref{eq:nuComposition} induce on $\{L_{[V],q}\}_{[V],q}$ the structure of a $\pi$-twisted extension of $\Irr^{\theta}(\Hh) \git \hat{\mathsf{Q}}$, which we denote by $\hat{\nu}$.

We can now state a Real generalization of the Mackey-type decomposition of complex $K$-theory \cite[\S 5]{freed2011b}.

\begin{Thm}
\label{thm:KRMackeyDecomp}
Let $1 \rightarrow \Hh \rightarrow \Gh \rightarrow \hat{\mathsf{Q}} \rightarrow 1$ be an exact sequence of $\Z_2$-graded compact Lie groups with $\hat{\mathsf{Q}}$ non-trivially graded and $\hat{\theta}$ a Real central extension of $\Gh$. Let $\Gh$ act on a compact Hausdorff space $X$ with contractible local slices\footnote{Existence of contractible local slices means that each $x \in X$ admits a closed $\Gh$-stable neighbourhood of the form $\Gh \times_{\Stab_{\Gh}(x)} S_x$ for a slice $S_x$ which is $\Stab_{\Gh}(x)$-equivariantly contractible.} such that $\Hh$ acts trivially. There is an isomorphism
\[
{^{\pi}}K^{\bullet+\hat{\theta}}_{\Gh}(X)
\simeq
{^{\pi}}K^{\bullet+\hat{\nu}}_{\hat{\mathsf{Q}}, \cpt}(X \times \Irr^{\theta}(\Hh)),
\]
where $\hat{\mathsf{Q}}$ acts diagonally on $X \times \Irr^{\theta}(\Hh)$, the pullback of $\hat{\nu}$ along $(X \times \Irr^{\theta}(\Hh)) \git \hat{\mathsf{Q}} \rightarrow \Irr^{\theta}(\Hh) \git \hat{\mathsf{Q}}$ is again denoted by $\hat{\nu}$ and $K_{\cpt}(-)$ is $K$-theory with compact supports.
\end{Thm}

\begin{proof}
The compactness and slice assumptions on $X$, together with Mayer--Vietoris sequences \cite[Theorem 3.11]{gomi2017}, reduce the theorem to the case of $\Stab_{\Gh}(x)$ acting on a singleton $\{x\}$ for some $x \in X$. In this case, the statement reduces to \cite[Theorem 9.42]{freed2013b}; see also \cite[Theorem 9.8]{freed2013b}.
\end{proof}

For later purposes, we reformulate Theorem \ref{thm:KRMackeyDecomp} in terms of $K$-theories of $X$, equivariant with respect to subgroups of $\hat{\mathsf{Q}}$. See \cite[Corollary 3.7]{angel2018} and \cite[Theorem 3.7]{gomez2021} for analogous results for complex $K$-theory. The $\Gh$-stabilizer $\Gh(\rho)$ of $\rho \in \Irr^{\theta}(\Hh)$ contains $\Hh$ as a normal subgroup. Set $\hat{\mathsf{Q}}(\rho) = \Gh(\rho) \slash \Hh$.

\begin{Cor}
\label{cor:KRMackeyDecompReform}
In the setting of Theorem \ref{thm:KRMackeyDecomp}, there is an isomorphism
\begin{equation*}
{^{\pi}}K^{\bullet+\hat{\theta}}_{\Gh}(X)
\simeq
\bigoplus_{\rho \in \Irr^{\theta}(\Hh) \slash \Gh}
{^{\pi}}K^{\bullet + \hat{\nu}_{\rho}}_{\hat{\mathsf{Q}}(\rho)}(X),
\end{equation*}
where $\hat{\nu}_{\rho}$ denotes the restriction of $\hat{\nu}$ to the subgroupoid $\{\rho\} \git \hat{\mathsf{Q}}(\rho) \subset \Irr^{\theta}(\Hh) \git \hat{\mathsf{Q}}$.
\end{Cor}

\begin{proof}
This follows from Theorem \ref{thm:KRMackeyDecomp} and the fact that projective representations of a compact Lie group are discrete.
\end{proof}

The next result is a twisted Real generalization of \cite[Proposition 9.31]{freed2013b} and gives a concrete description of $\hat{\nu}_{\rho}$.

\begin{Prop}
\label{prop:nuRestrict}
Work in the setting of Theorem \ref{thm:KRMackeyDecomp} with the additional assumption that $\Hh$ is abelian. For each $\rho \in \Irr^{\theta}(\Hh)$, the $\pi$-twisted extension $\hat{\nu}_{\rho}$ is realized by the Real central extension $\mathsf{E}$ in the diagram
\[
\begin{tikzpicture}[baseline= (a).base]
\node[scale=1.0] (a) at (0,0){
\begin{tikzcd}[column sep=3em]
1  \arrow{r} & \mathbb{T} \arrow{r} & \mathsf{E} \arrow{r} \arrow[dl, phantom, "\usebox\pushout" , very near start, color=black] & \hat{\mathsf{Q}}(\rho) \arrow{r} \arrow[equal]{d} & 1 \\
1  \arrow{r} & {^{\theta}}\Hh \arrow{r} \arrow[equal]{d} \arrow{u}[left]{\rho} & {^{\hat{\theta}}}\Gh(\rho) \arrow[hook]{d} \arrow{u} \arrow{r} \arrow[dr, phantom, "\usebox\pullback" , very near start, color=black] & \hat{\mathsf{Q}}(\rho) \arrow{r} \arrow[hook]{d}& 1 \\
1  \arrow{r} & {^{\theta}}\Hh \arrow{r} & {^{\hat{\theta}}}\Gh \arrow{r}[below]{p} & \hat{\mathsf{Q}} \arrow{r} & 1
\end{tikzcd}
};
\end{tikzpicture}.
\]
\end{Prop}

\begin{proof}
Since $\Hh$ is abelian, so too is ${^{\theta}}\Hh$. An irreducible $\theta$-twisted unitary representation of $\Hh$ can therefore be interpreted as a group homomorphism $\rho: {^{\theta}}\Hh \rightarrow \mathbb{T}$, as in the statement of the proposition. The group $\mathsf{E}$ can be realized as the associated principal $\mathbb{T}$-bundle ${^{\hat{\theta}}}\Gh(\rho) \times_{{^{\theta}}\Hh} \mathbb{T}$ over $\hat{\mathsf{Q}}(\rho)$. A direct comparison then shows that $\mathsf{E} \times_{\mathbb{T}} \mathbb{C}$ agrees with $\hat{\nu}_{\rho}$. 
\end{proof}

\begin{Ex}
\label{ex:KRPointAllReal}
Let $\Gh$ be a $\Z_2$-graded compact Lie group with the property that all irreducible Real representations of $\G$ have commuting field $\mathbb{R}$, so that $RR(\G)=RR(\G, \mathbb{R})$ and the forgetful map $RR(\G) \rightarrow R(\G)$ is an isomorphism. If $X$ is a compact Hausdorff space on which $\Gh$ acts with $\G$ acting trivially, then Theorem \ref{thm:KRMackeyDecomp} gives a graded ring isomorphism
\begin{equation}
\label{eq:KRTrivActIso}
KR_{\G}^{\bullet}(X)
\simeq
KR^{\bullet}(X) \otimes_{\Z} KR^0_{\G}(\pt).
\end{equation}
See \cite[Proposition 8.1]{atiyah1969} and \cite[Proposition 3.3]{fok2014}, which also give direct proofs of the isomorphism \eqref{eq:KRTrivActIso}.
\end{Ex}

Applied to the groups $\Z_n$ and $\mathbb{T}$, Example \ref{ex:KRPointAllReal} implies graded ring isomorphisms
\[
KR^{\bullet}_{\Z_n}(\pt)
\simeq
KR^{\bullet}(\pt)[\zeta] \slash \langle \zeta^n-1\rangle
\]
and
\[
KR^{\bullet}_{\mathbb{T}}(\pt)
\simeq
KR^{\bullet}(\pt)[q^{\pm 1}],
\]
where $\zeta$ and $q$ have degree zero.

Finally, we require a Real analogue of the Atiyah--Segal localization map. Let a $\Z_2$-graded finite group $\Gh$ act on a finite CW complex $X$ and $\hat{\theta} \in Z^{2+\pi}(B\Gh)$. For $g \in \G$, write $X^g$ for the $g$-fixed point set of $X$ and $\hat{\theta}_g \in Z^{2+\pi}(B C_{\Gh}^R(g))$ for the restriction of $\hat{\theta}$. Restriction along the inclusions $X^g \hookrightarrow X$
defines a map
\begin{equation}
\label{eq:twistedRealAtiyahSegal}
\phi: {^{\pi}}K^{\bullet + \hat{\theta}}_{\Gh}(X)
\rightarrow
\bigoplus_{g \in \pi_0(\G \git_R \Gh) } {^{\pi}}K^{\bullet + \hat{\theta}_g}_{C^R_{\Gh}(g)}(X^g).
\end{equation}
When the $\Z_2$-grading of $\Gh$ is trivial, $\phi \otimes_{\Z} \mathbb{C}$ is an isomorphism of complex vector spaces \cite[Theorem 2]{atiyah1989}, \cite[Theorem 7.3]{adem2003}.

\subsection{Twisted transgression}
\label{sec:transMaps}

Transgression of cocycles on an orbifold $\mathfrak{X}$ to cocycles on its loop space has been studied by many authors \cite{brylinski1993,lupercio2006,willerton2008}. As the cocycles used to twist $K$-theory in this paper are of discrete torsion type, that is, are pulled back along a map $\mathfrak{X} \rightarrow B \G$ for a finite group $\G$, we restrict attention to orbifolds of the form $B \G$. See \cite{ginot2012} for background on group actions on groupoids.

Let $\G$ be a finite group. The \emph{loop (or inertia) groupoid} $\mathcal{L}B \G$ is the category of functors $B \Z \rightarrow B \G$. There is an equivalence $\mathcal{L} B \G \simeq \G \git \G$. The geometric realization of $\mathcal{L}B \G$ is homotopy equivalent to the loop space of the geometric realization of $B \G$ \cite[Proposition 6.29]{strickland2000}, \cite[Theorem 2]{willerton2008}.

Loop transgression is a cochain map $\uptau : C^{\bullet}(B \G) \rightarrow C^{\bullet -1}(\mathcal{L} B \G)$ obtained by push-pull along the correspondence
\begin{equation}\label{diag:loopDiagGrpd}
B \G \xleftarrow[]{\ev} B \Z \times \mathcal{L} B \G \xrightarrow[]{\pr_{\mathcal{L} B\G}} \mathcal{L}B \G
\end{equation}
where $\ev$ is evaluation and $\pr_{\mathcal{L} B\G}$ is the projection. An explicit model for $\uptau$ was given in \cite[Theorem 3]{willerton2008}, which for $\lambda \in C^{n+1}(B \G)$ reads
\[
\uptau(\lambda)([g_n \vert \cdots \vert g_1]\gamma)
=
\prod_{i=0}^n \lambda ([g_n \vert \cdots \vert g_{i+1} \vert \gamma_{i+1} \vert g_i \vert \cdots \vert g_1])^{(-1)^{n-i}}.
\]
Here $g_i, \gamma \in \G$ and $\gamma_i := g_i \cdots g_1 \gamma g_1^{-1} \cdots g_i^{-1}$ and we interpret $\mathcal{L} B \G$ as $\G \git \G$.

Fix a Real structure $\Gh$ on $\G$. There are a number of natural variants of $\mathcal{L} B\G$ \cite[\S 1.4]{noohiyoung2022}. The variant relevant to this paper is defined as follows. There are two natural actions of $\Z_2$ on $\mathcal{L}B \G$. The first is by negation on $B \Z$, and so on $\vert B \Z \vert \sim S^1$ by an orientation reversing homotopy involution. The second is induced by the $\Z_2$-action on $B\G \rightarrow B \Gh$ by deck transformations. The diagonal quotient groupoid
\[
\mathcal{L}_{\pi}^{\refl} B \Gh
=
\mathcal{L} B \G \git \Z_2,
\]
is called the \emph{unoriented loop groupoid} of $B \Gh$. There is an equivalence $\mathcal{L}_{\pi}^{\refl} B \Gh \simeq \G \git_R \Gh$ under which the double cover $\mathcal{L} B \G \rightarrow \mathcal{L}_{\pi}^{\refl} B \Gh$ is identified with $\G \git \G \rightarrow \G \git_R \Gh$ \cite[Lemma 1.4]{noohiyoung2022}.

Consider the correspondence obtained from \eqref{diag:loopDiagGrpd} by quotienting by $\Z_2$:
\begin{equation}\label{diag:oriLoopDiagGrpd}
B \Gh \xleftarrow[]{\ev} B \Z \times_{\Z_2} \mathcal{L} B \G \xrightarrow[]{\pr_{\mathcal{L} B\G}} \mathcal{L}_{\pi}^{\refl} B \Gh.
\end{equation}
Reflection twisted transgression is the cochain map
\[
\tilde{\uptau}^{\refl}_{\pi} : C^{\bullet}(B \Gh) \xrightarrow[]{(\pr_{\mathcal{L} B\G})_* \circ \ev^*} C^{\bullet -1 + \pi}(\mathcal{L}_{\pi}^{\refl} B \Gh).
\]
An explicit model for $\tilde{\uptau}^{\refl}_{\pi}$ was given in \cite[Theorem 2.7]{noohiyoung2022}. We require only the case $\hat{\alpha} \in C^3(B \Gh)$, which reads
\begin{multline*}
\label{eq:twistTrans3Cocyc}
\tilde{\uptau}_{\pi}^{\refl}(\hat{\alpha})([\varsigma_2 \vert \varsigma_1 ]g)=
\hat{\alpha}([g \vert g^{-1} \vert g])^{\frac{\pi(\varsigma_1)-1}{2} \frac{\pi(\varsigma_2)-1}{2}} \cdot \\
\left(
\frac{\hat{\alpha}([\varsigma_1 g^{-\pi(\varsigma_1)} \varsigma_1^{-1} \vert \varsigma_1 g^{\pi(\varsigma_1)} \varsigma_1^{-1} \vert \varsigma_1]) \hat{\alpha}([\varsigma_1 \vert g^{-\pi(\varsigma_1)} \vert g^{\pi(\varsigma_1)}])}{\hat{\alpha}([\varsigma_1 g^{-\pi(\varsigma_1)} \varsigma_1^{-1} \vert \varsigma_1 \vert g^{\pi(\varsigma_1)}])}
\right)^{-\frac{\pi(\varsigma_2) -1}{2}} \cdot \\
\frac{\hat{\alpha}([\varsigma_2 \vert \varsigma_1 \vert g^{\pi(\varsigma_2 \varsigma_1)}]) \hat{\alpha}([\varsigma_2 \varsigma_1 g^{\pi(\varsigma_2 \varsigma_1)} (\varsigma_2 \varsigma_1)^{-1} \vert \varsigma_2 \vert \varsigma_1])}{\hat{\alpha}([\varsigma_2 \vert \varsigma_1 g^{\pi(\varsigma_2 \varsigma_1)} \varsigma_1^{-1} \vert \varsigma_1])}
\end{multline*}
for $g \in \G$ and $\varsigma_1, \varsigma_2 \in \Gh$.
The diagram
\begin{equation}
\label{eq:transCompat}
\begin{tikzpicture}[baseline= (a).base]
\node[scale=1.0] (a) at (0,0){
\begin{tikzcd}
C^{\bullet}(B \Gh) \arrow{r}[above]{\tilde{\uptau}^{\refl}_{\pi}} \arrow{d} & C^{\bullet -1 + \pi}(\mathcal{L}_{\pi}^{\refl} B \Gh) \arrow{d} \\
C^{\bullet}(B \G) \arrow{r}[below]{\uptau} & C^{\bullet-1}(\mathcal{L} B \G)
\end{tikzcd}
};
\end{tikzpicture}
\end{equation}
whose vertical maps are restrictions along double covers commutes.

A second twisted transgression map is obtained from the correspondence \eqref{diag:oriLoopDiagGrpd} by taking instead twisted cochains on $B \Gh$ as the codomain and results in a cochain map
\[
\uptau^{\refl}_{\pi} : C^{\bullet + \pi}(B \Gh) \xrightarrow[]{(\pr_{\mathcal{L} B\G})_* \circ \ev^*} C^{\bullet -1}(\mathcal{L}_{\pi}^{\refl} B \Gh).
\]
An explicit model for $\uptau_{\pi}^{\refl}$ is given in \cite[Theorem 2.6]{noohiyoung2022}, which for $\hat{\theta} \in C^{2+\pi}(B \Gh)$ reads
\begin{equation*}
\label{eq:twistTrans2Cocyc}
\uptau^{\refl}_{\pi}(\hat{\theta})([\varsigma]g) = \hat{\theta}([g^{-1} \vert g])^{\frac{\pi(\varsigma)-1}{2}} \frac{\hat{\theta}([\varsigma g^{\pi(\varsigma)} \varsigma^{-1} \vert \varsigma])}{\hat{\theta}([\varsigma \vert g^{\pi(\varsigma)}])}.
\end{equation*}
The natural modification of diagram \eqref{eq:transCompat}, with $\tilde{\uptau}_{\pi}^{\refl}$ replaced with $\uptau_{\pi}^{\refl}$, commutes.

\section{Enhanced loop groupoids} \label{sec:Realoopgrpd}

We construct loop groupoids which are enhanced by additional structures, including equivariance data for rotation and reflection of loops and gerbes.

\subsection{Enhanced centralizers}
\label{sec:enCent}

We recall background on twisted loop groups \cite[\S 2]{freed2013c}.

Let $\G$ be a compact Lie group and $q: P \rightarrow S^1 = \mathbb{R} \slash \Z$ a principal $\G$-bundle on which $\G$ acts from the right. The $P$-twisted loop group $L_P \G$ is the group of bundle automorphisms of $P$. Its rotation-extended version
\[
L^{\ext}_P \G
=
\{
(t, \phi) \mid t \in \mathbb{T}, \; \phi: P \rightarrow t^* P \mbox{ is a bundle morphism}
\},
\]
where $\mathbb{T}$ acts on $S^1$ by rotation, fits in the exact sequence
\[
1 \rightarrow L_P \G \rightarrow L^{\ext}_P \G \xrightarrow[]{(t,\phi)\mapsto t} \mathbb{T} \rightarrow 1.
\]
To be concrete, let $g \in \G$ and define a $\G$-bundle $P_g= \mathbb{R} \times_{\Z} \G \rightarrow S^1$, where $\Z$ acts on $\mathbb{R} \times \G$ by $n \cdot (t,h) = (t+n,g^n h)$. Any $\G$-bundle on $S^1$ is isomorphic to $P_g$ for some $g \in \G$. The group $L_{P_g} \G$ is isomorphic to
\[
L_g \G = \{\gamma \in C^{\infty}(\mathbb{R}, \G) \mid \gamma(s+1) = g \gamma(s) g^{-1} \mbox{ for all } s\in \mathbb{R}\}.
\]
The group $\mathbb{R}$ acts on $L_g \G$ by translations, leading to an isomorphism
\[
L^{\ext}_{P_g} \G
\simeq
L_g^{\ext} \G := (L_g \G \rtimes \mathbb{R}) \slash \langle (g,1) \rangle.
\]
On the right hand side, $g \in L_g \G$ is viewed as the constant map with value $g$.

\subsection{Enhanced Real centralizers}
\label{sec:enRealCent}

We generalize the constructions of Section \ref{sec:enCent} to incorporate reflections of $S^1$.

Fix $\mathfrak{c} \in \Aut(\G)$ and let $r: S^1 \rightarrow S^1$ be the involution induced by negation of $\mathbb{R}$. This defines an action of $\mathbb{T} \rtimes_{\pi} \Z$ by coherent autoequivalences on the groupoid $\Bun_{\G}(S^1)$ of $\G$-bundles on $S^1$. At the level of objects, $(t,n) \in \mathbb{T} \rtimes_{\pi} \Z$ sends a $\G$-bundle $P$ to $(r^n)^* t^* \mathfrak{c}^n(P)$. This descends to an action of $\mathbb{T} \rtimes_{\pi} \Z_{2n}$ when $\mathfrak{c}$ has order $2n$ in $\Out(\G)$.

Let $\Gh$ be a Real structure on $\G$. Fix $\omega \in \Gh \setminus \G$ and set $\mathfrak{c} = \Ad_{\omega}$, which has order two in $\Out(\G)$. There is a natural isomorphism $\id_{\Bun_{\G}(S^1)} \Rightarrow \Ad_{\omega^2}$ whose component at $P \in \Bun_{\G}(S^1)$ is $R_{\omega^{-2}}:
P \xrightarrow[]{\sim} \Ad_{\omega^2}(P)$, $p \mapsto p \omega^{-2}$.

Given a $\G$-bundle $P \rightarrow S^1$, consider the group
\[
\tilde{L}_P^{R \mhyphen \ext} \Gh
=
\{
((t,n), \phi) \mid (t,n) \in \mathbb{T} \rtimes_{\pi} \Z, \; \phi: P \rightarrow (t,n) \cdot P \mbox{ is a bundle morphism}
\}
\]
with composition law
\[
((t_2,n_2), \phi_2) \circ ((t_1,n_1), \phi_1)
=
((t_2 + \epsilon(n_2) t_1 , n_2 + n_1),(t_1,n_1)^*\phi_2\circ \phi_1).
\]
The assignment $((t,n), \phi) \mapsto n$ defines a $\Z$-grading of $\tilde{L}_P^{R \mhyphen \ext} \Gh$ which descends to a $\Z_2$-grading of
\[
L_P^{R \mhyphen \ext} \Gh
=
\tilde{L}_P^{R \mhyphen \ext} \Gh \slash \langle ((0,2),R_{\omega^{-2}}) \rangle
\]
with ungraded group $L^{\ext}_P\G$. The assignment $((t,n), \phi) \mapsto (t,n)$ defines a $\Z_2$-graded group homomorphism $L_P^{R \mhyphen \ext} \Gh \rightarrow \mathsf{O}_2$ which fits into the exact commutative diagram
\begin{equation}
\label{diag:enhancedTwistStabCpt}
\begin{tikzpicture}[baseline= (a).base]
\node[scale=1.0] (a) at (0,0){
\begin{tikzcd}
{} & 1 \arrow{d} & 1 \arrow{d} & {} & {} \\
{} & L_P \G \arrow{d} \arrow[r,equal] & L_P \G \arrow{d} & {} & {} \\
1 \arrow{r} & L^{\ext}_P \G \arrow{r} \arrow{d} & L^{R \mhyphen \ext}_P \Gh \arrow{r} \arrow{d} & \Z_2 \arrow{r} \arrow[d,equal] & 1\\
1 \arrow{r} & \mathbb{T} \arrow{r} \arrow{d} & \mathsf{O}_2 \arrow{r} \arrow{d} & \Z_2 \arrow{r} & 1 \\
{} & 1 & 1 & {} & {}
\end{tikzcd}
};
\end{tikzpicture}
.
\end{equation}
The subgroup $L_P^R \Gh \leq L_P^{R \mhyphen \ext} \Gh$ of elements with $t=0$ is $\Z_2$-graded with ungraded group $L_P\G$.

Turning to concrete models, fix (a representative) $g \in \pi_0(\G \git \G)_{-1}$ and let $\omega \in C_{\Gh}^R(g) \setminus C_{\G}(g)$. Then $\mathbb{R} \rtimes_{\pi} \Z$ acts on $L_g \G$ by
\begin{equation*}
\label{tnact}
((t,n) \cdot \gamma)(s)
=
\omega^{-n} \gamma(\epsilon(n) (s+t)) \omega^{n}.
\end{equation*}
Define a group homomorphism $\Phi: L_g \G \rtimes_{\pi} (\mathbb{R} \rtimes_{\pi} \Z) \rightarrow \tilde{L}_{P_g}^{R \mhyphen \ext} \Gh$ by
\[
\Phi(\gamma, (t,n))(s)
=
\left((\epsilon(n)(s+t),n),\gamma(\epsilon(n) s) \omega^n x \omega^{-n}) \right).
\]
Since $\Phi(\omega^{-2},(0,2))= ((0,2),R_{\omega^{-2}})$, we find that $\Phi$ descends to an isomorphism
\[
\left( L_g \G \rtimes (\mathbb{R} \rtimes_{\pi} \Z) \right) \slash \langle (g,(1,0)),  (\omega^{-2},(0,2)) \rangle \xrightarrow[]{\sim} L_{P_g}^{R \mhyphen \ext} \Gh
\]
which restricts to an isomorphism
\[
L_g^R \Gh :=
\left( L_g \G \rtimes_{\pi} \Z \right)
\slash \langle (\omega^{-2},2) \rangle \xrightarrow[]{\sim}
L_{P_g}^R \Gh.
\]

In this paper, we are concerned with subgroups of constant bundle morphisms, which we adorn with a subscript $c$. There are isomorphisms
\begin{equation}
\label{constLoopGroupIden}
L_{c,P_g} \G
\simeq
C_{\G}(g)
\end{equation}
and
\[
L^{\ext}_{c,P_g} \G
\simeq
\Lambda_{\G}(g) := (\mathbb{R} \times C_{\G}(g) ) \slash \langle (-1,g) \rangle.
\]
See \cite[\S 2.3]{freed2013c}. The isomorphism \eqref{constLoopGroupIden} lifts to a $\Z_2$-graded group isomorphism
\[
L_{c,P_g}^R \Gh
\simeq
C^R_{\Gh}(g).
\]
The element $(-1,g) \in \mathbb{R} \rtimes_{\pi} C_{\Gh}^R(g)$ is Real central and so generates a normal subgroup isomorphic to $\Z$. This leads to an isomorphism of $\Z_2$-graded groups
\[
L_{c,P_g}^{R \mhyphen \ext} \Gh
\simeq
\Lambda^R_{\Gh}(g)
:=
\left( \mathbb{R} \rtimes_{\pi} C_{\Gh}^R(g) \right) \slash \langle (-1,g) \rangle.
\]
We refer to $\Lambda_{\G}(g)$ and $\Lambda^R_{\Gh}(g)$ as the \emph{enhanced centralizer} and \emph{Real enhanced centralizer} of $g$, respectively.

\begin{Ex}
\label{ex:enhandCentIdent}
When $g=e$, the left column of diagram \eqref{diag:enhancedTwistStabCpt} splits, $L_e^{\ext} \G \simeq L \G \rtimes \mathbb{T}$. When $\Gh$ is split, the choice of $\omega \in \Gh \setminus \G$ which satisfies $\omega^2=e$ splits the middle column, $L_e^{R \mhyphen \ext} \Gh \simeq L \G \rtimes \mathsf{O}_2$. Here the identity component of $\mathsf{O}_2$ acts on $L\G$ by loop rotation and the non-identity component by loop reflection and conjugation of $\G$ by $\omega$. Without any splitting assumptions, we have
\[
L_e^{R \mhyphen \ext} \Gh
\simeq
\left( L \G \rtimes (\mathbb{T} \rtimes_{\pi} \Z) \right) \slash \langle (\omega^{-2},(0,2)) \rangle.
\]

At the level of constant loops, there are isomorphisms $\Lambda_{\G}(e) \simeq \mathbb{T} \times \G$ and, without any splitting assumptions, we have $\Lambda_{\Gh}^R(e) \simeq \mathbb{T} \rtimes_{\pi} \Gh$.
\end{Ex}

If $g \in \G$ has finite order, then there is a commutative diagram of exact sequences
\begin{equation}
\label{diag:finiteOrderRealConstLoop}
\begin{tikzpicture}[baseline= (a).base]
\node[scale=1.0] (a) at (0,0){
\begin{tikzcd}
1  \arrow{r} & \Z \arrow{r}[above]{1 \mapsto (-1,g)} \arrow[twoheadrightarrow]{d} & [40pt] \mathbb{R} \rtimes_{\pi} C_{\Gh}^R(g) \arrow{r} \arrow{d}[right]{(t,\omega) \mapsto (\left[\frac{t}{\vert g \vert} \right],\omega)} & [40pt] \Lambda_{\Gh}^R(g) \arrow{r} \arrow[equal]{d} & 1 \\
1  \arrow{r} & \Z_{\vert g \vert} \arrow{r}[below]{[1] \mapsto (-\left[\frac{1}{\vert g \vert} \right],g)} & \mathbb{T} \rtimes_{\pi} C_{\Gh}^R(g) \arrow{r}[below]{p} & \Lambda_{\Gh}^R(g) \arrow{r} & 1,
\end{tikzcd}
};
\end{tikzpicture}
\end{equation}
where $p([t],\omega) = [(\vert g \vert t, \omega)]$ and we view $\mathbb{T}$ as $\mathbb{R} \slash \Z$.

\subsection{Twisted enhanced (Real) centralizer subgroups}
\label{sec:twEnCent}

Keep the notation of Sections \ref{sec:enCent} and \ref{sec:enRealCent} and assume in addition that $\G$ and $\Gh$ are finite. Fixing representatives of the conjugacy classes of $\G$ induces an equivalence
\begin{equation}
\label{eq:loopGrpDecomp}
\mathcal{L} B \G
\simeq
\bigsqcup_{g \in \pi_0(\G \git \G)} B C_{\G}(g).
\end{equation}
Let $\alpha \in Z^3(B \G)$. Denote by $\uptau(\alpha)_g \in Z^2(B C_{\G}(g))$ the restriction of $\uptau(\alpha)$ to the $g$\textsuperscript{th} component of \eqref{eq:loopGrpDecomp} and\footnote{In the notation of Section \ref{sec:RealRepThy}, ${^{\alpha}}C_{\G}(g)$ would be denoted ${^{\uptau(\alpha)_g}}C_{\G}(g)$, but this is too cumbersome.} ${^{\alpha}}C_{\G}(g)$ the associated central extension of $C_{\G}(g)$.

\begin{Def}[{\cite[\S 2.3]{freed2013c}}]
The \emph{$\alpha$-twisted enhanced centralizer} of $g \in \G$ is
\[
{^{\alpha}}\Lambda_{\G}(g)
=
(\mathbb{R} \times {^{\alpha}} C_{\G}(g)) \slash \langle (-1,(g,1_{\mathbb{T}})) \rangle.
\]
\end{Def}

There is a commutative diagram of central extensions
\begin{equation*}
\label{diag:enhancedCentExt}
\begin{tikzpicture}[baseline= (a).base]
\node[scale=1.0] (a) at (0,0){
\begin{tikzcd}[column sep=4em]
{} & {} & 1 \arrow{d} & 1 \arrow{d} & {} \\
{} & {} & \Z \arrow[r,equal]  \arrow{d}[left]{1 \mapsto (-1,(g,z))} & \Z \arrow{d} & {} \\
1 \arrow{r} & \mathbb{T} \arrow{r}  \arrow[d,equal] & \mathbb{R} \times {^{\alpha}}C_{\G}(g) \arrow{r} \arrow{d} &  \mathbb{R} \times C_{\G}(g) \arrow{r} \arrow{d} & 1 \\
1 \arrow{r} & \mathbb{T} \arrow{r} & {^{\alpha}}\Lambda_{\G}(g) \arrow{r}[below]{(t,(h,z)) \mapsto (t,h)} \arrow{d} & \Lambda_{\G}(g) \arrow{r} \arrow{d} & 1 \\
{} & {} & 1 & 1 & {}
\end{tikzcd}
};
\end{tikzpicture}.
\end{equation*}
The projection $\mathbb{R} \times C_{\G}(g) \rightarrow \mathbb{R}$ induces a homomorphism $\Lambda_{\G}(g) \rightarrow \mathbb{T}$ which fits into the exact commutative diagram
\begin{equation}
\label{diag:enhancedTwistStab}
\begin{tikzpicture}[baseline= (a).base]
\node[scale=1.0] (a) at (0,0){
\begin{tikzcd}
{} & {} & 1 \arrow{d} & 1 \arrow{d} & {} \\
1 \arrow{r} & \mathbb{T} \arrow{r} \arrow[d,equal] & {^{\alpha}}C_{\G}(g) \arrow{r} \arrow{d} & C_{\G}(g) \arrow{r} \arrow{d} & 1 \\
1 \arrow{r} & \mathbb{T} \arrow{r} & {^{\alpha}}\Lambda_{\G}(g) \arrow{r} \arrow{d} & \Lambda_{\G}(g) \arrow{r} \arrow{d} & 1 \\
{} & {} & \mathbb{T} \arrow[r,equal] \arrow{d} & \mathbb{T} \arrow{d} & {} \\
{} & {} & 1 & 1 & {}
\end{tikzcd}
};
\end{tikzpicture}.
\end{equation}

Let now $\Gh$ be a $\Z_2$-graded finite group. Fixing representatives of the Real conjugacy classes of $\G$ which agree with the chosen representatives of the conjugacy classes of $\G$ under the bijection \eqref{eq:RealConjClasses} induces an equivalence
\begin{equation}
\label{eq:loopGrpDecompUnori}
\mathcal{L}_{\pi}^{\refl} B \Gh
\simeq
\bigsqcup_{g \in \pi_0(\G \git_R \Gh)} B C_{\Gh}^R(g).
\end{equation}
The choice of an element $\omega \in \Gh \setminus \G$ refines the equivalence \eqref{eq:loopGrpDecomp} to
\[
\mathcal{L} B \G
\simeq
\bigsqcup_{g \in \pi_0(\G \git \G)_{-1}} B C_{\G}(g) \sqcup \bigsqcup_{\{g,\omega g^{-1} \omega^{-1}\} \subset \pi_0(\G \git \G)_{+1}} B C_{\G}(g) \sqcup BC_{\G}(\omega g^{-1} \omega^{-1}),
\]
thereby making explicit the double cover $\mathcal{L} B \G \rightarrow \mathcal{L}_{\pi}^{\refl} B \Gh$.

Let $\hat{\alpha} \in Z^3(B \Gh)$ with restriction $\alpha \in Z^3(B \G)$. Denote by $\tilde{\uptau}_{\pi}^{\refl}(\hat{\alpha})_g \in Z^{2 + \pi}(B C_{\Gh}^R(g))$ the restriction of $\tilde{\uptau}_{\pi}^{\refl}(\hat{\alpha})$ to the $g$\textsuperscript{th} component of \eqref{eq:loopGrpDecompUnori}. Finally, let ${^{\hat{\alpha}}}C_{\Gh}^R(g)$ be the associated (Real, if $g \in \pi_0(\G \git \G)_{-1}$,) central extension of $C_{\Gh}^R(g)$.

The next result uses that the isomorphic class of the central extension ${^{\alpha}}\Lambda_{\G}(g)$ is unchanged if the element $(-1,(g,1_{\mathbb{T}}))$ used in its definition is replaced by $(-1,(g,z))$.

\begin{Lem}
\label{lem:LambdaIsoNonFixConj}
Let $g \in \pi_0(\G \git \G)_{+1}$ and $\omega \in \Gh \setminus \G$. There is commutative diagram of group homomorphisms
\[
\begin{tikzpicture}[baseline= (a).base]
\node[scale=1.0] (a) at (0,0){
\begin{tikzcd}
{^{\alpha}}\Lambda_{\G}(g) \arrow{d} \arrow{r}[above]{i_g} & {^{\alpha}}\Lambda_{\G}(\omega g^{-1} \omega^{-1}) \arrow{d} \\
\mathbb{T} \arrow{r}[below]{(-)^{-1}} & \mathbb{T}.
\end{tikzcd}
};
\end{tikzpicture}
\]
Moreover, $i_g$ is an isomorphism of central extensions which restricts to inversion on the central tori.
\end{Lem}

\begin{proof}
Define $f \in C^1(B C_{\G}(g))$ by
\[
f([h])
=
\frac{\tilde{\uptau}^{\refl}_{\pi}(\hat{\alpha})([\omega \vert h]g)}{\tilde{\uptau}^{\refl}_{\pi}(\hat{\alpha})([\omega h \omega^{-1} \vert \omega]g)},
\qquad
h \in C_{\G}(g).
\]
Since $\hat{\alpha}$ is closed, $\tilde{\uptau}_{\pi}^{\refl}(\hat{\alpha})$ is a twisted $2$-cocycle, from which it follows that
\[
\uptau(\alpha)([\omega h_2 \omega^{-1} \vert \omega h_1 \omega^{-1}] g) \uptau(\alpha)([h_2 \vert h_1] g)
=
\frac{f([h_2h_1])}{f([h_2])f([h_1])}.
\]
This equality implies that the map
\begin{eqnarray*}
i_g: \mathbb{R} \times {^{\alpha}}C_{\G}(g) & \rightarrow & \mathbb{R} \times {^{\alpha}}C_{\G}(\omega g^{-1} \omega^{-1}) \\
(t,(h,z)) &\mapsto& (-t,(\omega h \omega^{-1},f([h]) z^{-1}))
\end{eqnarray*}
is a group isomorphism.
Since
\[
i_g (-1,(g,1_{\mathbb{T}}))^{-1}
=
(-1,(\omega g^{-1} \omega^{-1},\uptau(\alpha)([\omega g \omega^{-1} \vert \omega g^{-1} \omega^{-1}]g) f([g])^{-1})),
\]
the map $i_g$ descends to an isomorphism of central extensions from ${^{\alpha}}\Lambda_{\G}(g)$ to
\[
\mathbb{R} \times {^{\alpha}}C_{\G}(\omega g^{-1} \omega^{-1})\slash \langle (-1,(\omega g^{-1} \omega^{-1},\uptau(\alpha)([\omega g \omega^{-1} \vert \omega g^{-1} \omega^{-1}]g) f([g])^{-1})) \rangle,
\]
the codomain of which is canonically isomorphic to ${^{\alpha}}\Lambda_{\G}(\omega g^{-1} \omega^{-1})$.
\end{proof}


\begin{Lem}
\label{lem:centralElement}
The element $(g,1_{\mathbb{T}}) \in {^{\hat{\alpha}}}C_{\Gh}^R(g)$ is Real central.
\end{Lem}

\begin{proof}
Let $(\varsigma,z) \in {^{\hat{\alpha}}}C_{\Gh}^R(g)$. We compute
\begin{equation}
\label{eq:RealConjCalc}
(\varsigma,z)  (g,1_{\mathbb{T}}) (\varsigma,z)^{-1}
=
(g^{\pi(\varsigma)}, \frac{\tilde{\uptau}_{\pi}^{\refl}(\hat{\alpha})([\varsigma \vert g] g) \tilde{\uptau}_{\pi}^{\refl}(\hat{\alpha})([\varsigma g \vert \varsigma^{-1}]g)}{\tilde{\uptau}_{\pi}^{\refl}(\hat{\alpha})([\varsigma \vert \varsigma^{-1}]g)}).
\end{equation}
Using the twisted $2$-cocycle condition on $\tilde{\uptau}_{\pi}^{\refl}(\hat{\alpha})_g$, the second component of \eqref{eq:RealConjCalc} is
\[
\frac{\tilde{\uptau}_{\pi}^{\refl}(\hat{\alpha})([g \vert \varsigma^{-1}]g)^{\pi(\varsigma)}}{\tilde{\uptau}_{\pi}^{\refl}(\hat{\alpha})([\varsigma^{-1} \vert g^{\pi(\varsigma)}]g)^{\pi(\varsigma)}}
=
\begin{cases}
1 & \mbox{ if } \pi(\varsigma) =1, \\
\tilde{\uptau}_{\pi}^{\refl}(\hat{\alpha})([g \vert g^{-1}]g)^{-1} & \mbox{ if } \pi(\varsigma) =-1.
\end{cases}
\]
It follows that $(\varsigma,z) (g,1_{\mathbb{T}}) (\varsigma,z)^{-1}
=(g,1_{\mathbb{T}})^{\pi(\varsigma)}$.
\end{proof}

By Lemma \ref{lem:centralElement}, the element $(-1,(g,1_{\mathbb{T}})) \in \mathbb{R} \rtimes_{\pi} {^{\hat{\alpha}}}C_{\Gh}^R(g)$ is Real central.

\begin{Def}
The \emph{$\hat{\alpha}$-twisted enhanced Real stabilizer} of $g \in \G$ is
\[
{^{\hat{\alpha}}}\Lambda_{\Gh}^R(g)
=
\left(
\mathbb{R} \rtimes_{\pi} {^{\hat{\alpha}}}C_{\Gh}^R(g)
\right) \slash
\langle (-1,(g,1_{\mathbb{T}})) \rangle.
\]
\end{Def}

The group ${^{\hat{\alpha}}}\Lambda_{\Gh}^R(g)$ is $\Z_2$-graded with ungraded group ${^{\alpha}}\Lambda_{\G}(g)$. There is an exact commutative diagram
\[
\begin{tikzpicture}[baseline= (a).base]
\node[scale=1.0] (a) at (0,0){
\begin{tikzcd}
{} & {} & 1 \arrow{d} & 1 \arrow{d} & {} \\
{} & {} & \Z \arrow[r,equal]  \arrow{d} & \Z \arrow{d} & {} \\
1 \arrow{r} & \mathbb{T} \arrow{r}  \arrow[d,equal] & \mathbb{R} \rtimes_{\pi} {^{\hat{\alpha}}}C_{\Gh}^R(g) \arrow{r} \arrow{d} &  \mathbb{R} \rtimes_{\pi} C_{\Gh}^R(g) \arrow{r} \arrow{d} & 1 \\
1 \arrow{r} & \mathbb{T} \arrow{r} & {^{\hat{\alpha}}}\Lambda^R_{\Gh}(g) \arrow{r} \arrow{d} & \Lambda^R_{\Gh}(g) \arrow{r} \arrow{d} & 1 \\
{} & {} & 1 & 1 & {} \\
\end{tikzcd}
};
\end{tikzpicture}
\]
whose rows and columns (by Lemma \ref{lem:centralElement}) are Real central extensions.

If $C_{\Gh}^R(g)$ is non-trivially graded, then the homomorphism ${^{\hat{\alpha}}}C^R_{\Gh}(g) \rightarrow C^R_{\Gh}(g)$ induces a homomorphism ${^{\hat{\alpha}}}\Lambda^R_{\Gh}(g) \rightarrow \Lambda^R_{\Gh}(g)$ which fits into the exact commutative diagram
\begin{equation*}
\label{diag:RealEnhancedTwistStabGrade}
\begin{tikzpicture}[baseline= (a).base]
\node[scale=1.0] (a) at (0,0){
\begin{tikzcd}
{} & {} & 1 \arrow{d} & 1 \arrow{d} & {} \\
1 \arrow{r} & \mathbb{T} \arrow{r} \arrow[d,equal] & {^{\alpha}}\Lambda_{\G}(g) \arrow{r} \arrow{d} & \Lambda_{\G}(g) \arrow{r} \arrow{d} & 1 \\
1 \arrow{r} & \mathbb{T} \arrow{r} & {^{\hat{\alpha}}}\Lambda^R_{\Gh}(g) \arrow{r} \arrow{d} & \Lambda^R_{\Gh}(g) \arrow{r} \arrow{d} & 1 \\
{} & {} & \Z_2 \arrow[r,equal] \arrow{d} & \Z_2 \arrow{d} & {} \\
{} & {} & 1 & 1 & {}
\end{tikzcd}
};
\end{tikzpicture}.
\end{equation*}
In view of the exact sequence \eqref{eq:univCovO2}, the homomorphism
\[
\mathbb{R} \rtimes_{\pi} {^{\hat{\alpha}}}C_{\Gh}^R(g) \rightarrow \mathbb{R} \rtimes_{\pi} \Z_2,
\qquad
(t,(\varsigma,z)) \mapsto (t, \pi(\varsigma))
\]
descends to a homomorphism ${^{\hat{\alpha}}}\Lambda_{\Gh}^R(g) \rightarrow \mathsf{O}_2$ which fits into the exact commutative diagram
\begin{equation}
\label{diag:RealEnhancedTwistStab}
\begin{tikzpicture}[baseline= (a).base]
\node[scale=1.0] (a) at (0,0){
\begin{tikzcd}
{} & 1 \arrow{d} & 1 \arrow{d} & {} & {} \\
{} & {^{\alpha}}C_{\G}(g) \arrow{d} \arrow[r,equal] & {^{\alpha}}C_{\G}(g) \arrow{d} & {} & {} \\
1 \arrow{r} & {^{\alpha}}\Lambda_{\G}(g) \arrow{r} \arrow{d} & {^{\hat{\alpha}}}\Lambda^R_{\Gh}(g) \arrow{r} \arrow{d} & \Z_2 \arrow{r} \arrow[d,equal] & 1\\
1 \arrow{r} & \mathbb{T} \arrow{r} \arrow{d} & \mathsf{O}_2 \arrow{r} \arrow{d} & \Z_2 \arrow{r} & 1 \\
{} & 1 & 1 & {} & {}
\end{tikzcd}
};
\end{tikzpicture}
.
\end{equation}

\subsection{Enhanced loop groupoids}
\label{sec:enhLoopGrpd}

Let a compact Lie group $\G$ act on a manifold $X$. We recall various notions of the loop groupoid of $X \git \G$ \cite{lupercio2002,ganter2013,huan2018}. The loop groupoids of primary interest are extensions of the inertia groupoid
\[
\mathcal{L} \left( X \git \G \right)
\simeq
\bigsqcup_{g \in \pi_0 (\G \git \G)} X^g \git C_{\G}(g).
\]

Let $\Loop_1^{\ext}(X\git \G)$ be the category whose objects are diagrams $P=(S^1 \xleftarrow[]{q} P \xrightarrow[]{f} X)$ with $q$ a $\G$-bundle and $f$ a $\G$-equivariant map and whose morphisms $(t,\phi) : P \rightarrow P^{\prime}$ are commutative diagrams
\[
\begin{tikzpicture}[baseline= (a).base]
\node[scale=1.0] (a) at (0,0){
\begin{tikzcd}[column sep=1em,row sep=1.5em]
& X & \\
P \arrow{rr}[above]{\phi} \arrow{ur}[above left]{f} \arrow{d}[left]{q}& & P^{\prime} \arrow{ul}[above right]{f^{\prime}} \arrow{d}[right]{q^{\prime}} \\
S^1 \arrow{rr}[below]{t} & & S^1.
\end{tikzcd}
};
\end{tikzpicture}
\]
The subgroupoid of morphisms with $t=0$ is the category $\textnormal{Bibun}(S^1, X \git \G)$ of bibundles, as defined in \cite[\S 3]{lerman2010}. In particular, $\Loop_1^{\ext}(X\git \G)$ is an extension of $\textnormal{Bibun}(S^1, X \git \G)$ by the groupoid of rotations $B \mathbb{T}$.

Given $g \in \G$, let
\[
\mathcal{P}_g X
=
\{
f \in C^{\infty}(\mathbb{R}, X) \mid f(s+1)=f(s)g \mbox{ for all } s \in \mathbb{R}
\}.
\]
Let $\Loop^{\ext}_2(X \git \G)$ be the groupoid with objects $\bigsqcup_{g \in \G} \mathcal{P}_g X$ and morphisms $f_1 \rightarrow f_2$ given by
\[
\{ (t,\phi) \in \mathbb{R} \times C^{\infty}(\mathbb{R},\G) \mid \phi(s) g_2 = g_1 \phi(s+1) \mbox{ and }
f_2(s) = f_1(s-t) \phi(s-t) \mbox{ for all } s\in \mathbb{R} \},
\]
where $f_i \in \mathcal{P}_{g_i}(X)$.
The group $L_g^{\ext} \G$ acts on $\mathcal{P}_g X$ (see equation \eqref{eq:twistLoopGroupAct} below) and the groupoid
\begin{equation}
\label{eq:skelLoop2}
\bigsqcup_{g \in \pi_0(\G \git \G)} \mathcal{P}_g X \git L_g^{\ext} \G
\end{equation}
is a skeleton of $\Loop_2^{\ext}(X \git \G)$ \cite[Proposition 2.11]{huan2018}.

Let $\G^{\tor} \subset \G$ be the subset of torsion elements and $\Loop_2^{\ext, \tor}(X \git \G) \subset \Loop_2^{\ext}(X \git \G)$ the full subgroupoid on objects $\bigsqcup_{g \in \G^{\tor}} \mathcal{P}_g X$. Let $L_{\orb}(X \git \G) \subset \Loop^{\ext, \tor}_2(X \git \G)$ be the subgroupoid of morphisms of the form $(t,\phi)$ with $\phi$ constant. We have
\[
\Hom_{L_{\orb}(X \git \G)}(f_1,f_2) = \Lambda_{\G}(g_1,g_2),
\]
where $T_{\G}(g_1,g_2) = \{h \in \G \mid g_1 h=hg_2 \}$ and $\Lambda_{\G}(g_1,g_2)$ is the quotient of $\mathbb{R} \times T_{\G}(g_1,g_2)$ by the equivalence relation generated by $(t,h) \sim (t-1,g_1 h)$.

Finally, let $\Lambda \left( X \git \G \right)$ be the full subgroupoid of $L_{\orb}(X \git \G)$ of constant loops. There is an equivalence
\begin{equation}
\label{eq:conjClassDecompLambda}
\Lambda \left( X \git \G \right)
\simeq
\bigsqcup_{g \in \pi_0 (\G^{\tor} \git \G)} X^g \git \Lambda_{\G}(g).
\end{equation}
In particular, $\Lambda \left( X \git \G \right)$ is a local quotient groupoid. The torsion inertia groupoid
\[
\mathcal{L}^{\tor} (X \git \G)
=
\bigsqcup_{g \in \pi_0(\G^{\tor} \git \G)} X^g \git C_{\G}(g). 
\]
is the subgroupoid of $\Lambda \left( X \git \G \right)$ on morphisms which do not involve loop rotation.

In \cite[Lemma 2.10]{huan2018} it is proved that $\Loop^{\ext}_1(X \git \G)$ is equivalent to a full subgroupoid of $\Loop^{\ext}_2(X \git \G)$. Under this equivalence, the subgroupoid of ghost loops of $X \git \G$ is equivalent to $\Lambda \left( X \git \G \right)$ \cite[Proposition 2.17]{huan2018}.

\subsection{Involutions of enhanced loop groupoids}
\label{sec:invLoopGrpd}

Let a $\Z_2$-graded compact Lie group $\Gh$ act on $X$. Fix $\omega \in \Gh \setminus \G$ and thereby a model $(\iota_{\omega}, \Theta_{\omega})$ for the deck transformation of $X \git \G$. We describe how $(\iota_{\omega}, \Theta_{\omega})$ induces involutions, again denoted by $(\iota_{\omega},\Theta_{\omega})$, of the enhanced loop groupoids of Section \ref{sec:enhLoopGrpd}. Analogues of Lemma \ref{lem:invTorsor} hold so that up to equivalence the involutions constructed below, and their associated quotients, depend only on the Real structure $\Gh$.

The involution $\iota_{\omega}$ of $\Loop_1^{\ext}(X\git \G)$ is given on objects and morphisms by
\[
\iota_{\omega}(S^1\xleftarrow[]{q} P \xrightarrow[]{f} X)
=
(S^1 \xleftarrow[]{r^*q} r^*\Ad_{\omega}(P) \xrightarrow[]{p \mapsto (r^*f)(p) \omega^{-1}} X)
\]
and $\iota_{\omega}(t,\phi) = (-t,r^*\phi)$, respectively. The component of $\Theta_{\omega}$ at $S^1 \xleftarrow[]{q} P \xrightarrow[]{f} X$ is induced by the $\G$-bundle isomorphism $R_{\omega^{-2}}$.

The involution of $\Loop_2^{\ext}(X \git \G)$ is given on objects and morphisms by $\iota_{\omega}(f)(s) = f(-s) \omega$ and $\iota_{\omega}(t,\phi) =( -t,s \mapsto \omega^{-1} \phi(-s) \omega)$, respectively. The component of $\Theta_{\omega}$ at $f$ is $\Theta_{\omega,f} =(0,\omega^2)$. By \cite[Proposition 4.9]{ginot2012}, there is a quotient topological stack
\[
\Loop_2^{\refl}(X \git \Gh) := \Loop_2^{\ext}(X \git \G) \git (\iota_{\omega}, \Theta_{\omega}).
\]
To clarify the geometric meaning of $\Loop_2^{\refl}(X \git \Gh)$, define a right $L_g \G \rtimes(\mathbb{R} \rtimes_{\pi} \Z)$-action on $\mathcal{P}_gX$ by
\begin{equation}
\label{eq:twistLoopGroupAct}
(f \cdot (\gamma,t,n))(s)
=
f(\epsilon(n)(t+s)) \omega^{-n} \gamma(\epsilon(s+t)).
\end{equation}
It is immediate that this descends to a right $L_{g}^{R \mhyphen \ext} \Gh$-action.

\begin{Prop}
\label{prop:Loop2Skel}
The $B\Z_2$-graded groupoid
\[
\bigsqcup_{g \in \pi_0(\G \git \G)_{-1}} \mathcal{P}_g X \git L_g^{\refl} \Gh
\sqcup
\bigsqcup_{g \in \pi_0(\G \git \G)_{+1} \slash \Z_2} \mathcal{P}_g X \git L^{\ext}_g \G
\]
is a skeleton of $\Loop_2^{\refl}(X \git \Gh)$. The corresponding double cover is equivalent to the skeleton \eqref{eq:skelLoop2} of $\Loop_2^{\refl}(X \git \Gh)$ and models $\Loop_2^{\ext}(X \git \G) \rightarrow \Loop_2^{\refl}(X \git \Gh)$.
\end{Prop}

Since Proposition \ref{prop:Loop2Skel} is not used in what follows, we omit its proof. See Proposition \ref{prop:loopGrpdModel} below for a similar result with proof.

The involution $(\iota_{\omega},\Theta_{\omega})$ of $\Loop_2^{\ext}(X \git \G)$ restricts to involutions of $\Loop_2^{\ext, \tor}(X \git \G)$ and $\Lambda \left( X \git \G \right)$. For example, $\iota_{\omega}$ sends a morphism $[t,h]$ in $\Lambda \left( X \git \G \right)$ to $[- t,\omega^{-1} h^{-1} \omega]$.

\begin{Def}
Let $\Lambda_{\pi}^{\refl} (X \git \Gh)$ be the quotient groupoid $\Lambda(X \git \G) \git (\iota_{\omega},\Theta_{\omega})$.
\end{Def}

Since $\Lambda(X \git \G)$ is a topological stack, so too is $\Lambda_{\pi}^{\refl} (X \git \Gh)$ \cite[Proposition 4.10]{ginot2012}. To give an explicit presentation of $\Lambda_{\pi}^{\refl} (X \git \Gh)$, note that the right $\mathbb{R} \rtimes_{\pi} C_{\Gh}^R(g)$-action on $X^g$ given by $x \cdot (t,\varsigma) = x \varsigma$
descends to a right $\Lambda_{\Gh}^R(g)$-action.

\begin{Prop}
\label{prop:loopGrpdModel}
There is an equivalence of $B\Z_2$-graded groupoids
\begin{equation}
\label{eq:conjClassDecompLambdaUnori}
\Lambda^{\refl}_{\pi} (X \git \Gh)
\simeq
\bigsqcup_{g \in \pi_0(\G^{\tor} \git \G)_{-1}} X^g \git \Lambda_{\Gh}^R(g)
\sqcup
\bigsqcup_{g \in \pi_0(\G^{\tor} \git \G)_{+1} \slash \Z_2}
X^g \git \Lambda_{\G}(g).
\end{equation}
\end{Prop}

\begin{proof}
Model $\Lambda( X \git \G)$ by the equivalence \eqref{eq:conjClassDecompLambda}. Fix $\omega \in \Gh \setminus \G$ and model $\Lambda^{\refl}_{\pi} (X \git \Gh)$ by $\Lambda (X \git \G) \git (\iota_{\omega}, \Theta_{\omega})$. Then $\iota_{\omega}$ maps $X^g$ to $X^{\omega^{-1} g^{-1} \omega}$.

Suppose first that $g \in \pi_0(\G^{\tor} \git \G)_{-1}$. Then there exists $\omega^{\prime} \in C_{\Gh}^R(g) \setminus C_{\Gh}(g)$ and the involutions $\iota_{\omega}$ and $\iota_{\omega^{\prime}}$ are equivalent. In fact, there are equivalences
\[
(X^g \git \Lambda_{\G}(g)) \git (\iota_{\omega}, \Theta_{\omega})
\simeq
X^g \git \Lambda_{\Gh}^R(g)
\simeq
(X^g \git \Lambda_{\G}(g)) \git (\iota_{\omega^{\prime}}, \Theta_{\omega^{\prime}}).
\]
Indeed, the choices of $\omega$ and $\omega^{\prime}$ correspond to different presentations of $\Lambda_{\Gh}^R(g)$ as an extension of $\Z_2$ by $\Lambda_{\G}(g)$ via Schreier theory.

If instead $g \in \pi_0(\G^{\tor} \git \G)_{+1}$, then the conjugacy class of $g$ is mapped bijectively to the conjugacy class containing $\omega g^{-1} \omega^{-1}$. In particular, over the component $g \in \pi_0(\G^{\tor} \git \G)_{+1} \slash \Z_2$ the double cover is modelled by the morphism
\[
X^g \git \Lambda_{\G}(g) \sqcup X^{\omega^{-1} g^{-1} \omega} \git \Lambda_{\G}(\omega^{-1} g^{-1} \omega) \rightarrow X^g \git \Lambda_{\G}(g)
\]
which is simply projection to the first factor.
\end{proof}

By Proposition \ref{prop:loopGrpdModel}, the double cover $\Lambda \left( X \git \G \right) \rightarrow \Lambda^{\refl}_{\pi} (X \git \Gh)$ is modelled by
\[
\bigsqcup_{g \in \pi_0(\G^{\tor} \git \G)} X^g \git \Lambda_{\G}(g)
\rightarrow
\bigsqcup_{g \in \pi_0(\G^{\tor} \git_R \Gh)} X^g \git \Lambda_{\Gh}^R(g).
\]
This presentation can be used to realize $\Lambda^{\refl}_{\pi} (X \git \Gh)$ as a quotient of $\mathcal{L}^{\tor} (X \git \G)$ by the $\mathsf{O}_2$-action which rotates and reflects loops, together with the $\Z_2$-action on $X \git \G$ induced by $\Gh$. We focus on the case $X=\pt$ and use the equivalence \eqref{eq:loopGrpDecomp} to describe $\mathcal{L}^{\tor} B \G$; the generalization to arbitrary $X$ will be clear. The summand $B C_{\G}(g)$ of $\mathcal{L}^{\tor} B \G$ associated to $g \in \pi_0(\G^{\tor} \git \G)_{-1}$ has an $\mathsf{O}_2$-action arising from the exact sequence given by the middle column of diagram \eqref{diag:RealEnhancedTwistStab}.
The $\mathsf{O}_2$-action on the summand $B C_{\G}(g) \sqcup BC_{\G}(\omega g^{-1} \omega^{-1})$ associated to a $\Z_2$-orbit $\{g,\omega g^{-1} \omega^{-1}\} \subset \pi_0(\G \git \G)_{+1}$ relies on the untwisted specialization of Lemma \ref{lem:LambdaIsoNonFixConj}. The right column of diagram \eqref{diag:enhancedTwistStab} defines a $\mathbb{T}$-action on each summand. The $\Z_2$-action identifies one summand with the other using the map $i_g$ or its inverse. Since $i_g$ covers $(-)^{-1}: \mathbb{T} \rightarrow \mathbb{T}$, the actions of $\mathbb{T}$ and $\Z_2$ assemble to an action of $\mathsf{O}_2$. In this way, the equivalence of Proposition \ref{prop:loopGrpdModel} is one of $B \mathsf{O}_2$-graded groupoids.

We end this section by incorporating twists into the construction of $\Lambda^{\refl}_{\pi} (X \git \Gh)$. We therefore restrict attention to $\Z_2$-graded finite groups. Fix $\hat{\alpha} \in Z^3(B \Gh)$. The right $\mathbb{R} \rtimes_{\pi} {^{\hat{\alpha}}}C_{\Gh}^R(g)$-action on $\mathcal{P}_gX$ defined by $(f \cdot (t,(\varsigma,z)))(s)
=f(\pi(\varsigma)(s+t)) \varsigma$
descends to a ${^{\hat{\alpha}}}\Lambda^R_{\Gh}(g)$-action which preserves constant loops.

\begin{Def}
Let ${^{\alpha}}\Lambda \left( X \git \G \right) \rightarrow {^{\hat{\alpha}}}\Lambda^{\refl}_{\pi} (X \git \Gh)$ be the double cover
\[
\bigsqcup_{g \in \pi_0(\G \git \G)} X^g \git {^{\alpha}}\Lambda_{\G}(g)
\rightarrow
\bigsqcup_{g \in \pi_0(\G \git_R \Gh)} X^g \git {^{\hat{\alpha}}}\Lambda_{\Gh}^R(g).
\]
\end{Def}

The double cover ${^{\alpha}}\Lambda \left( X \git \G \right) \rightarrow {^{\hat{\alpha}}}\Lambda^{\refl}_{\pi} (X \git \Gh)$ can be described explicitly, generalizing the untwisted discussion following Proposition \ref{prop:loopGrpdModel}. The only modification is that the twisted version of Lemma \ref{lem:LambdaIsoNonFixConj} is used.

Since $\mathbb{T} \leq {^{\alpha}}\Lambda_{\G}(g)$ acts trivially on $X^g$, the morphism ${^{\alpha}}\Lambda ( X \git \G ) \rightarrow \Lambda ( X \git \G)$ is a $\mathbb{T}$-gerbe. Similarly, the morphism ${^{\hat{\alpha}}}\Lambda_{\pi}^{\refl} ( X \git \Gh) \rightarrow \Lambda_{\pi}^{\refl} ( X \git \Gh )$ is a Jandl $\mathbb{T}$-gerbe, in the sense of \cite[\S 2]{schreiber2007}.

\subsection{Enhanced loop groupoids of Lie groupoids}
\label{sec:invLoopGrpdLie}

We generalize the untwisted constructions of Sections \ref{sec:enhLoopGrpd} and \ref{sec:invLoopGrpd} to Lie groupoids.

Given a Lie groupoid $\mathfrak{X}$, let $\Loop_1(\mathfrak{X})$ be the category of $1$-morphisms $S^1 \rightarrow \mathfrak{X}$ in the bicategory $\Bibun$ of bibundles \cite[\S 3.2]{lerman2010}. An object of $\Loop_1(\mathfrak{X})$ is a diagram
\begin{equation}
\label{eq:Loop1Obj}
S^1 \xleftarrow[]{q} P \xrightarrow[]{f} \mathfrak{X}_0
\end{equation}
with $q$ a principal $\mathfrak{X}$-bundle and $f$ an $\mathfrak{X}$-equivariant map. Let $\Loop_1^{\ext}(\mathfrak{X})$ be the groupoid with the same objects as $\Loop_1(\mathfrak{X})$ and morphisms which incorporate loop rotation, as in Section \ref{sec:enhLoopGrpd}; see \cite[Definition 4.2]{huan2018}. Let $\Lambda \mathfrak{X}$ be the full subgroupoid of $\Loop_1^{\ext}(\mathfrak{X})$ on objects \eqref{eq:Loop1Obj} for which there exists a section $s$ of $q$ such that $f \circ s$ is constant. The existence of $s$ implies that $P$ is isomorphic to a bibundle arising from a smooth functor $S^1 \rightarrow \mathfrak{X}$ \cite[Lemma 3.36]{lerman2010}.

Let $Q: \mathfrak{X} \rightarrow \mathfrak{Y}$ be a bibundle. Interpreting $Q$ as a $1$-morphism in $\Bibun$ gives a functor $Q\circ (-) : \Loop_1(\mathfrak{X}) \rightarrow \Loop_1(\mathfrak{Y})$ which lifts to $Q\circ (-):\Loop_1^{\ext}(\mathfrak{X}) \rightarrow \Loop_1^{\ext}(\mathfrak{Y})$. If $Q$ is an equivalence, then so too is $Q\circ (-)$ and $Q\circ (-)$ restricts to an equivalence $\Lambda \mathfrak{X} \rightarrow \Lambda \mathfrak{Y}$.

Let now $\hat{\mathfrak{X}}$ be a $B \Z_2$-graded Lie groupoid with associated double cover $\mathfrak{X} \rightarrow \hat{\mathfrak{X}}$ and non-trivial deck transformation $\omega: \mathfrak{X} \rightarrow \mathfrak{X}$. Given a bibundle \eqref{eq:Loop1Obj}, define the twisted $\mathfrak{X}$-bundle ${^{\omega}}q: {^{\omega}}P \rightarrow S^1$ so that $P={^{\omega}}P$ and ${^{\omega}} q = q$ with ${^{\omega}}f: {^{\omega}}P \rightarrow \mathfrak{X}_0$ equal to $\omega_0 \circ f$ and right $\mathfrak{X}$-action $p \cdot \xi = p \omega(\xi)$. This definition extends to an involution of $\Loop_1^{\ext}(\mathfrak{X})$. The associated quotient map $\Loop_1^{\ext}(\mathfrak{X}) \rightarrow \Loop_1^{\refl}(\hat{\mathfrak{X}})$ is then a double cover. The subgroupoid $\Lambda(\mathfrak{X}) \subset \Loop_1^{\ext}(\mathfrak{X})$ is stable under the involution and so we obtain by restriction a double cover $\Lambda(\mathfrak{X}) \rightarrow \Lambda_{\pi}^{\refl}(\hat{\mathfrak{X}})$.

\begin{Prop}
\label{prop:reflTwLoopFunctoriality}
Let $\hat{\mathfrak{X}}$ and $\hat{\mathfrak{Y}}$ be weakly equivalent $B\Z_2$-graded Lie groupoids. Then $\Lambda_{\pi}^{\refl} \hat{\mathfrak{X}}$ and $\Lambda_{\pi}^{\refl} \hat{\mathfrak{Y}}$ are weakly equivalent $B \Z_2$-graded groupoids.
\end{Prop}

\begin{proof}
It suffices to prove the statement for $\hat{\mathfrak{X}} \simeq \hat{\mathfrak{Y}}$ a $B \Z_2$-graded local equivalence. In this case, there is a $\Z_2$-equivariant local equivalence $\mathfrak{X} \simeq \mathfrak{Y}$. By functoriality of $\Lambda$, there is an induced $\Z_2$-equivariant local equivalence $\Lambda \mathfrak{X} \simeq \Lambda \mathfrak{Y}$. Passing to quotients completes the proof.
\end{proof}

\begin{Lem}
\label{lem:enhancedLoopLQG}
Let $\hat{\mathfrak{X}}$ be a $B\Z_2$-graded local quotient Lie groupoid. Then $\Lambda \mathfrak{X}$ and $\Lambda_{\pi}^{\refl}(\hat{\mathfrak{X}})$ are local quotient groupoids.
\end{Lem}

\begin{proof}
The definitions of $\Lambda \mathfrak{X}$ and $\Lambda_{\pi}^{\refl} \hat{\mathfrak{X}}$ show that it suffices to work locally on $\hat{\mathfrak{X}}$. Since $\hat{\mathfrak{X}}$ is a local quotient groupoid, it suffices to prove the lemma for $\hat{\mathfrak{X}}$ of the form $X \git \Gh$, where $\Gh$ is a $\Z_2$-graded compact Lie group acting on a Hausdorff space $X$, in which case the statement follows from Proposition \ref{prop:loopGrpdModel}.
\end{proof}

\begin{Rem}
The analogue of Lemma \ref{lem:enhancedLoopLQG} for $\Lambda \mathfrak{X}$ is implied by \cite[Lemma 5.4.1]{luecke2019}, where a groupoid locally equivalent to $\Lambda \mathfrak{X}$ is proved to be a local quotient. 
\end{Rem}

\section{Real quasi-elliptic cohomology}
\label{sec:QR}
\subsection{The complex case}

We briefly recall the definition of twisted quasi-elliptic cohomology \cite{huan2018,huan2018b,huan2020}.

Let a finite group $\G$ act on a manifold $X$. Fix $\alpha \in Z^3(B \G)$.

\begin{Def}[{\cite[Definition 5.6]{huan2020}}]
The \emph{$\alpha$-twisted quasi-elliptic cohomology} of $X \git \G$ is
\[
\Q^{\bullet+\alpha}(X \git \G)
=
K^{\bullet + \uptau(\alpha)}(\Lambda \left( X \git \G \right)).
\]
\end{Def}

The morphism $\Lambda (X \git \G) \rightarrow B \mathbb{T}$ which tracks loop rotation gives $\Q^{\bullet}(X \git \G)$ the structure of a $K_{\mathbb{T}}^{\bullet}(\pt)$-algebra while ${^{\alpha}}\Lambda ( X \git \G ) \rightarrow \Lambda ( X \git \G )$ gives $\Q^{\bullet+\alpha}(X \git \G)$ the structure of a $\Q^{\bullet}(X \git \G)$-module.

The equivalence \eqref{eq:conjClassDecompLambda} implies an isomorphism
\begin{equation}
\label{eq:QDecomp}
\Q^{\bullet + \alpha}(X \git \G)
\simeq
\prod_{g \in \pi_0(\G \git \G)}
K^{\bullet + \uptau(\alpha)}_{\Lambda_{\G}(g)}(X^g).
\end{equation}
In particular, by Example \ref{ex:enhandCentIdent}, the summand labelled by $e \in \pi_0(\G \git \G)$ is
\begin{equation}
\label{eq:trivSummSplit}
K^{\bullet + \uptau(\alpha)}_{\Lambda_{\G}(e)}(X^e)
=
K^{\bullet}_{\mathbb{T} \times \G}(X)
\simeq
K^{\bullet}_{\G}(X) \otimes_{\Z} K^{\bullet}_{\mathbb{T}}(\pt)
\simeq
K^{\bullet}_{\G}(X)[q^{\pm 1}].
\end{equation}

If instead $\G$ is assumed to be compact the twist is trivial, then $\Q^{\bullet}(X \git \G)$ is defined to be the $K^{\bullet}_{\mathbb{T}}(\pt)$-subalgebra of the \eqref{eq:QDecomp} obtained by restricting the product to $g$ in torsion conjugacy classes of $\G$. More generally, the quasi-elliptic cohomology of a local quotient Lie groupoid $\mathfrak{X}$ is $\Q^{\bullet}(\mathfrak{X}) = K^{\bullet}(\Lambda \mathfrak{X})$. Basic properties of $\Q^{\bullet}$ are established in \cite[\S 3]{huan2018}. While $\Q^{\bullet}$ is a generalized cohomology theory, it is not elliptic. However, Tate $K$-theory, which is elliptic, can be recovered from $\Q^{\bullet}$. See \cite[\S 4.2]{huan2018}, \cite[Remark 6.19]{dove2019} and Section \ref{sec:RealTateKThy} below.

\subsection{The Real case}

Let a non-trivially $\Z_2$-graded finite group $\Gh$ act on a manifold $X$. Fix $\hat{\alpha} \in Z^3(B \Gh)$.

\begin{Def}
\label{def:QR}
The \emph{$\hat{\alpha}$-twisted Real quasi-elliptic cohomology} of $X \git \G$ is
\[
\QR^{\bullet + \hat{\alpha}}(X \git \G)
=
KR^{\bullet + \tilde{\uptau}^{\refl}_{\pi}(\hat{\alpha})}(\Lambda ( X \git \G )),
\]
where $\Lambda ( X \git \G )$ is considered as the double cover $\Lambda ( X \git \G ) \rightarrow \Lambda_{\pi}^{\refl} ( X \git \Gh )$.
\end{Def}

The $B \Z_2$-graded morphism $\Lambda^{\refl}_{\pi} (X \git \Gh) \rightarrow B \mathsf{O}_2$ which tracks loop rotation and reflection makes $\QR^{\bullet}(X \git \G)$ into a $KR^{\bullet}_{\mathbb{T}}(\pt)$-algebra and, in particular, a module over $\Z[q^{\pm 1}] \subset KR^{\bullet}_{\mathbb{T}}(\pt)$. The $B \Z_2$-graded morphism ${^{\hat{\alpha}}}\Lambda^{\refl}_{\pi} (X \git \Gh) \rightarrow \Lambda^{\refl}_{\pi} (X \git \Gh)$ gives $\QR^{\bullet+\hat{\alpha}}(X \git \G)$ the structure of a $\QR^{\bullet}(X \git \G)$-module.

\begin{Prop}
\label{prop:QRGlobQuot}
There is a $KR^{\bullet}_{\mathbb{T}}(\pt)$-module isomorphism
\begin{equation}
\label{eq:QRDecomp}
\QR^{\bullet + \hat{\alpha}}(X \git \G)
\simeq
\prod_{g \in \pi_0(\G \git_R \Gh)}{^{\pi}}K^{\bullet + \tilde{\uptau}_{\pi}^{\refl}(\hat{\alpha})}_{\Lambda^R_{\Gh}(g)}(X^g).
\end{equation}
\end{Prop}

\begin{proof}
This follows from the equivalence of Proposition \ref{prop:loopGrpdModel} and the comments thereafter which explain that this equivalence is one of $B\mathsf{O}_2$-graded groupoids.
\end{proof}

Using the partition \eqref{eq:RealConjClasses}, the isomorphism \eqref{eq:QRDecomp} can be written explicitly as
\begin{equation*}
\QR^{\bullet + \hat{\alpha}}(X \git \G)
\simeq
\prod_{g \in \pi_0(\G \git \G)_{-1}}
KR^{\bullet+\tilde{\uptau}_{\pi}^{\refl}(\hat{\alpha})}_{\Lambda_{\G}(g)}(X^g)
\times
\prod_{g \in \pi_0(\G \git \G)_{+1} \slash \Z_2}
K^{\bullet+\uptau(\alpha)}_{\Lambda_{\G}(g)}(X^g).
\end{equation*}
By Example \ref{ex:enhandCentIdent}, the summand labelled by $e \in \G$ is
\[
KR^{\bullet + \tilde{\uptau}_{\pi}^{\refl}(\hat{\alpha})}_{\Lambda^R_{\Gh}(e)}(X^e)
\simeq
KR^{\bullet}_{\mathbb{T} \times \G}(X),
\]
where $\mathbb{T} \times \G$ has Real structure $\mathbb{T} \rtimes_{\pi} \Gh$. In general, there is no simple $KR^{\bullet}_{\mathbb{T}}(\pt)$-algebra decomposition of $KR^{\bullet}_{\mathbb{T} \times \G}(X)$ analogous to \eqref{eq:trivSummSplit}. However, if $X$ is compact, then Corollary \ref{cor:KRMackeyDecompReform} implies a $KR^{\bullet}_{\mathbb{T}}(\pt)$-module isomorphism $KR^{\bullet}_{\mathbb{T} \times \G}(X) \simeq KR^{\bullet}_{\G}(X)[q^{\pm 1}]$. See the end of this section for details on a similar calculation.

To extend $\QR^{\bullet}$ to Lie groupoids, let $\hat{\mathfrak{X}}$ be a $B \Z_2$-graded local quotient Lie groupoid. Since $\Lambda_{\pi}^{\refl}$ is local quotient (Lemma \ref{lem:enhancedLoopLQG}), we can deinfe
\[
\QR^{\bullet}(\mathfrak{X})
=
KR^{\bullet}(\Lambda \mathfrak{X}),
\]
where $\Lambda \mathfrak{X} \rightarrow \Lambda_{\pi}^{\refl} \mathfrak{X}$ is the double cover constructed in Section \ref{sec:invLoopGrpdLie}.

\begin{Prop}
\label{prop:QRFunctoriality}
Let $\hat{\mathfrak{X}}$ and $\hat{\mathfrak{Y}}$ be weakly equivalent $B \Z_2$-graded local quotient Lie groupoids. Then $\QR^{\bullet}(\mathfrak{X})$ and $\QR^{\bullet}(\mathfrak{Y})$ are isomorphic $KR^{\bullet}_{\mathbb{T}}(\pt)$-algebras.
\end{Prop}

\begin{proof}
By Proposition \ref{prop:reflTwLoopFunctoriality}, a $B \Z_2$-graded weak equivalence $\hat{\mathfrak{X}} \simeq \hat{\mathfrak{Y}}$ induces a $B \Z_2$-graded weak equivalence $\Lambda_{\pi}^{\refl} \hat{\mathfrak{X}} \simeq \Lambda_{\pi}^{\refl} \hat{\mathfrak{Y}}$. The proposition then follows from the invariance of twisted $K$-theory under weak equivalence \cite[\S 3.1]{gomi2017}.
\end{proof}

We end this section by discussing some $K$-theory groups related to the decomposition \eqref{eq:QRDecomp}. Work in the setting of Definition \ref{def:QR} and set $\hat{\theta}_g = \uptau_{\pi}^{\refl}(\hat{\alpha})_g$. Consider ${^{\pi}}K^{\bullet+\hat{\theta}_g}_{\mathbb{T} \rtimes_{\pi} C^R_{\Gh}(g)}(X^g)$, where $\mathbb{T} \rtimes_{\pi} C^R_{\Gh}(g)$ acts on $X^g$ via the projection to $C^R_{\Gh}(g)$. The homomorphism $p$ of diagram \eqref{diag:finiteOrderRealConstLoop} relates this $K$-theory group to those appearing in \eqref{eq:QRDecomp}. Since the normal subgroup $\mathbb{T}$ acts trivially on $X^g$, we are in the setting of Corollary \ref{cor:KRMackeyDecompReform}. The restriction of $\hat{\theta}_g$ to $\mathbb{T}$ is trivial and all irreducible representations of $\mathbb{T}$ admit a Real structure with respect to the Real structure $\mathsf{O}_2$, through which $\mathbb{T} \rtimes_{\pi} C_{\Gh}^R(g)$ acts. It follows that $\mathbb{T} \rtimes_{\pi} C_{\Gh}^R(g)$ acts trivially on $\Irr(\mathbb{T})$ and
\[
\hat{\mathsf{Q}}(\rho)
=
\left(\mathbb{T} \rtimes_{\pi} C_{\Gh}^R(g)\right) (\rho) \slash \mathbb{T}
\simeq
C_{\Gh}^R(g)
\]
for each $\rho \in \Irr(\mathbb{T})$. To identify $\hat{\nu}_{\rho}$, apply Proposition \ref{prop:nuRestrict} with $\Hh$, $\hat{\mathsf{Q}}$ and $\Gh$ equal to $\mathbb{T}$, $C_{\Gh}^R(g)$ and $\mathbb{T} \rtimes_{\pi} C_{\Gh}^R(g)$, respectively. Since $C_{\Gh}^R(g)$ acts trivially on $\Irr(\mathbb{T})$, the Real central extension $\mathsf{E}$ is the pushout of ${^{\hat{\theta}_g}}C_{\Gh}^R(g)$ along $\rho$, which a direct calculation shows is again ${^{\hat{\theta}_g}}C^R_{\Gh}(g)$. It follows that $\hat{\nu}_{\rho} = \hat{\theta}_g$ for all $\rho \in \Irr^{\theta}(\Hh)$. Summarizing, Corollary \ref{cor:KRMackeyDecompReform} implies the isomorphisms
\begin{equation}
\label{eq:MackeyDecompExample}
{^{\pi}}K^{\bullet+\hat{\theta}_g}_{\mathbb{T} \rtimes_{\pi} C^R_{\Gh}(g)}(X^g)
\simeq
\bigoplus_{\rho \in \Irr(\mathbb{T})}
{^{\pi}}K^{\bullet + \hat{\theta}_g}_{C^R_{\Gh}(g)}(X^g)
\simeq
{^{\pi}}K^{\bullet + \hat{\theta}_g}_{C^R_{\Gh}(g)}(X^g)[q^{\pm 1}].
\end{equation}
When the $\Z_2$-grading of $C_{\Gh}^R(g)$ is trivial, in which case $\mathbb{T} \rtimes_{\pi} C^R_{\Gh}(g) = \mathbb{T} \times C_{\G}(g)$, the isomorphisms \eqref{eq:MackeyDecompExample} can be verified without Corollary \ref{cor:KRMackeyDecompReform}.

\subsection{Examples}
\label{sec:examples}

Keep the notation of Definition \ref{def:QR}.

\begin{Ex}
Let $X = \pt$. The groupoid $\mathcal{L} B \G$ is a $B\Z$-groupoid in the sense of \cite[Definition 2.0.1]{luecke2019}, with $B\Z$-action defined by the automorphism of the identity functor whose component at $g \in \mathcal{L} B \G$ is $g$. Viewing this $B\Z$-action as a rigidified $\mathbb{T}$-action, we have an equivalence
\[\Lambda B\G \simeq 
\bigsqcup\limits_{g \in \pi_0(\G \git \G)} B \Lambda_{\G}(g) \simeq 
\bigsqcup\limits_{g \in \pi_0(\G \git \G)} B C_{\G}(g)\git \mathbb{T},\]
whence $\Q^{\bullet + \alpha}_{\G}(\pt)$ models $K_{\mathbb{T}}^{\bullet + \uptau(\alpha)} (\mathcal{L} B \G)$.

In the $\Z_2$-graded case there is an equivalence $\Lambda_{\pi}^{\refl} B \Gh \simeq \mathcal{L} B\G \git \mathsf{O}_2$ and $\QR^{\bullet + \hat{\alpha}}_{\G}(\pt)$ models $KR_{\mathbb{T}}^{\bullet +\tilde{\uptau}_{\pi}^{\refl}(\hat{\alpha})} (\mathcal{L} B \G)$.
The equivalences \eqref{eq:loopGrpDecompUnori} and \eqref{eq:conjClassDecompLambdaUnori} give
\[
KR^{\bullet + \tilde{\uptau}_{\pi}^{\refl}(\hat{\alpha})} (\mathcal{L} B \G)
\simeq
\prod_{g \in \pi_0(\G \git \G)_{-1}}
KR^{\bullet+\tilde{\uptau}_{\pi}^{\refl}(\hat{\alpha})}_{C_{\G}(g)}(\pt)
\times
\prod_{g \in \pi_0(\G \git \G)_{+1} \slash \Z_2}
K^{\bullet+\uptau(\alpha)}_{C_{\G}(g)}(\pt)
\]
and
\begin{equation*}
\label{eq:QRpoint}
\QR^{\bullet + \hat{\alpha}}_{\G}(\pt)
\simeq
\prod_{g \in \pi_0(\G \git \G)_{-1}}
KR^{\bullet+\tilde{\uptau}_{\pi}^{\refl}(\hat{\alpha})}_{\Lambda_{\G}(g)}(\pt)
\times
\prod_{g \in \pi_0(\G \git \G)_{+1} \slash \Z_2}
K^{\bullet+\uptau(\alpha)}_{\Lambda_{\G}(g)}(\pt).\qedhere
\end{equation*}
\end{Ex}

\begin{Ex}
When $\G = \{e\}$, there is an equivalence $\Lambda \mathfrak{X} \simeq X \times B \mathbb{T}$ which leads to $K^{\bullet}_{\mathbb{T}}(\pt)$-algebra isomorphisms
\[
\Q^{\bullet}(X)
\simeq
K^{\bullet}_{\mathbb{T}}(X)
\simeq
K^{\bullet}(X)[q^{\pm 1}].
\]

The analogous $\Z_2$-graded setting takes $\Gh = \Z_2$, in which case $\Lambda X \rightarrow \Lambda_{\pi}^{\refl} (X \git \Gh)$ is equivalent to $X \times B \mathbb{T} \rightarrow X \times_{\Z_2} B \mathbb{T}$, where $B\mathbb{T}$ is viewed as the double cover $B\mathbb{T} \rightarrow B \mathsf{O}_2$. This leads to a $KR^{\bullet}_{\mathbb{T}}(\pt)$-algebra isomorphism
\[
\QR^{\bullet}(X)
\simeq 
KR^{\bullet}_{\mathbb{T}}(X)
\simeq
KR^{\bullet}(X)[q^{\pm 1}].
\]
When the $\Gh$-action on $X$ is trivial this reduces to $\QR^{\bullet}(X) \simeq KO^{\bullet}(X)[q^{\pm 1}]$.
\end{Ex}

\begin{Ex} \label{QEllRcyclic}
Let $\G = \Z_n$ with multiplicative generator $r$ and $\Gh = D_{2n}$. The $\Z_2$-action on $\pi_0(\Z_n \git \Z_n) =\Z_n$ is trivial and Proposition \ref{prop:QRGlobQuot} gives
\begin{equation*}
\label{eq:RealZnDecompDihed}
\QR^{\bullet}_{\Z_n}(\pt)
\simeq
\prod_{m=0}^{n-1}
KR^{\bullet}_{\Lambda_{\Z_n}(r^m)}(\pt).
\end{equation*}
We compute $KR^{\bullet}_{\Lambda_{\Z_n}(r^m)}(\pt)$ as follows. Since $C^R_{D_{2n}}(r^m) = D_{2n}$, we have
\[
\Lambda_{D_{2n}}^R(r^m) \simeq (\mathbb{R} \rtimes_{\pi} D_{2n}) \slash \langle (-1,r^m) \rangle.
\]
Given $\lambda \in \mathbb{R}$, denote by $V_{\lambda} = \mathbb{C}$ the representation of $\mathbb{R}$ on which $t \in \mathbb{R}$ acts by $e^{2\pi i t \lambda}$. The irreducible representations of $\mathbb{R} \times \Z_n$ are of the form $V_{\lambda} \boxtimes U_k$, where $\lambda \in \mathbb{R}$ and $k \in \{0, \dots, n-1\}$. With respect to the Real structure $\mathbb{R} \rtimes_{\pi} D_{2n}$, each representation $V_{\lambda} \boxtimes U_k$ admits a Real structure and $RR(\mathbb{R} \times \Z_n)=RR(\mathbb{R} \times \Z_n,\mathbb{R})$. The same conclusions hold for the quotient $\Lambda_{\Z_n}(r^m)$, where only the representations $V_{\lambda} \boxtimes U_k$ with $\lambda \equiv \frac{mk}{n} \mod \Z$ are relevant. Writing $x_m$ for the class of $V_{\frac{m}{n}} \boxtimes U_1$, there is an isomorphism
\begin{equation}
\label{eq:QRZnPt}
KR^{\bullet}_{\Lambda_{\Z_n}(r^m)}(\pt)
\simeq
KR^{\bullet}(\pt)[q^{\pm 1}, x_m] \slash \langle x_m^n - q^m \rangle.
\end{equation}
For comparison and later use, recall from \cite[Example 3.3]{huan2018} that
\begin{equation}
\label{eq:QZnPt}
\Q^{\bullet}_{\Z_n}(\pt)
\simeq
\prod_{m=0}^{n-1}
K^{\bullet}(\pt)[q^{\pm 1}, x_m] \slash \langle x_m^n - q^m \rangle. \qedhere
\end{equation}
\end{Ex}

\begin{Ex}
Let $\G = \mathbb{T}$ and $\Gh = \mathsf{O}_2$. The $\Z_2$-action on $\pi_0(\mathbb{T}^{\tor}\git \mathbb{T}) = \mathbb{T}^{\tor}$ is trivial and Proposition \ref{prop:QRGlobQuot} gives
\[
\QR^{\bullet}_{\mathbb{T}}(\pt)
\simeq
\prod_{g \in \mathbb{T}^{\tor}} KR^{\bullet}_{\Lambda_{\mathbb{T}}(g)}(\pt).
\]
To compute $KR^{\bullet}_{\Lambda_{\mathbb{T}}(g)}(\pt)$, identify $\mathbb{T}$ with $\mathbb{R} \slash \Z$, so that $\mathbb{T}^{\tor} \simeq \mathbb{Q} \cap [0,1)$ and choose a lift $r \in \mathbb{Q}$ of $g$. Let $Z_r = \mathbb{C}$ be the representation of $\Lambda_{\mathbb{T}}(g)$ given by
\[
\Lambda_{\mathbb{T}}(g) \simeq (\mathbb{R} \times \mathbb{R}) \slash \langle (-1,r), (0,1) \rangle \xrightarrow[]{[t,x] \mapsto [x+rt]} \mathbb{R} \slash \Z \xrightarrow[]{\exp(2 \pi i (-))} \mathbb{T}.
\]
Any irreducible representation of $\Lambda_{\mathbb{T}}(g)$ is isomorphic to $Z_r$ for some lift $r$ of $g$, or its dual $Z_r^{\vee}$. Note that $Z_r \boxtimes U_k \simeq Z_{r+k}$ for all $k \in \Z$. The representation $Z_r$ admits a Real structure. Using these observations and writing $z_r= [Z_r]$, we conclude that $RR(\mathbb{T})=RR(\mathbb{T},\mathbb{R})$ and
\[
KR^{\bullet}_{\Lambda_{\mathbb{T}}(g)}(\pt)
\simeq
KR^{\bullet}(\pt)[q^{\pm 1},z_r^{\pm 1}].\qedhere
\]
\end{Ex}

\subsection{Basic properties of $\QR^{\bullet}$}\label{QR:prop}

We describe several important constructions for $\QR^{\bullet}$, including change-of-group isomorphisms, induction and restriction, by combining constructions from $KR$-theory and quasi-elliptic cohomology \cite{huan2018}.

\subsubsection{Relation to quasi-elliptic cohomology}

We begin by proving that $\QR^{\bullet}$ reduces to $\Q^{\bullet}$ for trivial double covers. The analogous statement for $KR$-theory is well-known \cite[Proposition 3.3]{atiyah1966}.

\begin{Prop}
\label{prop:QFromQR}
Let $\mathfrak{X}$ be a local quotient Lie groupoid and $\Gh=\Z_2$ act on $\mathfrak{X} \sqcup \mathfrak{X}$ by swapping the summands. There is an isomorphism $\QR^{\bullet}(\mathfrak{X} \sqcup \mathfrak{X}) \simeq
\Q^{\bullet}(\mathfrak{X})$.
\end{Prop}

\begin{proof}
Under the equivalence $\Lambda (\mathfrak{X} \sqcup \mathfrak{X}) \simeq
\Lambda \mathfrak{X} \sqcup \Lambda \mathfrak{X}$, the deck transformation of $\Lambda (\mathfrak{X} \sqcup \mathfrak{X})$ corresponds to swapping the summands while inverting the morphisms of $B \mathbb{T}$. In particular, there is an equivalence $\Lambda_{\pi}^{\refl} (\mathfrak{X} \sqcup \mathfrak{X}\git \Z_2) \simeq \Lambda \mathfrak{X}$. It follows that
\[
\QR^{\bullet}(\mathfrak{X} \sqcup \mathfrak{X})
\simeq
KR^{\bullet}(\Lambda \mathfrak{X} \sqcup \Lambda \mathfrak{X})
\simeq
K^{\bullet}(\Lambda \mathfrak{X})
=
\Q^{\bullet}(\mathfrak{X}). \qedhere
\]
\end{proof}

A similar proof verifies the analogue of Proposition \ref{prop:QFromQR} in the setting of Definition \ref{def:QR}, where the twist of $X \git \G \sqcup X \git \G$ is obtained from a twist of $X \git \G$ by putting its inverse twist on the second summand.

\subsubsection{K\"{u}nneth maps} \label{kunnethfmk}
Let $\pi_{\G} :\Gh \rightarrow \Z_2$ and $\pi_{\Hh} : \Hhh \rightarrow \Z_2$ be $\Z_2$-graded finite groups. The pullback $\Gh \times_{\Z_2} \Hhh$ is $\Z_2$-graded with ungraded group $\G \times \Hh$.
is also a pullback, $C^R_{\Gh \times_{\Z_2} \Hhh}(g, h) \simeq C^R_{\Gh}(g) \times_{\Z_2} C^R_{\Hhh}(h)$,
as is its enhanced variant,
$
\Lambda^R_{\Gh \times_{\Z_2} \Hhh}(g, h)
\simeq
\Lambda^R_{\Gh}(g) \times_{\mathsf{O}_2} \Lambda^R_{\Hhh}(h).
$
Pointwise product defines a cochain map
\[
\bigoplus_{n=0}^{\infty}
C^n(B\Gh) \times C^n(B\Hhh)
\rightarrow
\bigoplus_{n=0}^{\infty} C^n(B(\Gh \times_{\Z_2} \Hhh)),
\qquad
(\hat{\alpha}, \hat{\beta}) \mapsto \hat{\alpha} \hat{\beta}.
\]

Let $X$ be a $\Gh$-manifold and $Y$ an $\Hhh$-manifold. Fix $\hat{\alpha} \in Z^3(B \Gh)$ and $\hat{\beta} \in Z^3(B \Hhh)$ and let $(g,h) \in \G \times \Hh$. The K\"{u}nneth morphism for $K$-theory \cite[\S 3.2]{gomi2017} is
\[
{^{\pi}}K^{\bullet+\tilde{\uptau}_{\pi}^{\refl}(\hat{\alpha})}_{\Lambda^R_{\Gh}(g)}(X^g)\otimes_{\Z} {^{\pi}}K^{\bullet+\tilde{\uptau}_{\pi}^{\refl}(\hat{\beta})}_{\Lambda^R_{\Hhh}(h)}(Y^h)
\rightarrow
{^{\pi}}K^{\bullet+\tilde{\uptau}_{\pi}^{\refl}(\hat{\alpha}\hat{\beta})}_{\Lambda^R_{\Gh \times_{\Z_2} \Hhh}(g,h)}((X \times Y)^{(g,h)}),
\]
where, for ease of notation, we have written $\pi$ for all $\Z_2$-gradings. This map descends to a morphism of $KR^{\bullet}_{\mathbb{T}}(\pt)$-modules
\begin{equation}
\label{eq:basicKunn}
{^{\pi}}K^{\bullet+\tilde{\uptau}_{\pi}^{\refl}(\hat{\alpha})}_{\Lambda^R_{\Gh}(g)}(X^g)\otimes_{KR^{\bullet}_{\mathbb{T}}(\pt)}{^{\pi}}K^{\bullet+\tilde{\uptau}_{\pi}^{\refl}(\hat{\beta})}_{\Lambda^R_{\Hhh}(h)}(Y^h)
\rightarrow
{^{\pi}}K^{\bullet+\tilde{\uptau}_{\pi}^{\refl}(\hat{\alpha}\hat{\beta})}_{\Lambda^R_{\Gh \times_{\Z_2} \Hhh}(g,h)}((X \times Y)^{(g,h)}),
\end{equation}
where we view complex $K$-theory groups as $KR^{\bullet}_{\mathbb{T}}(\pt)$-modules via the forgetful map.

Define
\begin{multline*}
\QR_{\G}^{\bullet + \hat{\alpha}}(X) \hat{\otimes}_{KR^{\bullet}_{\mathbb{T}}(\pt)} \QR_{\Hh}^{\bullet + \hat{\beta}}(Y)
:= \\
\prod_{\substack{g \in \pi_0(\G \git_R \Gh) \\ h \in \pi_0(\Hh \git_R \Hhh)}} {^{\pi}}K^{\bullet+\tilde{\uptau}_{\pi}^{\refl}(\hat{\alpha})}_{\Lambda^R_{\Gh}(g)}(X^g)\otimes_{KR^{\bullet}_{\mathbb{T}}(\pt)}{^{\pi}}K^{\bullet+\tilde{\uptau}_{\pi}^{\refl}(\hat{\beta})}_{\Lambda^R_{\Hhh}(h)}(Y^h).
\end{multline*}
Observe that there is a natural inclusion
\begin{equation}
\label{eq:RealConjIncl}
\pi_0(\G \git_R \Gh) \times \pi_0(\Hh \git_R \Hhh)
\hookrightarrow
\pi_0(\G \times \Hh \git_R \Gh \times_{\Z_2} \Hhh).
\end{equation}
The K\"{u}nneth map for Real quasi-elliptic cohomology
is defined to be the composition
\begin{eqnarray*}
\QR_{\G}^{\bullet + \hat{\alpha}}(X) \hat{\otimes}_{KR^{\bullet}_{\mathbb{T}}(\pt)} \QR_{\Hh}^{\bullet + \hat{\beta}}(Y)
&\xrightarrow[]{\eqref{eq:basicKunn}}&
\prod_{\substack{g \in \pi_0(\G \git_R \Gh) \\ h \in \pi_0(\Hh \git_R \Hhh)}} {^{\pi}}K^{\bullet+\tilde{\uptau}_{\pi}^{\refl}(\hat{\alpha}\hat{\beta})}_{\Lambda^R_{\Gh \times_{\Z_2} \Hhh}(g,h)}((X \times Y)^{(g,h)}) \\
& \xhookrightarrow[]{\eqref{eq:RealConjIncl}} &
\prod_{\pi_0(\G \times \Hh \git_R \Gh \times_{\Z_2} \Hhh)} {^{\pi}}K^{\bullet+\tilde{\uptau}_{\pi}^{\refl}(\hat{\alpha}\hat{\beta})}_{\Lambda^R_{\Gh \times_{\Z_2} \Hhh}(g,h)}((X \times Y)^{(g,h)}) \\
&=&
\QR_{\G \times \Hh}^{\bullet + \hat{\alpha} \hat{\beta}}(X\times Y).
\end{eqnarray*}

The next result is a Real generalization of \cite[Proposition 3.19]{huan2018}.

\begin{Prop}
\label{QEllRcog}
Let $\Gh$ be a $\Z_2$-graded compact Lie group with $\Z_2$-graded closed subgroup $i: \Hhh \hookrightarrow \Gh$. Fix $\hat{\alpha} \in Z^3(B \Gh)$ and let $\Hhh$ act on a space $X$. Then there is a change-of-group isomorphism
\begin{equation}
\label{eq:cgicqellR}
\rho^{\Gh}_{\Hhh}:
\QR^{\bullet + \hat{\alpha}}_{\G}(X\times_{\Hhh} \Gh)\xrightarrow[]{\sim} \QR^{\bullet + i^*\hat{\alpha}}_{\Hh}(X).
\end{equation}
\end{Prop}

\begin{proof}
Define the \eqref{eq:cgicqellR} as the composition
\begin{equation}
\label{cgicqellR}
\rho^{\Gh}_{\Hhh}:
\QR^{\bullet}_{\G}(X\times_{\Hhh} \Gh)\xrightarrow[]{\Res}
\QR_{\Hh}^{\bullet}(X\times_{\Hhh} \Gh)\buildrel{j^*}\over\rightarrow \QR^{\bullet}_{\Hh}(X). 
\end{equation}
where the first map is restriction along $i:\Hhh\hookrightarrow \Gh$ and the second is pullback along the $\Hhh$-equivariant inclusion $j: X\rightarrow X\times_{\Hhh} \Gh$, $x\mapsto [x, e]$. Geometrically, the map \eqref{cgicqellR} sends a twisted Real $\G$-equivariant bundle $V$ over $\Gh\times_{\Hhh}X$ to its restriction to $e\times_{\Hhh} X\simeq X \git \Hhh$. For each $g \in \G$, there is an equivalence
\begin{equation}
\label{eq:RealdecompfixGH}
(X\times_{\Hhh} \Gh)^{g} \simeq \bigsqcup_{h} X^{h} \times_{C_{\Hhh}^R(h)} C^R_{\Gh}(g)  \simeq \bigsqcup_{h}  X^{h} \times_{\Lambda^R_{\Hhh}(h)} \Lambda^R_{\Gh}(g),
\end{equation} 
where $h$ runs over elements of $\pi_0(\Hh\git_R \Hhh)$ which are Real $\Gh$-conjugate to $g$. Note that each $h \in \pi_0(\Hh \git_R \Hhh)_{\pm 1}$ is Real $\Gh$-conjugate to a unique $g_h \in \pi_0(\G \git_R \Gh)_{\pm 1}$. With this, the map \eqref{cgicqellR} can be written as the composition
\begin{multline*}
\QR^{\bullet + \hat{\alpha}}_{\G}(X\times_{\Hhh} \Gh) \rightarrow 
\prod_{g \in \pi_0(\G \git_R \Gh)} \prod_{h} KR^{\bullet + \tilde{\uptau}_{\pi}^{\refl}(\hat{\alpha})_{i(h)}}_{\Lambda_{\G}(h)}( X^{h} \times_{\Lambda^R_{\Hhh}(h)} \Lambda^R_{\Gh}(g))  \\
\xrightarrow[]{\prod_g \prod_h \rho^{\Lambda^R_{\Gh}(h)}_{\Lambda^R_{\Hhh}(h)}}  \prod_{h \in \pi_0(\Hh \git_R \Hhh)} KR^{\bullet + \tilde{\uptau}_{\pi}^{\refl}(i^*\hat{\alpha})_h}_{\Lambda_{\Hh}(h)}(X^{h})
= \QR^{\bullet + i^* \hat{\alpha}}_{\Hh} (X),
\end{multline*}
where the first map is induced by the equivalence \eqref{eq:RealdecompfixGH} and the second is a product of change-of-group isomorphisms for twisted $KR$-theory. This composition is an isomorphism then follows from the observations at the beginning of the proof.
\end{proof}

\subsubsection{Induction maps}

In \cite{strickland1998}, the quotient of the Morava $E$-theory of the symmetric group by a transfer ideal, generated by the image of the induction maps $E^0(B\Sigma_i\times\Sigma_{n-i}) \rightarrow E^0(B\Sigma_n)$, $0<i<n$, with $n=p^k$ for some prime $p$, is shown to classify subgroups of order $p^k$ of the formal group of Morava $E$-theory. In view of the fact that the second Morava $E$-theory is a form of elliptic cohomology, the question arises of whether the finite subgroups of an elliptic curve are classified by the quotient of the corresponding elliptic cohomology of the symmetric group by a transfer ideal. Results suggesting an affirmative answer to this question come from generalized Morava $E$-theory \cite{schlank2014} and quasi-elliptic cohomology and Tate $K$-theory \cite{huan2018b}. Moreover, transfer plays an important role in understanding additive properties of power operations and interacts nicely with Hopkins--Kuhn--Ravenel character theory \cite{hopkins2000}. This section is motivated by the possibility of Real analogues of these results. Our approach to induction and transfer is motivated by the approach for $K$-theory given in \cite{rezk2006}. 
 
Let $\G$ be a compact Lie group with Real structure $\Gh$. Associated to a finite covering of $\Gh$-spaces $f: X\rightarrow Y$ is the pushforward $f_!: KR^{\bullet}_{\G}(X) \rightarrow KR^{\bullet}_{\G}(Y)$. Explicitly, if $V \rightarrow X$ is a Real $\G$-equivariant vector bundle, then $f_!V\rightarrow Y$ is the vector bundle with fibers \[(f_!V)_y=\prod_{x\in f^{-1}(y)} V_x,
\qquad
y \in Y
\]
and Real $\G$-equivariant structure
\begin{equation*}
\label{gactionfincoverKR}
\omega \cdot (v_x)_{x\in f^{-1}(y) } = (v_{x^{\prime} \omega^{-1}} \omega)_{x^{\prime}\in f^{-1}(y\omega)},
\qquad
\omega \in \Gh, \; y\in  Y.
\end{equation*}
In particular, if $\Gh$ is finite with $\Z_2$-graded subgroup $i:\Hhh \hookrightarrow \Gh$ and $\hat{\theta} \in Z^{2+\pi}(B \Gh)$, then the finite $\Gh$-equivariant covering $f: X\times_{\Hhh} \Gh\rightarrow X$ defines the Real induction map
\begin{equation*}
\RInd^{\G}_{\Hh}: KR^{\bullet + i^*\hat{\theta}}_{\Hh}(X) \xrightarrow[]{\sim} KR^{\bullet + \hat{\theta}}_{\G}(X\times_{\Hhh} \Gh)
\xrightarrow[]{f_!} 
KR^{\bullet + \hat{\theta}}_{\G}(X).
\end{equation*}

Using this, define induction for $\QR^{\bullet}$ as the composition
\begin{equation}
\label{eq:indQR}
\mathcal{IR}^{\Gh}_{\Hhh}: \QR^{\bullet + i^* \hat{\alpha}}_{\Hh}(X) \xrightarrow[]{{\rho_{\Hhh}^{\Gh}}^{-1}} \QR^{\bullet+\hat{\alpha}}_{\G}(X\times_{\Hhh} \Gh)
\rightarrow \QR^{\bullet+\hat{\alpha}}_{\G}(X),
\end{equation}
where the second map is induced by the finite covering
\[
\Lambda^{\refl}_{\pi} ((X\times_{\Hhh} \Gh)\git \Gh) \rightarrow \Lambda^{\refl}_{\pi} ( X \git \Gh )
\]
given on objects and morphisms by $(\sigma, [x, g]) \mapsto (\sigma, xg)$ and $([g^{\prime}, t], (\sigma, [x, g])) \mapsto ([g^{\prime}, t], (\sigma, xg))$, respectively. 

The composition \eqref{eq:indQR} can be computed as follows. We focus on the untwisted case. Let
\[
a=\prod_{h\in \pi_0(\Hh\git_R \Hhh)} a_{h}
\in
\QR^{\bullet}_{\Hh}(X),
\]
where we have used the decomposition \eqref{eq:QRDecomp}. Using the explicit formula for the map $\rho_{\Hhh}^{\Gh}$ given in the proof of Proposition \ref{QEllRcog}, the first map of \eqref{eq:indQR} sends $a$ to
\[
\prod_{g\in \pi_0(\G\git_R \Gh) }\prod_{\omega}  a_{\omega g^{\pi(\omega)}\omega^{-1}}\cdot \omega \in  \QR^{\bullet}_{\G}(X\times_{\Hhh} \Gh)
\]
where $\omega$ runs over a set of representatives of $(\Gh\slash\Hhh)^g$. Following \cite[\S 7.2]{huan2018b}, the second map of \eqref{eq:indQR} sends this image to\[
\prod_{g\in \pi_0(\G\git_R \Gh) }\sum_\omega  a_{\omega g^{\pi(\omega)}\omega^{-1}}\cdot \omega.
\]
Then $\mathcal{IR}^{\Gh}_{\Hhh}(a)_{g}$ is equal to
\begin{equation*}
\begin{cases}\RInd^{\Lambda^R_{\Gh}(h)}_{\Lambda^R_{\Hhh}(h)}(a_{h}) &\mbox{if  }g \in \pi_0(\G\git\G)_{-1}\mbox{ is Real conjugate to some  }h \in \pi_0(\Hh\git \Hh )_{-1},\\ 
\Ind^{\Lambda_{\G}(h)}_{\Lambda_{\Hh}(h)}(a_{h}) &\mbox{if  }g \in \pi_0(\G\git\G)_{+1}\mbox{ is Real conjugate to some   }h  \in \pi_0(\Hh\git \Hh)_{+1},\\0 &\mbox{otherwise}.
\end{cases}
\end{equation*}
\begin{Def} 
The \emph{transfer ideal} of $\QR^0_{\G}(\pt)$ is

\[
\mathcal{IR}_{\G}:= \sum_{\Hhh} \im \left( \mathcal{IR}^{\Gh}_{\Hhh}:
\QR^0_{\Hh}(\pt)\rightarrow
\QR^0_{\G}(\pt) \right),
\]
where the sum runs over all $\Z_2$-graded subgroups $\Hhh \leq \Gh$ such that $\Hh \neq \G$.
\end{Def}

The transfer ideal can be computed in terms of Real and complex representation rings as
\[\mathcal{IR}_{\G}\simeq\prod_{g \in \pi_0(\G \git \G)_{-1}}\sum_{\Hhh}   \im(\RInd^{\Lambda_{\G}(g)}_{\Lambda_{\Hh}(g_{\Hh})} )
\times
\prod_{g \in \pi_0(\G \git \G)_{+1} \slash \Z_2}
\sum_{\Hh^{\prime}} \im ( \Ind^{\Lambda_{\G}(g)}_{\Lambda_{\Hh^{\prime}} (g_{\Hh^{\prime}})} ).\]
Here $\Hhh$ (resp. $\Hh^{\prime}$) runs over all proper $\Z_2$-graded subgroups of $\Gh$ (resp. proper subgroups of $\G$) such that $g$ is Real $\Gh$-conjugate to some $g_{\Hhh^{\prime}} \in \pi_0(\Hh \git \Hh)_{-1}$ (resp. $g_{\Hhh^{\prime}} \in \pi_0(\Hh \git \Hh)_{+1}
\slash \Z_2$). It follows that
\begin{multline*}
\QR^{\bullet}_{\G}(\pt)\slash  \mathcal{IR}_{\G} \simeq
\prod_{g \in \pi_0(\G \git \G)_{-1}}
RR(\Lambda_{\G}(g)) \slash \big( \sum_{\Hh} \sum_{g_{\Hhh}}\im(\RInd^{\Lambda_{\G}(g)}_{\Lambda_{\Hh}(g_{\Hhh})} )\big) \times \\
\prod_{g \in \pi_0(\G \git \G)_{+1} \slash \Z_2}
R(\Lambda_{\G}(g)) \slash 
\big(\sum_{\Hh^{\prime}} \sum_{g_{\Hh^{\prime}}} \im ( \Ind^{\Lambda_{\G}(g)}_{\Lambda_{\Hh^{\prime}} (g_{\Hh^{\prime}} )} \big).
\end{multline*}

\subsection{Real quasi-elliptic cohomology and Tate $KR$-theory}
\label{sec:RealTateKThy}

We relate Real quasi-elliptic cohomology to a Real version of Tate $K$-theory, in much the same way that quasi-elliptic cohomology is related to Tate $K$-theory \cite{huan2018,dove2019,huan2020}. See \cite[\S 2]{ando2001}, \cite[\S 2]{ganter2013} for background on Tate $K$-theory.

\begin{Def}\label{def:twistedRealTateK}
Let a $\Z_2$-graded finite group $\Gh$ act on a manifold $X$ and $\hat{\alpha} \in Z^3(B \Gh)$. The \emph{$\hat{\alpha}$-twisted Tate $K$-theory} of $X \git \Gh$ is the subgroup
\[
{^{\pi}}K_{\Tate}^{\bullet + \hat{\alpha}}(X \git \Gh)\subset
\prod_{g \in \pi_0(\G \git_R \Gh)}{^{\pi}}K^{\bullet + \hat{\theta}_g}_{C^R_{\Gh}(g)}(X^g)\Lser{q^{\frac{1}{|g|}}}
\]
whose $g$\textsuperscript{th} component consists of formal $q$-Laurent series $\bigoplus_{k \in \Z} V_k q^{\frac{k}{|g|}}$ which satisfy the following \emph{rotation condition}: for each $k \in \Z$, the coefficient $V_k$ is a $(\pi,\hat{\theta}_g)$-twisted vector bundle on $X^g \git  C_{\Gh}^R(g) $ on which $g$ acts by multiplication by $e^{\frac{2 \pi i k}{\vert g \vert}}$.
\end{Def}

When $\Gh$ is non-trivially $\Z_2$-graded, write $KR_{\Tate}^{\bullet + \hat{\alpha}}(X \git \G)$ for ${^{\pi}}K_{\Tate}^{\bullet + \hat{\alpha}}(X \git \Gh)$.

\begin{Ex}
When the $\Z_2$-grading $\Gh$ is trivial, Definition \ref{def:twistedRealTateK} recovers the twisted equivariant Tate $K$-theory of \cite[Definition 6.23]{dove2019}.
\end{Ex}

\begin{Ex}
When $\Gh = \Z_2$, we have $KR_{\Tate}^{\bullet}(X) \simeq KR^{\bullet}(X) \Lser{q}$. In particular, if $\Gh$ acts trivially on $X$, then $KR_{\Tate}^{\bullet}(X) \simeq KO^{\bullet}(X) \Lser{q}$, a well-known group in topology, appearing, for example, in the context of the Witten genus \cite{ando2000}.
\end{Ex}

Fix $g \in \G$ and set $\hat{\theta}_g= \tilde{\uptau}_{\pi}^{\refl}(\hat{\alpha})_g$. The $\Z_2$-graded group homomorphism
\[
(\mathbb{R} \slash \vert g \vert \Z) \rtimes_{\pi} C_{\Gh}^R(g) \twoheadrightarrow \Lambda_{\Gh}^R(g),
\qquad
([t], \omega) \mapsto [(t, \omega)]
\]
induces a ring homomorphism
\begin{equation}
\label{eq:QRvsKRTate}
{^{\pi}}K^{\bullet + \hat{\theta}_g}_{\Lambda^R_{\Gh}(g)}(X^g) \rightarrow {^{\pi}}K^{\bullet + \hat{\theta}_g}_{(\mathbb{R} \slash \vert g \vert \Z) \rtimes_{\pi} C^R_{\Gh}(g)}(X^g)
\end{equation}
whose image is generated by $(\pi, \hat{\theta}_g)$-twisted vector bundles on $X^g \git ( \mathbb{R} \slash \vert g \vert \Z \rtimes_{\pi} C^R_{\Gh}(g) )$ on which $(-1,g)$ acts trivially. Since the normal subgroup $\mathbb{R} \slash \vert g \vert \Z$ acts trivially on $X^g$, there is an isomorphism
\[
{^{\pi}}K^{\bullet + \hat{\theta}_g}_{(\mathbb{R} \slash \vert g \vert \Z) \rtimes_{\pi} C^R_{\Gh}(g)}(X^g)
\simeq
{^{\pi}}K^{\bullet + \hat{\theta}_g}_{C^R_{\Gh}(g)}(X^g)[q^{\pm\frac{1}{\vert g \vert}}],
\]
as is seen by replacing $\mathbb{T} = \mathbb{R} \slash \Z$ with $\mathbb{R} \slash \vert g \vert \Z$ in the isomorphism \eqref{eq:MackeyDecompExample}. In this way, the map \eqref{eq:QRvsKRTate} becomes a ring homomorphism
\[
{^{\pi}}K^{\bullet + \hat{\theta}_g}_{\Lambda^R_{\Gh}(g)}(X^g) \rightarrow {^{\pi}}K^{\bullet + \hat{\theta}_g}_{C^R_{\Gh}(g)}(X^g)[q^{\pm\frac{1}{\vert g \vert}}].
\]

\begin{Thm}
\label{thm:TateVsQR}
Assume that $\Gh$ is non-trivially $\Z_2$-graded. There is an isomorphism
\[
KR_{\Tate}^{\bullet + \hat{\alpha}}(X \git \G)
\simeq
\QR^{\bullet + \hat{\alpha} }(X \git \G) \otimes_{KR^{\bullet}(\pt)[q^{\pm 1}]} KR^{\bullet}(\pt)\Lser{q}.
\]
\end{Thm}

\begin{proof}
This follows from the previous discussion and the isomorphism \eqref{eq:QRDecomp}. 
\end{proof}

The above discussion can be adapted to the case of a $\Z_2$-graded compact Lie group with trivial twist. In particular, the analogue of Theorem \ref{thm:TateVsQR} holds.

\subsection{The twisted elliptic Pontryagin character}
\label{sec:ellPhChar}

We construct a character map for Real quasi-elliptic cohomology using the Pontryagin character for $KR$-theory.

We first recall the twisted Chern and Pontryagin characters for $K$-theory. Let $\G$ be a finite group acting on a compact manifold $X$ and $\theta \in Z^2(B \G)$. Then $\uptau(\theta)$ determines $\uptau(\theta)([-]g) \in Z^1(B C_{\G}(g))$, which we interpret as a one dimensional representation $\mathbb{C}_{\uptau(\theta)([-]g)}$ of $C_{\G}(g)$. The twisted equivariant Chern character
\[
\ch^{\theta} : K^{i + \theta}_{\G}(X) \otimes_{\Z} \mathbb{C}
\xrightarrow[]{\sim}
\bigoplus_{g \in \pi_0(\G \git \G)} \left( H^{2 \bullet + i}(X^g) \otimes_{\Z} \mathbb{C}_{\uptau(\theta)([-]g)} \right)^{C_{\G}(g)},
\qquad
i=0,1
\]
can be seen as the composition of twisted Atiyah--Segal localization (equation \eqref{eq:twistedRealAtiyahSegal}) and the ordinary Chern character. Similarly, if $\Gh$ is a Real structure on $\G$ and $\hat{\theta} \in Z^{2+\pi}(B \Gh)$, there is a twisted equivariant Pontryagin character
\[
\ph^{\hat{\theta}} : KR^{\bullet + \hat{\theta}}_{\G}(X) \otimes_{\Z} \mathbb{C}
\rightarrow
\bigoplus_{g \in \pi_0(\G \git_R \Gh)} \left( H^{2 \bullet}(X^g) \otimes_{\Z} \mathbb{C}_{\uptau_{\pi}^{\refl}(\hat{\theta})([-]g)} \right)^{C^R_{\Gh}(g)}.
\]
When $X$ is a point, the restriction of $\ph^{\hat{\theta}}$ to degree zero is the Real character map and is in fact an isomorphism \cite[Theorem 3.10]{noohiyoung2022}.

Returning to the elliptic setting, let $\Gh$ act on $X$ and $\hat{\alpha} \in Z^{3}(B \Gh)$. Set $\hat{\theta}_g =\tilde{\uptau}_{\pi}^{\refl}(\hat{\alpha})_g$. For each $h \in \G$, the cocycle $\uptau_{\pi}^{\refl}(\hat{\theta}_g) \in Z^1(\mathcal{L}_{\pi}^{\refl} B C_{\Gh}^R(g))$ determines a one dimensional representation $\mathbb{C}_{\uptau_{\pi}^{\refl}(\hat{\theta}_g)([-]h)}$ of $C_{\Gh}^R(g,h)$.

\begin{Thm}
\label{thm:QEllRChern}
There is a homomorphism of abelian groups
\[
\ph^{\hat{\alpha}}:
\QR^{\bullet + \hat{\alpha}}(X \git \G) \otimes_{\Z} \mathbb{C}
\rightarrow
\bigoplus_{(g, h) \in \pi_0(\G^{(2)} \git_R \Gh)}
\left(
H^{\bullet}(X^{g,h}) \otimes_{\mathbb{C}} \mathbb{C}_{\uptau_{\pi}^{\refl}(\hat{\theta}_g)([-]h)} [q^{\pm 1}] \right)^{C^R_{\Gh}(g,h)}
\]
where $X^{g,h}$ is the joint fixed point set of $g, h \in \G$.
\end{Thm}

\begin{proof}
Proposition \ref{prop:QRGlobQuot} gives an isomorphism
\[
\QR^{\bullet+\hat{\alpha}}(X \git \G)
\simeq
\bigoplus_{g \in \pi_0(\G \git_R \Gh)}{^{\pi}}K^{\bullet+\hat{\theta}_g}_{\Lambda^R_{\Gh}(g)}(X^{g}).
\]
Pullback along the homomorphism $p$ from diagram \eqref{diag:finiteOrderRealConstLoop} gives a map
\begin{equation*}
\label{eq:cPulllback}
\bigoplus_{g \in \pi_0(\G \git_R \Gh)}{^{\pi}}K^{\bullet+\hat{\theta}_g}_{\Lambda^R_{\Gh}(g)}(X^{g})
\rightarrow
\bigoplus_{g \in \pi_0(\G \git_R \Gh)}{^{\pi}}K^{\bullet+\hat{\theta}_g}_{\mathbb{T} \rtimes_{\pi} C^R_{\Gh}(g)}(X^{g}).
\end{equation*}
Consider the summand labelled by $g \in \pi_0(\G \git_R \Gh)$. Since the normal subgroup $\mathbb{T} \trianglelefteq \mathbb{T} \rtimes_{\pi} C^R_{\Gh}(g)$ acts trivially on $X^g$, we can form the composition
\[
{^{\pi}}K^{\bullet+\hat{\theta}_g}_{\mathbb{T} \rtimes_{\pi} C^R_{\Gh}(g)}(X^{g})
\rightarrow
{^{\pi}}K^{\bullet+\hat{\theta}_g}_{C^R_{\Gh}(g)}(X^{g})[q^{\pm 1}]
\rightarrow
\bigoplus_{\substack{h \in \G \\ (g, h) \in \pi_0 ( \G^{(2)} \git_R \Gh)}}  {^{\pi}}K^{\bullet + \hat{\theta}_{g,h}}_{C_{\Gh}^R(g,h)}(X^{g,h})[q^{\pm 1}],
\]
where $\hat{\theta}_{g,h}$ is the restriction of $\hat{\theta}_g$ to $C^R_{\Gh}(g,h)$. The first and second maps in the composition are the isomorphism \eqref{eq:MackeyDecompExample} and Atiyah--Segal map \eqref{eq:twistedRealAtiyahSegal}, respectively. Further applying the twisted equivariant Chern or Pontryagin character to each summand of the codomain, according to whether $(g,h)$ lies in $\pi_0(\G^{(2)} \git \G)_{+1}$ or $\pi_0(\G^{(2)} \git \G)_{-1}$, gives a map
\[
{^{\pi}}K^{\bullet+\hat{\theta}_g}_{\mathbb{T} \rtimes_{\pi} C^R_{\Gh}(g)}(X^{g}) \otimes_{\Z} \mathbb{C}
\rightarrow
\bigoplus_{\substack{h \in \G \\ (g, h) \in \pi_0 ( \G^{(2)} \git_R \Gh)}}(H^{\bullet}(X^{g, h})\otimes_{\mathbb{C}} \mathbb{C}_{\uptau_{\pi}^{\refl}(\hat{\theta}_g)([-]h)}[q^{\pm 1}])^{C^R_{\Gh}(g, h)}.
\]
Assembling these maps for various $g \in \pi_0(\G \git_R \Gh)$ gives the desired map $\ph^{\hat{\alpha}}$.
\end{proof}

In \cite[\S 7]{huan2020}, a Chern character for twisted quasi-elliptic cohomology was constructed, taking the form
\[
\ch^{\alpha} : \Q^{\bullet+\alpha}(X \git \G) \otimes_{\Z} \mathbb{C}
\rightarrow
\bigoplus_{(g,h) \in \pi_0(\G^{(2)} \git \G)} (H^{\bullet}(X^{g,h}) \otimes_{\mathbb{C}} \mathbb{C}_{\uptau(\theta_g)([-]h)} [q^{\pm 1}])^{C_{\G}(g,h)}.
\]
It is immediate from the constructions that $\ch^{\alpha} \circ c = c \circ \ph^{\hat{\alpha}}$, where on the left hand side $c: \QR^{\bullet+\hat{\alpha}}(X \git \G) \rightarrow \Q^{\bullet+\alpha}(X \git \G)$ is the forgetful map while on the right $c$ is the natural forgetful map from the codomain of $\ph^{\hat{\alpha}}$ to that of $\ch^{\alpha}$.

\section{The Real Tate curve and Real quasi-elliptic cohomology} \label{sec:TateReal}

We give an explicit formula for the involution on the Tate curve. As shown in Example \ref{ex:Realtortate}, Real quasi-elliptic cohomology describes the torsion points of the Real Tate curve in a way which is compatible with the complex case.

\subsection{The Tate curve}
\label{sec:Tate}

We collect basic facts about the Tate curve \cite[\S 8]{katz1985}.

The cubic equation
\begin{equation}
\label{eq:cubCurve}
y^2+a_1xy +a_3y=x^3+ a_2x^2+a_4x+a_6
\end{equation}
defines an elliptic curve $E$. The Tate curve $\Tate(q)$ over $\Z\pser{q}$ os obtained from equation \eqref{eq:cubCurve} by setting $a_1=1$, $a_2=a_3=0$ and
\[
a_4=-5\sum_{n=1}^{\infty} n^3q^n/(1-q^n),
\qquad
a_6=\frac{1}{12}\sum_{n=1}^{\infty} (7n^5+5n^3)q^n/(1-q^n).
\]
The $N$-torsion points $T[N]$ of $\Tate(q)$ over $\Z[q^{\pm 1}]$ is the disjoint union of schemes $T_0[N], \dots, T_{N-1}[N]$, where
\[
T_i[N]=\Spec(\Z[q^{\pm 1}][x]/(x^N-q^i)).
\]
See \cite[\S 8.7]{katz1985}. There is an exact sequence of group schemes
\begin{equation*}
0 \rightarrow \mu_N \simeq T_0(N)  \xrightarrow[]{a_N} T[N] \xrightarrow[]{b_N} \Z[\frac{1}{N}]/\Z \rightarrow 0
\end{equation*}
where $a_N$ sends $\zeta\in \mu_N$ to $(\zeta, 0)$ and $b_N$ sends $(X, \frac{i}{N})$ to $\frac{i}{N} \mod \Z$. See \cite[Eq. (8.7.1.4)]{katz1985}.

Let $R$ be a $\Z[q^{\pm 1}]$-algebra with connected spectrum. An element of $T[N](R)$ is determined by a pair $(\xi, \frac{i}{N})$, where $0 \leq i \leq N-1$ and $\xi \in R$ satisfies $\xi^N=q^i$. Multiplication in $T[N](R)$ is defined by
\[
(\xi_1, \frac{i_1}{N}) \cdot (\xi_2, \frac{i_2}{N})
=
\begin{cases}
(\xi_1 \xi_2, \frac{i_1+i_2}{N}) & \mbox{if } i_1 + i_2 < N, \\
(\xi_1 \xi_2q^{-1}, \frac{i_1+i_2-N}{N}) & \mbox{if } i_1 + i_2 \geq N.
\end{cases}
\]
See \cite[Eq. (8.7.1.2)]{katz1985}.

In \cite[\S 8.7.2]{katz1985}, a smooth one dimensional commutative group scheme over
$\Z[q^{\pm 1}]$ is defined by $T = \bigsqcup_{\alpha \in \mathbb{Q} \cap [0,1)} T_{\alpha}$, where each $T_{\alpha} =\mathbb{G}_m$. There is an exact sequence
\[
1 \rightarrow \mathbb{G}_m\rightarrow T\rightarrow \mathbb{Q}/\Z\rightarrow 0
\]
of group schemes over $\Z[q^{\pm 1}]$. We interpret $T$ as a functor from the category of $\Z[q^{\pm 1}]$-algebras with connected spectrum to the category of abelian groups. For any $\Z[q^{\pm 1}]$-algebra $R$ with connected spectrum, $T(R)=(R^{\times}\times\mathbb{Q} ) \slash \langle (q, -1)\rangle$ is an abelian group and $T$ defines a functor from the category of $\Z[q^{\pm 1}]$-algebras with connected spectrum to the category of abelian groups. See \cite[Eq. (8.7.2.3)]{katz1985}.

The relationship between $T$ and $\Tate(q)$ is described by the isomorphism of ind-$\Z \pser{q}$-schemes $T_{\tor} \otimes_{\Z[q^{\pm 1}]} \Z \pser{q} \simeq \Tate(q)_{\tor}$ \cite[\S 8.8]{katz1985}.

\subsection{Involutions of the Tate curve} \label{invtater}

The elliptic curve $E$ defined by equation \eqref{eq:cubCurve} has an involution
\begin{equation}
(x, y)\mapsto -(x, y):=(x, y^*),
\end{equation}
where $y^*:=-a_1x-a_3-y$ \cite[\S 10]{husemoller2004}.
For the Tate curve over $\mathbb{C}$, this involution can be understood in terms of group inversion of $\mathbb{C}^{\times}$. Given $q \in \mathbb{C}$ which satisfies $0 < \vert q \vert < 1$, there is an isomorphism $\phi_q: \mathbb{C}^{\times}/q^{\Z}\rightarrow \Tate(q)$ such that if $ww^{\prime}=1$ in $\mathbb{C}^{\times}/q^{\Z}$, then $\phi_q(w^{\prime})=-\phi_q(w)$ \cite[Theorem 10.5.7]{husemoller2004}. In other words, $\phi_q$ is $\Z_2$-equivariant when $\mathbb{C}^{\times} \slash q^{\Z}$ is given the involution of induced by inversion of $\mathbb{C}^{\times}$.

As indicated in \cite{hahn2020}, complex conjugation on complex $K$-theory corresponds at the level of formal group laws to group inverse. Motivated by this, we use group inversion to define an involution of $\Tate(q)$ over any coefficient ring.

Using the model for torsion points from Section \ref{sec:Tate}, the involution at the level of the torsion part of $\Tate(q)$ over $\Z[q^{\pm }]$ can be described as follows. Denote by $\iota: T[N] \rightarrow T[N]$ the group inverse. Then $\iota(R): T[N](R)\rightarrow T[N](R)$ is given by 
\begin{equation*}
\iota(R)(\xi, \frac{i}{N})
=
\begin{cases}
(\xi^{-1}, \frac{0}{N}) & \mbox{if } i=0, \\
(\xi^{-1} q, \frac{N-i}{N}) & \mbox{if } 0 < i \leq N-1.
\end{cases}
\end{equation*}
In particular, $\iota$ maps $T_i[N]$ to $T_{N-i}[N]$ and makes the diagram
\begin{equation}\label{invTate2}
\begin{tikzpicture}[baseline= (a).base]
\node[scale=1.0] (a) at (0,0){
\begin{tikzcd}[column sep=2.5em,row sep=1.5em]
0 \arrow{r} & \mu_N \arrow{r}[above]{a_N} \arrow{d}[left]{(-)^{-1}} & T[N] \arrow{d}[right]{\iota} \arrow{r}[above]{b_N} & \Z[\frac{1}{N}]/\Z \arrow{d}[right]{(-)^{-1}} \arrow{r} & 0\\
0 \arrow{r} & \mu_N \arrow{r}[below]{a_N} & T[N] \arrow{r}[below]{b_N} &\Z[\frac{1}{N}]/\Z \arrow{r} & 0
\end{tikzcd}
};
\end{tikzpicture}
\end{equation}
commute. There is an isomorphism $T[N] \simeq \mu_N\times \Z[\frac{1}{N}]/\Z $ over $\Z[q^{\pm 1}][q^{\frac{1}{N}}]$ under which the involution becomes $\iota(\alpha^j, \frac{i}{N}) = (\alpha^{-j}, \frac{N-i}{N})$.

\subsection{Connection between $KR$, $\QR$ and the Tate curve} \label{r:QEllRTatecurve}
\label{sec:connectQR}

We study involutions of the formal group $\mathbb{G}_m$ over $K^0(\pt) \simeq \Z$ induced by Real structures. Given a finite abelian group $\A$, its Pontryagin dual is $\A^*=\Hom_{\Grp}(\A, S^1)$.

Consider first the product Real structure, $\hat{\A} = \A \times \Z_2$. Denote by $\omega$ the generator of $\Z_2$. There is an isomorphism
\[
\Spec(K^0_{\A}(\pt)) \xrightarrow[]{\sim} \Hom (\A^*, \mathbb{G}_m)
\]
of group schemes over $\Z$ which is natural in $\A$. Complex conjugation and pullback of $\A$-equivariant vector bundles, $V \mapsto \omega^*\overline{V}$, defines a graded ring involution $I: K^{\bullet}_{\A}(\pt) \rightarrow K^{\bullet}_{\A}(\pt)$. Restricting to degree zero, $\Spec(I)$ is an involution of $\Spec(K^0_{\A}(\pt))$ in the category of schemes over $\Z$. Since complex conjugation of vector bundles corresponds to the formal inverse of $\mathbb{G}_m$, the corresponding involution of $\Hom (\A^*, \mathbb{G}_m)$ sends $f$ to
\begin{equation}
\label{eq:invMorphisms}
f^{-1}: \A^* \xrightarrow[]{f} \mathbb{G}_m \xrightarrow[]{(-)^{-1}} \mathbb{G}_m.
\end{equation}
The ring $KR^0_{\A}(\pt)$ is the homotopy fixed points of the $\Z_2$-action on $K^0_{\A}(\pt)$ generated by $I$. Thus, $\Spec(KR^0_{\A}(\pt))$ is the homotopy fixed points of $\Spec(I)$ and is isomorphic to the homotopy fixed points $\Hom_{\Z_2} (\A^*, \mathbb{G}_m)$, where $\Z_2$ acts as in \eqref{eq:invMorphisms}.

In the case of arbitrary Real structure $\hat{\A}$, fix $\omega\in \hat{\A} \setminus \A$ and consider the involution $I^{\omega}: K^0_{\A}(\pt) \rightarrow K^0_{\A}(\pt)$, $V \mapsto P^{\omega}(V)$, where $P^{\omega}$ is defined in Section \ref{sec:RealRepThy}. The corresponding involution of $\Hom (\A^*, \mathbb{G}_m)$ sends $f$ to
\begin{equation}
\label{eq:invMorphismsGen}
I^{\omega} (f): \A^* \xrightarrow[]{\Ad_{\omega}^*} \A^* \xrightarrow[]{f} \mathbb{G}_m \xrightarrow[]{(-)^{-1}} \mathbb{G}_m.
\end{equation}
The ring $KR^0_{\A}(\pt)$ is the homotopy fixed points of the $\Z_2$-action on $K^0_{\A}(\pt)$ generated by $I^{\omega}$ and $\Spec(KR^0_{\A}(\pt))$ is the homotopy fixed points of $\Spec(I^{\omega})$. Finally, $\Spec(KR^0_{\A}(\pt))$ is isomorphic to $\Hom_{\Z_2} (\A^*, \mathbb{G}_m)$, where now $\Z_2$ acts by \eqref{eq:invMorphismsGen}.

\begin{Ex} \label{ex:Realtortate}
The $N$-torsion points of $\Tate(q)$ are $T[N] \simeq \Hom (\Z_N^*, \Tate(q))$ and, over $\Z[q^{\pm 1}]$, are isomorphic to $\Spec(\Q^0_{\Z_N}(\pt))$ \cite[Remark 3.13]{huan2018}. As shown in diagram \eqref{invTate2}, the involution of $T[N]$ is group inversion and the induced maps on $\mu_N$ and $\Z[\frac{1}{N}]/\Z \leq \mathbb{Q}/\Z$ are the respective group inverses.

Consider the Real structure $\Gh = D_{2N}$ on $\G=\Z_N$. The involution of $\pi_0(\G \git \G)$ induced by $\Gh$ is trivial. The group inverses on $\mu_N$ and $\mathbb{G}_m$ correspond to complex conjugation on vector bundles and induce on $\Hom (\Z_N^*, \Tate(q))$ the involution which sends $f$ to the composition \eqref{eq:invMorphismsGen}.
The homotopy fixed points of this $\Z_2$-action
are
\[
\Hom_{\Z_2} (\Z_N^*, \Tate(q))
\simeq
\QR^0_{\Z_N}(\pt),
\]
which we henceforth denote by $TR[N]$.
\end{Ex}

By Example \ref{ex:KRPointAllReal}, there is an isomorphism $KR^0_{\mathbb{T}}(\pt) \simeq KR^0(\pt)[q^{\pm 1}]$.
The isomorphisms \eqref{eq:QRZnPt} and \eqref{eq:QZnPt} imply a pushout square
\begin{equation*}
\begin{tikzpicture}[baseline= (a).base]
\node[scale=1.0] (a) at (0,0){
\begin{tikzcd}[column sep=3.5em,row sep=1.5em]
KR^0_{\mathbb{T}}(\pt) \arrow{r}[above]{c} \arrow{d} & K^0_{\mathbb{T}}(\pt) \arrow{d}\\
\QR^0_{\Z_N}(\pt) \arrow{r}[below]{c} & \Q^0_{\Z_N}(\pt) \arrow[ul, phantom, "\usebox\pullback" , very near start, color=black]
\end{tikzcd}
};
\end{tikzpicture}
\end{equation*}
whose spectrum is then a pullback square
\begin{equation*}
\begin{tikzpicture}[baseline= (a).base]
\node[scale=1.0] (a) at (0,0){
\begin{tikzcd}[column sep=3.5em,row sep=1.5em]
T[N] \arrow{r} \arrow{d} \arrow[dr, phantom, "\usebox\pullback" , very near start, color=black] & TR[N] \arrow{d}\\
\Spec(K^0_{\mathbb{T}}(\pt)) \arrow{r} & \Spec(KR^0_{\mathbb{T}}(\pt)).
\end{tikzcd}
};
\end{tikzpicture}
\end{equation*}

\section{Power operations for Real quasi-elliptic cohomology} \label{sec:powerOp}

We construct power operations for $\QR^{\bullet}$. After preliminary material on wreath products in Section \ref{sec:wreathProd}, in Section \ref{clarealpowerKRtw} we treat the case of twisted equivariant $KR$-theory.
We use the untwisted specialization of these results in Section \ref{QEllRPower} to construct power operations for untwisted $\QR^{\bullet}$.

\subsection{Wreath products and their Real variants}
\label{sec:wreathProd}

Let $\G$ be a group and $\Sigma_N$ the symmetric group on $N$ letters. The wreath product $\G\wr\Sigma_N$ is the set $\G^N \times \Sigma_N$ with group structure
\[
(g_1, \dots, g_N; \sigma) (h_1, \dots, h_N; \tau) =(g_1h_{\sigma^{-1}(1)}, \dots, g_N h_{\sigma^{-1}(N)}; \sigma\tau).
\]
We often write $(\underline{g};\sigma)$ for $(g_1, \dots, g_N; \sigma)$ in what follows.

Given a Real structure $\Gh$ on $\G$, the subgroup
\begin{equation*}
\widehat{\G\wr\Sigma_N}:= \{(\underline{g}; \sigma) \in \Gh \wr \Sigma_N \mid \pi(g_i) = \pi(g_j) \mbox{ for all } i, j \}
\end{equation*}
with grading $\pi: \widehat{\G\wr\Sigma_N} \rightarrow \Z_2$, $(\underline{g}; \sigma)\mapsto \pi(g_i)$, is a Real structure on $\G\wr\Sigma_N$.

Assume that $\Gh$ is finite.

\begin{Lem}
For each $N \geq 1$, the graded abelian group homomorphism $\wp_N: C^{\bullet + \pi}(B \Gh) \rightarrow C^{\bullet + \pi}(B \widehat{\G \wr \Sigma_N})$ defined on $\hat{\alpha} \in C^{n+\pi}(B \Gh)$ by
\begin{equation*}
\wp_N(\hat{\alpha})([a_1 \vert \cdots \vert a_n]) =
\prod_{j=1}^N \hat{\alpha}([g_{1_j} \vert g_{2_{\sigma_1^{-1}(j)}} \vert g_{3_{(\sigma_1\sigma_2)^{-1}(j)}} \vert \cdots \vert g_{n_{(\sigma_1\cdots \sigma_{n-1})^{-1}(j)}}]),
\end{equation*}
where $a_i=(\underline{g_i}; \sigma_i)\in \widehat{\G\wr\Sigma_N}$, is a cochain map.
\end{Lem}

\begin{proof}
Let $\hat{\alpha} \in C^{n-1 + \pi}(B \Gh)$. Direct calculations give
\begin{multline*}
d\wp_N(\hat{\alpha})([a_1 \vert \cdots \vert a_n])= \prod_{j=1}^N \Big( \hat{\alpha}([g_{2_j} \vert \cdots \vert g_{n_{(\sigma_1\cdots \sigma_{n-2})^{-1}(j)}}])^{\pi(a_1)} \\
\cdot \prod_{i=1}^{n-1}\hat{\alpha}([g_{1_j} \vert \cdots \vert g_{i_{(\sigma_1\cdots \sigma_{i-1})^{-1}(j)}}  g_{(i+1)_{(\sigma_1\cdots \sigma_{i})^{-1}(j)}} \vert \cdots \vert g_{n_{(\sigma_1\cdots \sigma_{n-1})^{-1}(j)}}])^{(-1)^{n-i}} \\
\cdot 
\hat{\alpha}([g_{1_j} \vert \cdots \vert g_{(n-1)_{(\sigma_1\cdots \sigma_{n-2})^{-1}(j)}}])^{(-1)^n} \Big)
\end{multline*}
and
\begin{multline*}
\wp_N(d\hat{\alpha})([a_1 \vert \cdots \vert a_n])= \prod_{j=1}^N \Big( \hat{\alpha}([g_{2_{\sigma_1^{-1}(j)}}\vert \cdots \vert g_{n_{(\sigma_1\cdots \sigma_{n-1})^{-1}(j)}}])^{\pi(g_{1_j})} \\  \cdot
\prod_{i=1}^{n-1} \hat{\alpha}([g_{1_j} \vert \cdots \vert g_{i_{(\sigma_1\cdots \sigma_{i-1})^{-1}(j)}}g_{(i+1)_{(\sigma_1\cdots \sigma_{i})^{-1}(j)}}\vert \cdots \vert g_{n_{(\sigma_1\cdots \sigma_{n-1})^{-1}(j)}}])^{(-1)^{n-i}} \\ \cdot 
\hat{\alpha}([g_{1_j} \vert \cdots \vert g_{(n-1)_{(\sigma_1\cdots \sigma_{n-2})^{-1}(j)}}])^{(-1)^n} \Big).
\end{multline*}
By definition, $\pi(g_{1_j}) = \pi (a_1)$ for all $j$. Since the product is taken over $j \in \{1,\dots, N\}$, we conclude that $d\wp_N(\hat{\alpha})([a_1 \vert \cdots \vert a_n]) = \wp_N(d\hat{\alpha})([a_1 \vert \cdots \vert a_n])$.
\end{proof}

Define $\wp_0(\hat{\alpha})= 1$ for all $\hat{\alpha} \in C^{\bullet + \pi}(B \Gh)$.

\begin{Lem} \label{twistprop}
Let $\hat{\alpha}, \hat{\beta} \in C^{n+\pi}(B\Gh)$. The cochain operations $\{\wp_N\}_{N\geq 0}$ satisfy:
\begin{enumerate}[wide,labelwidth=!, labelindent=0pt,label=(\roman*)]
\item $\wp_0(\hat{\alpha}) = 1$ and $\wp_1(\hat{\alpha})=\hat{\alpha}$.

\item $\wp_M(\hat{\alpha})\boxtimes \wp_N(\hat{\alpha}) = \iota^*\wp_{M+N}(\hat{\alpha})$,
where $\iota^*$
is the pullback along the inclusion $\iota: \widehat{\G\wr \Sigma_M}\times_{\Z_2}\widehat{\G\wr \Sigma_N}\rightarrow \widehat{\G\wr \Sigma_{M+N}}$

\item $\wp_M(\wp_N(\hat{\alpha}))= \iota^*\wp_{MN}(\hat{\alpha})$, where $\iota^*$ is the pullback along the inclusion $\iota: \widehat{\G\wr (\Sigma_N\wr\Sigma_M)} \rightarrow \widehat{\G\wr \Sigma_{MN}}$.

\item $\wp_N(\hat{\alpha}\hat{\beta}) = \iota^*(\wp_N(\hat{\alpha})\boxtimes \wp_N(\hat{\beta}))$, where $\iota^*$ is the pullback along the inclusion $\iota: \widehat{\G\wr \Sigma_N} \rightarrow \widehat{\G\wr (\Sigma_{N}\times \Sigma_N)}$.
\end{enumerate}
\end{Lem}

\begin{proof}
\begin{enumerate}[wide,labelwidth=!, labelindent=0pt,label=(\roman*)]
\item This is clear.

\item
By definition, $\iota((\underline{g}; \sigma), (\underline{h}; \tau)) = (\underline{g},\underline{h}; (\sigma,\tau))$. Writing $a_i= (\underline{g_i}; \sigma_i) \in \widehat{\G\wr\Sigma_M}$ and $b_i=(\underline{h_i}; \tau_i) \in \widehat{\G\wr\Sigma_N}$, we compute
\begin{align*}
&\left( \iota^*\wp_{M+N}(\hat{\alpha}) \right) ([(a_1, b_1) \vert \cdots \vert (a_n, b_n)])
=
\wp_{M+N}(\hat{\alpha}) ([\iota(a_1, b_1) \vert \cdots \vert \iota(a_n, b_n)])\\
&=
\prod_{i=1}^N \hat{\alpha}([g_{1_i} \vert g_{2_{\sigma_1^{-1}(i)}}  \vert \cdots \vert g_{n_{(\sigma_1\cdots \sigma_{n-1})^{-1}(i)}}]) \prod_{j=1}^N \hat{\alpha}([h_{1_j} \vert h_{2_{\tau_1^{-1}(j)}} \vert \cdots \vert h_{n_{(\tau_1\cdots \tau_{n-1})^{-1}(j)}}])\\
&=
\wp_M(\hat{\alpha})\boxtimes \wp_N(\hat{\alpha}) ([(a_1, b_1) \vert \cdots \vert (a_n, b_n)]).
\end{align*}

\item

View $\Sigma_N\wr\Sigma_M$ as a subgroup of $\Sigma_{MN}$ by letting $(\underline{\tau};\sigma) \in \Sigma_N\wr\Sigma_M$ act on the set $\{j_k \mid 1\leq j\leq M, \; 1\leq k \leq N\}$ by
\begin{equation}\label{eq:bigSymmAct}
(\underline{\tau}; \sigma)\cdot j_k=\sigma(j)_{\tau_{\sigma(j)}(k)}.
\end{equation}
Then $\iota(\underline{f}; \sigma) = (\underline{g_1},\dots, \underline{g_M}; (\underline{\tau};\sigma))$, where $f_i = (\underline{g_i}; \tau_i) \in \widehat{\G \wr \Sigma_N}$ and $\sigma \in \Sigma_M$.

For $i =1, \dots, n$, let $c_i=(\underline{f_i}; \sigma_i) \in \widehat{\G \wr (\Sigma_N \wr \Sigma_M)}$ with $f_{i_j} \in \widehat{\G \wr \Sigma_N}$ and $\sigma_i \in \Sigma_M$. With this notation, we have
\[
\wp_M(\wp_N(\hat{\alpha})) ([c_1 \vert \cdots \vert c_n])
=
\prod_{j=1}^N \wp_N(\hat{\alpha}) ([f_{1_j} \vert f_{2_{\sigma_1^{-1}(j)}} \vert \cdots \vert f_{n_{(\sigma_1 \cdots \sigma_{n-1})^{-1}(j)}}]).
\]
Writing $f_{i_j} = (\underline{g_{i_j}}; \tau_{i_j})$, we have
\begin{multline*}
\wp_N(\hat{\alpha}) ([f_{1_j} \vert f_{2_{\sigma_1^{-1}(j)}} \vert \cdots \vert f_{n_{(\sigma_1 \cdots \sigma_{n-1})^{-1}(j)}}]) \\
=
\prod_{k=1}^M \hat{\alpha}([g_{1_{j_k}} \vert g_{2_{\sigma_1^{-1}(j)_{\tau_{1_j}^{-1}(k)}}} \vert \cdots \vert g_{n_{(\sigma_1 \cdots \sigma_{n-1})^{-1}(j)_{(\tau_{1_j} \cdots \tau_{n-1_j})^{-1}(k)}}}])
\end{multline*}
from which we conclude
\[
\wp_M(\wp_N(\hat{\alpha})) ([c_1 \vert \cdots \vert c_n])
=
\prod_{j=1}^N \prod_{k=1}^M \hat{\alpha}([g_{1_{j_k}} \vert g_{2_{\sigma_1^{-1}(j)_{\tau_{1_j}^{-1}(k)}}} \vert \cdots \vert g_{n_{(\sigma_1 \cdots \sigma_{n-1})^{-1}(j)_{(\tau_{1_j} \cdots \tau_{n-1_j})^{-1}(k)}}}]).
\]
On the other hand,
\[
\left( \iota^*\wp_{MN}(\hat{\alpha}) \right) ([c_1 \vert \cdots \vert c_n])
=
\wp_{MN}(\hat{\alpha})  ([\iota(c_1) \vert \cdots \vert \iota(c_n)]).
\]
Writing $\iota(c_i)=(\underline{h_i}; \mu_i)$, so that $\mu_i=(\underline{\tau_i};\sigma_i)$, the definition of $\wp_{MN}$ gives
\[
\wp_{MN}(\hat{\alpha})  ([\iota(c_1) \vert \cdots \vert \iota(c_n)])
=
\prod_{l=1}^{MN} \hat{\alpha}([h_{1_l} \vert h_{2_{\mu_1^{-1}(l)}} \vert h_{3_{(\mu_1\mu_2)^{-1}(l)}} \vert \cdots \vert h_{n_{(\mu_1\cdots \mu_{n-1})^{-1}(l)}}]).
\]
Since $(\underline{h_i};\mu_i) = (\underline{g_{i_1}}, \dots, \underline{g_{i_M}}; (\underline{\tau_i};\sigma_i))$, after identifying the sets $\{1, \dots, MN\}$ and $\{j_k \mid 1 \leq j \leq M, \; 1 \leq k \leq N\}$, the previous expression becomes
\[
\wp_{MN}(\hat{\alpha})  ([\iota(c_1) \vert \cdots \vert \iota(c_n)])
=
\prod_{j=1}^M \prod_{k=1}^N \hat{\alpha}([g_{1_{j_k}} \vert g_{2_{\mu_1^{-1}(j_k)}} \vert g_{3_{(\mu_1\mu_2)^{-1}(j_k})} \vert \cdots \vert g_{n_{(\mu_1\cdots \mu_{n-1})^{-1}(j_k)}}]).
\]
Using equation \eqref{eq:bigSymmAct}, the previous expression becomes
\[
\prod_{j=1}^N \prod_{k=1}^M \hat{\alpha}([g_{1_{j_k}} \vert g_{2_{\sigma_1^{-1}(j)_{\tau^{-1}_{1_{\sigma_1(j)}}(k)}}} \vert \cdots \vert g_{n_{(\sigma_1 \cdots \sigma_{n-1})^{-1}(j)_{(\tau_{1_j} \cdots \tau_{n-1_j})^{-1}(k)}}}]),
\]
as required.

\item Writing $a_i= ( \underline{g_i}; \sigma_i) \in \widehat{\G\wr\Sigma_N}$, we compute
\begin{align*}
&\wp_N(\hat{\alpha}\hat{\beta}) ([a_1 \vert \cdots \vert a_n])
=\prod_{j=1}^N (\hat{\alpha}\hat{\beta}) ([g_{1_j} \vert g_{2_{\sigma_1^{-1}(j)}} \vert \cdots \vert g_{n_{(\sigma_1\cdots \sigma_{n-1})^{-1}(j)}}]) \\
=& \prod_{j=1}^N \hat{\alpha} ([g_{1_j} \vert g_{2_{\sigma_1^{-1}(j)}} \vert \cdots \vert g_{n_{(\sigma_1\cdots \sigma_{n-1})^{-1}(j)}}])\prod_{j=1}^N \hat{\beta} ([g_{1_j} \vert g_{2_{\sigma_1^{-1}(j)}}  \vert \cdots \vert g_{n_{(\sigma_1\cdots \sigma_{n-1})^{-1}(j)}}])
    \\ & = \iota^*(\wp_N(\hat{\alpha})\boxtimes \wp_N(\hat{\beta}))([a_1 \vert \cdots \vert a_n]) . \qedhere
\end{align*}
\end{enumerate}
\end{proof}

Wreath products are extended from groups to groupoids in \cite[\S 4.1]{ganter2007}. In the Real setting, let $\mathbb{G}$ be a groupoid with involution $(\iota_{\mathbb{G}},\Theta_{\mathbb{G}})$. Then $\iota:=\iota_{\mathbb{G}}^N\times \id_{\Sigma_N}$ and $\Theta:= \Theta_{\mathbb{G}}^N\times \id_{\Sigma_N}$ define an involution $(\iota, \Theta)$ of the wreath product $\mathbb{G}\wr \Sigma_N$. Denote by $\widehat{\mathbb{G}\wr \Sigma_N}$ the quotient groupoid $(\mathbb{G} \wr \Sigma_N) \git (\iota,\Theta)$.

\subsection{Power operations for twisted equivariant $KR$-theory}\label{clarealpowerKRtw}

\begin{Ex} 
\label{ex:symmPowerRealRepTw}
Let $\pi: \Gh \rightarrow \Z_2$ be a $\Z_2$-graded finite group and $\hat{\theta} \in Z^{2+\pi}(B \Gh)$. Let $(V,\rho_V)$ be a $\hat{\theta}$-twisted Real representation of $\G$. The group $\widehat{\G\wr\Sigma_N}$ acts on $V^{\otimes N}$ by
\[
\rho_{V^{\boxtimes N}}(\underline{g}; \sigma) (v_1\otimes\cdots \otimes v_N)=
\rho_V(g_1) v_{\sigma^{-1}(1)} \otimes \cdots \otimes \rho_V(g_N)v_{\sigma^{-1}(N)}.
\]
A direct calculation verifies that $\rho_{V^{\boxtimes N}}$ is a $\wp_N(\hat{\theta})$-twisted Real representation of $\G \wr \Sigma_N$.
\end{Ex}

Example \ref{ex:symmPowerRealRepTw} globalizes as follows. Given a $(\pi,\hat{\theta})$-twisted $\Gh$-equivariant vector bundle $V \rightarrow X$, the external tensor product $V^{\boxtimes N} \rightarrow X^N$ is naturally a $(\pi,\wp_N(\hat{\theta}))$-twisted $\widehat{\G\wr\Sigma_N}$-equivariant vector bundle. The assignment $V \mapsto V^{\boxtimes N}$ extends to
\begin{equation*}
P^{R,\hat{\theta}}_N: {^{\pi}}K^{\bullet + \hat{\theta}}_{\Gh} (X) \rightarrow {^{\pi}}K^{\bullet + \wp_N(\hat{\theta})}_{\widehat{\G\wr\Sigma_N}} (X^N). 
\end{equation*}
The map $P^{R,\hat{\theta}}_N$ can be realized as the composition
\[
{^{\pi}}K^{\bullet+\hat{\theta}}_{\Gh}(X) \xrightarrow{\Delta} {^{\pi}}K^{\bullet+\hat{\theta}}_{\Gh}(X)\otimes_{\Z} \cdots \otimes_{\Z}{^{\pi}}K^{\bullet+\hat{\theta}}_{\Gh}(X) \rightarrow {^{\pi}}K^{\bullet+\hat{\theta}^N}_{\Gh^N}(X^N)
\]
whose first map is the diagonal and second is the $N$-fold composition of the K\"{u}nneth morphisms of Section \ref{kunnethfmk}. Note that the composition factors through ${^{\pi}}K^{\bullet + \wp_N(\hat{\theta})}_{\widehat{\G\wr\Sigma_N}} (X^N)$. Set $P^{R,\hat{\theta}}_0(V)=1$ for all $V \in {^{\pi}}K^{\bullet + \hat{\theta}}_{\Gh} (X)$.

\begin{Prop}
Let $V \in V \in {^{\pi}}K^{\bullet + \hat{\theta}}_{\Gh} (X)$ and $W\in {^{\pi}}K^{\bullet+\hat{\eta}}_{\Gh}(X)$. The operations $\{P^{R, \hat{\theta}}_N\}_{N \geq 0}$ have the following properties:
\begin{enumerate}[wide,labelwidth=!, labelindent=0pt,label=(\roman*)]
\item $P^{R, \hat{\theta}}_0(V)=1$ and $P^{R, \hat{\theta}}_1(V)=V$.

\item \label{FMKPowerExtProd} The (external) product of two operations is
\[P^{R, \hat{\theta}}_M(V) \boxtimes P^{R, \hat{\theta}}_N(V) = \Res^{\G\wr\Sigma_{M+N}}_{\G\wr(\Sigma_M\times \Sigma_N)} P^{R, \hat{\theta}}_{M+N}(V).\]

\item The composition of two operations is 
\[P^{R, \wp_N(\hat{\theta})}_M(P^{R, \hat{\theta}}_N(V)) = \Res^{\G\wr\Sigma_{MN}}_{\G\wr(\Sigma_N\wr\Sigma_M)} P^{R, \hat{\theta}}_{MN}(V). \]

\item The operations preserve external products:
\[P^{R, \hat{\theta}\hat{\eta}}_N(V\boxtimes W) = \Res^{\G\wr(\Sigma_N\times \Sigma_N)}_{\G\wr\Sigma_N} (P^{R, \hat{\theta}}_N(V)\boxtimes P^{R, \hat{\eta}}_N(W)).
\]
\end{enumerate}
\label{FMKPower}\end{Prop}

\begin{proof}
The first statement is immediate. It suffices to verify the remaining statements at the level of twisted equivariant vector bundles. For part \ref{FMKPowerExtProd}, let $V \rightarrow X$ and $W \rightarrow Y$ be $(\pi,\hat{\theta})$-twisted and $(\pi,\hat{\eta})$-twisted $\Gh$-equivariant vector bundles, respectively. By Proposition \ref{twistprop}, the cocycle twists on each side of each claimed identity are equal. Moreover, the canonical isomorphism of $\pi$-twisted $\widehat{\G \wr (\Sigma_M \times \Sigma_N)}$-equivariant vector bundles $V^{\boxtimes M} \boxtimes V^{\boxtimes N} \simeq V^{\boxtimes (M+N)}$ lifts to the the twisted case. The remaining parts are analogous.
\end{proof}

\subsection{Power operations for Real quasi-elliptic cohomology}\label{QEllRPower}

Power operations for $\Q^{\bullet}$ were constructed in \cite[\S 4.1]{huan2018b} using power operations for orbifold $K$-theory \cite{ganter2013} and loop spaces of symmetric power groupoids. We generalize these ideas to construct power operations for $\QR^{\bullet}$.

Let a $\Z_2$-graded compact Lie group $\Gh$ act on a manifold $X$. The power operation will take the form of maps
\begin{equation*}
\mathbb{P}^R_N=\prod_{(\underline{g}; \sigma)\in
\pi_0((\G\wr\Sigma_N)^{\tor} \git \G\wr\Sigma_N)}\mathbb{P}^R_{ (\underline{g}; \sigma)}:
\QR^{\bullet}_\G(X)\rightarrow \QR^{\bullet}_{\G\wr\Sigma_N}(X^N),
\qquad
N \in \Z_{\geq 0}
\end{equation*}
where $\mathbb{P}^R_{ (\underline{g}; \sigma)}$ is the composition
\begin{multline}
\QR^{\bullet}_{\G}(X) \xrightarrow[]{U^*_R}
{^{\pi}}K^{\bullet}(\Lambda^{\refl,1}_{(\underline{g}; \sigma)}(X))
\xrightarrow[]{\scaleLam k}
{^{\pi}}K^{\bullet}(\Lambda^{\refl,\var}_{(\underline{g}; \sigma)}(X)) \xrightarrow[]{\boxtimes}
\\ {^{\pi}}K^{\bullet}(d^{\refl}_{(\underline{g};\sigma)}(X)) \xrightarrow[]{f^*_{(\underline{g};\sigma)}} {^{\pi}}K^{\bullet }_{\Lambda^R_{\widehat{\G\wr\Sigma_N}}(\underline{g}; \sigma)}((X^N)^{(\underline{g}; \sigma)}).
\label{pgs2r}
\end{multline}
The groupoids and morphisms appearing in this composition are defined in the following sections.

\subsubsection{The funtors $U_R$ and $\scaleLam k$} \label{construrescale}

We require some preliminary definitions. Fix an integer $k \geq 1$.

\begin{Def}[{\cite[Definition 4.1]{huan2018b}}]
Let $\Lambda^k(X \git  \G)$ be the groupoid with the same objects as
$\Lambda(X\git \G)$, that is, pairs $(x,g) \in X \times \G^{\tor}$ with $x \in X^g$, and morphisms
\[
\Hom_{\Lambda^k(X \git  \G)}((x,g),(x^{\prime},g^{\prime})) = \Lambda^k_{\G}(g, g')
\]
where $\Lambda^k_{\G}(g, g')$ is the quotient of $\mathbb{R} \times T_{\G}(g, g')$ by the equivalence relation generated by $(t,h)\sim (t-k,gh)$.
\label{lambdaoidefk}
\end{Def}

Fix $\omega\in \Gh\setminus \G$. Let $(\iota_{\omega,k}, \Theta_{\omega,k})$ be the involution of $\Lambda^k(X\git \G)$ given on objects and morphisms by $\iota_{\omega, k}(x, g) = (x\omega, \omega^{-1}g^{-1}\omega)$ and $\iota_{\omega, k}([t,h]) = [-t,  \omega^{-1} h^{-1} \omega]$, respectively, and such that the component of $\Theta_{\omega, k}$ at $(x,g)$ is $[0,\omega^2]$. 

\begin{Def}
Let $\Lambda_{\pi}^{\refl, k} (X \git \Gh)$ be the quotient groupoid $\Lambda^k(X \git \G) \git (\iota_{\omega, k},\Theta_{\omega, k})$.
\end{Def}

The obvious morphism $\Lambda_{\pi}^{\refl, k} (X \git \Gh) \rightarrow B(\mathbb{R} \slash k \Z \rtimes_{\pi} \Z_2)$ induces a $B \mathsf{O}_2$-grading
\begin{equation}
\label{eq:BOGradingLambda}
\Lambda_{\pi}^{\refl, k} (X \git \Gh) \rightarrow B(\mathbb{R} \slash k \Z \rtimes \Z_2) \rightarrow B(\mathbb{R} \slash \Z \rtimes \Z_2) \simeq B \mathsf{O}_2
\end{equation}
which recovers that of $\Lambda_{\pi}^{\refl} (X \git \Gh)$ when $k=1$. In particular, $\Lambda_{\pi}^{\refl, k} (X \git \Gh)$ is $B \Z_2$-graded.


\begin{Def}
Let $\scale k :\Lambda_{\pi}^{\refl, k}(X\git\Gh)\rightarrow \Lambda_{\pi}^{\refl}(X\git\Gh)$ be the functor which is the identity on objects and sends a morphism $[t,h]$ to $[\frac{t}{k},h]$.
\end{Def}

Pullback along $\scale k$ is a map $\scale k: KR^{\bullet}(\Lambda(X \git \G))\rightarrow KR^{\bullet}(\Lambda^k(X \git \G))$. The composition law $(\scale k)_{k^{\prime}}=\scale {kk^{\prime}}$ holds.

We require a wreath product of the groupoid $\Lambda_{\pi}^{\refl, k}(X \git \Gh)$, which is a Real version of the wreath product $\Lambda^k(X \git \G)\wr_{\mathbb{T}}\Sigma_N$ defined in \cite[Definition 4.2]{huan2018b}.

\begin{Def}
Let $\Lambda_{\pi}^{\refl, k}(X \git \Gh)\wr_{\mathsf{O}_2}\Sigma_N$ be the subgroupoid of $\Lambda_{\pi}^{\refl, k}(X \git \Gh)\wr\Sigma_N$ on morphisms $([t_1,h_1], \dots, [t_N,h_{N}];\tau)$ for which all $[t_i,h_i]$ have the same image in $B\mathsf{O}_2$ with respect to the $B\mathsf{O}_2$-grading \eqref{eq:BOGradingLambda}.
\end{Def}

Before we construct the groupoid $\Lambda_{\pi, (\underline{g}; \sigma)}^{\refl,1}(X)$, we introduce some notation. Given $g, g^{\prime} \in \G$, define $T^R_{\Gh}(g, g^{\prime})=\{\omega \in \Gh \mid \omega g^{\pi(\omega )}=g^{\prime} \omega\}$. For each integer $r \geq 1$, let $\Lambda^{R, r}_{\Gh}(g, g')$ be the quotient of $\mathbb{R} \times T^{R}_{\Gh}(g, g')$ by the equivalence relation generated by $(t,h)\sim (t-r,gh)$.

\begin{Def}
For each $(\underline{g}; \sigma)\in (\G\wr\Sigma_N)^{\tor}$, let $\Lambda_{\pi, (\underline{g};\sigma)}^{\refl,r}(X)$ be the groupoid 
with objects $\bigsqcup_k \bigsqcup_{i \in \Cyc_k(\sigma)}X^{g_{i_k}\cdots g_{i_1}}$, where $k \in \{1, \dots, N\}$, and morphisms
$$\bigsqcup_k \bigsqcup_{i,j \in \Cyc_k(\sigma)}\Lambda^{R, r}_{\Gh}(g_{i_k}\cdots g_{i_1}, g_{j_k}\cdots
g_{j_1})\times X^{g_{i_k}\cdots g_{i_1}}.$$ 
\end{Def}

There is a $B \Z_2$-graded equivalence $\Lambda_{\pi, (\underline{g}; \sigma)}^{\refl,r}(X) \simeq \Lambda^r_{(\underline{g}; \sigma)}(X) \git (\iota_{\omega, r},\Theta_{\omega, r})$, where $\Lambda^r_{(\underline{g}; \sigma)}(X) $ is the groupoid defined in \cite[\S4.2]{huan2018}.

\begin{Def}
Let
$
U_R:\Lambda_{\pi, (\underline{g}; \sigma)}^{\refl,1}(X)\rightarrow \Lambda_{\pi}^{\refl}(X \git
\Gh)
$
be the functor which sends an object $x \in X^{g_{i_k}\cdots g_{i_1}}$ to $x \in X^{g_{i_k}\cdots g_{i_1}}$ and a morphism $[t,h]$ to $[t,h]$.
\end{Def}

It is immediate that $U_R$ is $B \Z_2$-graded.

\begin{Def}
Let $\Lambda^{\refl, \var}_{\pi, (\underline{g}; \sigma)}(X)$ be the groupoid with the same objects as $\Lambda^{\refl, 1}_{\pi, (\underline{g}; \sigma)}(X)$ and morphisms
$$\bigsqcup_k \bigsqcup_{i,j \in \Cyc_k(\sigma)}\Lambda^{R,k}_{\Gh}(g_{i_k}\cdots g_{i_1}, g_{j_k}\cdots
g_{j_1})\times X^{g_{i_k}\cdots g_{i_1}}. $$
\end{Def}

\begin{Def}
Let
$
\scaleLam k: \Lambda^{\refl, \var}_{\pi, (\underline{g};\sigma)}(X)\rightarrow \Lambda^{\refl, 1}_{\pi, (\underline{g}; \sigma)}(X)
$
be the functor which is the identity on objects and sends a morphism $[t,g]$ to $[\frac{t}{k},g]$.
\end{Def}

We also denote by $\scaleLam k$ the pullback
$
\scaleLam k: KR^{\bullet}(\Lambda^1_{(\underline{g}; \sigma)}(X))\rightarrow KR^{\bullet}(\Lambda^{\var}_{(\underline{g};\sigma)}(X)).
$

\subsubsection{The equivalence $f_{(\underline{g}; \sigma)}$} \label{constrfgsigma}

We begin by discussing the Real centralizer $C^R_{\widehat{\G\wr\Sigma_N}}(\underline{g}; \sigma)$.

\begin{Ex}
Let $(\underline{g}; \sigma) \in \G\wr\Sigma_N$. An element $(\underline{h};\tau)\in \widehat{\G\wr\Sigma_N}$ is in $C^R_{\widehat{\G\wr\Sigma_N}}(\underline{g}; \sigma)$ if and only if one of the following conditions hold:
\begin{enumerate}[wide,labelwidth=!, labelindent=0pt]
\item $\pi(\underline{h};\tau) = 1$ and $\tau\sigma=\sigma\tau$ and $g_{\sigma(\tau(i))}h_{\tau(i)}=h_{\tau(\sigma(i))}g_{\sigma(i)}$ for all $i=1, \dots, N$. In this case, there is a bijection $\tau: \Cyc_k(\sigma) \rightarrow \Cyc_k(\sigma)$, sending $i=(i_1,\dots, i_k)$ to $\tau(i)=j=(j_1, \dots, j_k)$, where  $\tau(i_l)=j_{l+m_i}$, $l \in \Z_k$, for some $m_i$ which depends only on $\tau$ and $i$. It follows that $g_{j_l}h_{j_{l-1}}=h_{j_l}g_{i_{l-m_i}}$, $l \in \Z_k$ and that
\begin{equation}
\beta^{(\underline{h};\tau)}_{j,i}:= h_{j_k}g^{-1}_{i_{1-m_i}}\cdots
g^{-1}_{i_{k-1}}g^{-1}_{i_k}=g^{-1}_{j_1}\cdots g^{-1}_{j_{m_i}} h_{j_{m_i}}
\in 
T_{\G} (g_{j_k}\cdots g_{j_1}, g_{i_k}\cdots g_{i_1})
\label{betaDef}
\end{equation}
with $\pi({\beta}^{(\underline{h}; \tau)}_{j,i})= \pi(\underline{h};\tau)$.

\item $\pi(\underline{h};\tau) = -1$
and 
$\tau\sigma^{-1}=\sigma\tau$ 
and $g_{\tau(\sigma(i))} h_{\tau(i)} = h_{ \tau(\sigma(i))} g^{-1}_{i}$ for all $i =1, \dots, N$. In this case, $\tau$ sends a $k$-cycle $i=(i_1,\dots, i_k)$ of $\sigma$ to the reverse of another $k$-cycle, say $\tau(i)=j=(j_1, \dots, j_k)$, so that $\tau(i_l)=j_{-l+m_i}$, $l\in \Z_k$, for some $m_i$. It follows that $g_{j_r} h_{j_{r-1}} = h_{j_r} g^{-1}_{i_{m_i+1-r}}$, $r \in \Z_k$, from which we deduce
\begin{equation*}
    h_{j_k}g_{i_{m_i}}\cdots g_{i_{m_i+1-r}} = g_{j_1}^{-1}\cdots g_{j_r}^{-1} h_{j_{r}},
    \qquad r \in \Z_k
\end{equation*}
and the other $h_{j_r}$ are determined by $h_{j_k}$ and $\underline{g}$. Moreover,
\begin{equation*} \hat{\beta}^{(\underline{h}; \tau)}_{j,i}:= h_{j_k}g_{i_{m_i}}\cdots g_{i_{1}} = g_{j_1}^{-1}\cdots g_{j_{m_i}}^{-1} h_{j_{m_i}} \in T_{\Gh}^R(g_{j_k}\cdots g_{j_1}, g_{i_k}\cdots g_{i_1})
\label{hbetadef}
\end{equation*}
with $\pi(\hat{\beta}^{(\underline{h}; \tau)}_{j,i})= \pi(\underline{h};\tau)$.
\qedhere
\end{enumerate}
\end{Ex}

\begin{Ex}
Let $\sigma\in\Sigma_N$ correspond to the partition $N = \sum_k kN_k$, so that $\sigma$ has exactly $N_k$ $k$-cycles. Assume that each $k$-cycle is written as $(i_1, \dots, i_k)$ with $i_1< \cdots < i_k$.
Set $\underline{N} = \{1,\dots, N\}$. Consider the orbits of the bundle $\G\times\underline{N}\rightarrow\underline{N}$ under $(\underline{g}; \sigma)\in \G\wr\Sigma_N$. The $\sigma$-orbits of $\underline{N}$ correspond to cycles of $\sigma$. Correspondingly, $\G\times\underline{N}\rightarrow\underline{N}$ is a disjoint union
\[
\bigsqcup_k \bigsqcup_{i \in \Cyc_k(\sigma)} (\G \times i \rightarrow i)
\]
and each $\G\times i \rightarrow i$ is a $(\underline{g}; \sigma)$-orbit. Two $\G$-bundles
\[
\G\times \{i_1, \dots, i_k\}\rightarrow \{i_1, \dots,i_k\},
\qquad
\G\times \{j_1, \dots, j_l\}\rightarrow \{j_1, \dots, j_l\}
\]
are Real $(\underline{g}; \sigma)$-isomorphic if and only if $k=l$ and $T^R_{\Gh}(g_{i_k}\cdots g_{i_1},g_{j_k}\cdots g_{j_1}) \neq \varnothing$. Let $W^{\sigma}_i$ denote the set of all the $\G$-subbundles $\G\times j \rightarrow j$ which are Real $(\underline{g};\sigma)$-isomorphic to $\G\times i \rightarrow i$. The discussion above shows that $W^{\sigma}_i$ is in bijection with
\[
\{j \in \Cyc_k(\sigma) \mid T^R_{\Gh}(g_{i_k}\cdots g_{i_1}, g_{j_k}\cdots g_{j_1}) \neq \varnothing\}.
\]
Let $M^{\sigma}_{i}$ be the cardinality of $W^{\sigma}_i$ and $\alpha^i_1, \dots, \alpha^i_{M^{\sigma}_i}$ the elements of $W^{\sigma}_i$, labelled so that $i=\alpha^i_1 \in W^{\sigma}_i$. Note that if $T^R_{\Gh}(g_{i_k}\cdots g_{i_1}, g_{j_k}\cdots g_{j_1}) = T_{\G}(g_{i_k}\cdots g_{i_1}, g_{j_k}\cdots g_{j_1})$, then $W^{\sigma}_i$ reduces to the set defined in \cite[\S4]{huan2018b}.

Let $i \in \Cyc_k(\sigma)$ and $j \in \Cyc_l(\sigma)$. If $k=l$ and $T^R_{\Gh}(g_{i_k}\cdots g_{i_1}, g_{j_k}\cdots g_{j_1}) \neq \varnothing$, then $W^{\sigma}_i = W^{\sigma}_j$ and $W^{\sigma}_i$ and $W^{\sigma}_j$ are disjoint otherwise. Fix representatives $\theta_k$ of $\Cyc_k(\sigma)$ such that $\Cyc_k(\sigma) = \bigsqcup_{i\in\theta_k}W^{\sigma}_i.$
\end{Ex}

Given $B \mathbb{T}$-graded groupoids $p_{\mathbb{G}}:\mathbb{G} \rightarrow B \mathbb{T}$ and $p_{\mathbb{H}}: \mathbb{H} \rightarrow B\mathbb{T}$, let $\mathbb{G} \times_{\mathbb{T}} \mathbb{H}$ be the subgroupoid of $\mathbb{G} \times \mathbb{H}$ on morphisms of the form $(f,g)$ with $p_{\mathbb{G}}(f)=p_{\mathbb{H}}(g)$. If instead $\mathbb{G}$ and $\mathbb{H}$ are $B \mathsf{O}_2$-graded, then there is a fibre product $\mathbb{G} \times_{\mathsf{O}_2} \mathbb{H}$.

\begin{Def} 
Let $(\underline{g}; \sigma)\in (\G\wr\Sigma_N)^{\tor}$.
\begin{enumerate}[wide,labelwidth=!, labelindent=0pt,label=(\roman*)]
\item (See \cite[\S4]{huan2018b}.) Let $d_{(\underline{g}; \sigma)}(X)$ be the full subgroupoid of
\begin{equation}
\label{eq:fibProfGrpd}
\prod\limits_{k}\!{_{\mathbb{T}}}\prod\limits_{i\in\theta_k}\!{_{\mathbb{T}}}\Lambda^k(X \git \G)\wr_{\mathbb{T}}\Sigma_{M^{\sigma}_i}
\end{equation}
with objects $\prod_k\prod_{ i \in \Cyc_k(\sigma)} X^{g_{i_k}\cdots g_{i_1}}$. 

\item Let $d^{\refl}_{(\underline{g}; \sigma)}(X)$ be the full subgroupoid of
\[
\prod\limits_{k}\!{_{\mathsf{O}_2}}\prod\limits_{i\in\theta_k}\!{_{\mathsf{O}_2}}\Lambda_{\pi}^{\refl, k}(X\git \Gh)\wr_{\mathsf{O}_2}\Sigma_{M^{\sigma}_i}
\]
with objects $\prod_k\prod_{i \in \Cyc_k(\sigma)}X^{g_{i_k}\cdots g_{i_1}}$. 
\end{enumerate}
\label{ADgroupoid}
\end{Def}

The involution of $\Lambda^k(X \git \G)$ induced by $\Gh$ induces one on the groupoid \eqref{eq:fibProfGrpd} under which $d_{(\underline{g};\sigma)}$ is stable if and only if $(\underline{g};\sigma) \in \pi_0((\G\wr\Sigma_N)^{\tor} \git \G \wr\Sigma_N)_{-1}$. In particular, $d^{\refl}_{(\underline{g}; \sigma)}(X)$ is canonically $B \mathsf{O}_2$-graded, and hence $B\Z_2$-graded.

\begin{Thm} 
\begin{enumerate}[wide,labelwidth=!, labelindent=0pt,label=(\roman*)]
\item For each $(\underline{g}; \sigma)\in \pi_0((\G\wr\Sigma_N)^{\tor} \git \G\wr\Sigma_N)$, there is an isomorphism $f_{(\underline{g}; \sigma)}$ such that, for any $\omega\in \widehat{\G\wr\Sigma_N}\setminus \G\wr\Sigma_N$, the following diagram commutes:
\begin{equation}\label{fgsfree}
\begin{tikzpicture}[baseline= (a).base]
\node[scale=1.0] (a) at (0,0){
\begin{tikzcd}[column sep=5.5em,row sep=1.5em]
(X^N)^{(\underline{g};\sigma)} \git \Lambda_{\G\wr\Sigma_N}(\underline{g}; \sigma) \arrow{r}[above]{f_{(\underline{g}; \sigma)}} \arrow{d}[left]{\iota_{\omega}}& d_{(\underline{g}; \sigma)}(X) \arrow{d}[right]{\iota_{\omega}} \\
(X^N)^{\omega^{-1}(\underline{g};\sigma)^{-1}\omega}\git\Lambda_{\G\wr\Sigma_N}(\omega^{-1}(\underline{g}; \sigma)^{-1}\omega) \arrow{r}[below]{f_{\omega^{-1}(\underline{g}; \sigma)^{-1}\omega}} & d_{\omega^{-1}(\underline{g}; \sigma)^{-1}\omega}(X).
\end{tikzcd}
};
\end{tikzpicture}
\end{equation}

\item For each $(\underline{g}; \sigma)\in \pi_0((\G\wr\Sigma_N)^{\tor} \git \G\wr\Sigma_N)$, the isomorphism $f_{(\underline{g};\sigma)}$ induces a $B \Z_2$-graded equivalence $f_{(\underline{g};\sigma)}: (X^N)^{(\underline{g};\sigma)}\git\Lambda^R_{\widehat{\G\wr\Sigma_N}}(\underline{g}; \sigma) \xrightarrow[]{\sim} d^{\refl}_{(\underline{g}; \sigma)}(X)$.
\end{enumerate}
\end{Thm}

\begin{proof}
\begin{enumerate}[wide,labelwidth=!, labelindent=0pt,label=(\roman*)]
\item Recall from \cite[Theorem 4.9]{huan2018b} that the isomorphism $f_{(\underline{g};\sigma)}$ is defined on an object $x=(x_1, \dots, x_N)\in (X^N)^{(\underline{g}; \sigma)}$ by
\[
f_{(\underline{g}; \sigma)}(x) = \prod_k\prod_{i \in \Cyc_k(\sigma)}x_{i_k}
\]
and on a morphism $[t,(\underline{h};\tau)] \in \Lambda_{\G \wr \Sigma_N}(\underline{g};\sigma)$ by
\[
f_{(\underline{g}; \sigma)}([t,(\underline{h}; \tau)])
=
\prod_k \prod_{i\in\theta_k}([m_1+t,\beta^{(\underline{h}; \tau)}_{\tau(1), 1}], \dots, [m_{M^{\sigma}_i}+t,\beta^{(\underline{h}; \tau)}_{\tau(M^{\sigma}_i), M^{\sigma}_i}];\tau_{|{W^{\sigma}_i}}).
\]
Here $\tau_{|{W^{\sigma}_i}}$ denotes the permutation induced by $\tau$ on the set $W^{\sigma}_i=\{\alpha^i_1, \alpha^i_2, \dots, \alpha^i_{M^{\sigma}_i}\}$ and $\tau(i_l)=j_{l+m_i}$.

Without loss of generality, we may assume that $\omega = (\eta, \dots, \eta; 1)\in \widehat{\G\wr\Sigma_N}\setminus \G\wr\Sigma_N$ for some $\eta\in \Gh\setminus \G$, so that $\omega^{-1}(\underline{g}; \sigma)^{-1}\omega= (\eta^{-1}g_{\sigma(1)}^{-1}\eta, \dots, \eta^{-1} g_{\sigma(N)}^{-1}\eta; \sigma^{-1})$. At the level of objects, the clockwise and counterclockwise compositions of diagram \eqref{fgsfree} send an object $x \in (X^N)^{(\underline{g};\sigma)}$ to
\[
x \mapsto \prod_k\prod_{i \in \Cyc_k(\sigma)}x_{i_k} \mapsto \prod_k\prod_{i \in \Cyc_k(\sigma)}x_{i_k} \eta
\]
and
\[
x \mapsto x\omega = (x_1\eta, \dots, x_N\eta) \mapsto \prod_k\prod_{i \in \Cyc_k(\sigma)}x_{i_k} \eta,
\]
respectively. Similarly, the clockwise and counterclockwise compositions send a morphism $[t,(\underline{h}; \tau)]$ to
\begin{multline*}
[t,(\underline{h}; \tau)] \mapsto
\prod_k \prod_{i\in\theta^{(\sigma)}_k}([m_1+t,\beta^{(\underline{h}; \tau)}_{\tau(1), 1}], \dots, [m_{M^{\sigma}_i}+t,\beta^{(\underline{h}; \tau)}_{\tau(M^{\sigma}_i), M^{\sigma}_i}];\tau_{|{W^{\sigma}_i}})
\mapsto \\
\prod_k \prod_{i^{-1}\in\theta^{(\sigma^{-1})}_k}([-m_{M^{\sigma}_i}-t,\eta^{-1} (\beta^{(\underline{h}; \tau)}_{\tau(M^{\sigma}_i), M^{\sigma}_i})^{-1} \eta], \dots, [-m_1-t,\eta^{-1} (\beta^{(\underline{h}; \tau)}_{\tau(1), 1})^{-1} \eta];\tau^{-1}_{|{W^{\sigma}_i}}),
\end{multline*}
where $\beta$ is defined using $\underline{g}$ via equation \eqref{betaDef}, and
\begin{multline*}
[t,(\underline{h}; \tau)]
\mapsto
[-t,\omega^{-1}(\underline{h}; \tau)^{-1}\omega]
\mapsto \\
\prod_k \prod_{i\in\theta^{(\sigma^{-1})}_k}([m^{\prime}_1-t,\beta^{(\eta^{-1}\underline{h}^{-1} \eta; \tau^{-1})}_{\tau^{-1}(1), 1}], \dots, [m^{\prime}_{M^{\sigma^{-1}}_i}-t,\beta^{(\eta^{-1} \underline{h}^{-1} \eta; \tau^{-1})}_{\tau^{-1}(M^{\sigma^{-1}}_i), M^{\sigma^{-1}}_i}];\tau^{-1}_{|{W^{\sigma^{-1}}_i}})
\end{multline*}
where now $\beta$ is defined using $\eta^{-1} \underline{g}^{-1}\eta$ via equation \eqref{betaDef}, respectively. We have $\tau^{-1}(i_l)=j_{l+m^{\prime}_l}$. There is a canonical bijection $W^{\sigma}_i \rightarrow W^{\sigma^{-1}}_{i^{-1}}$, so that $M_i^{\sigma} = M_{i^{-1}}^{\sigma^{-1}}$. It follows from the definitions that
\[
\beta^{(\eta^{-1}\underline{h}^{-1} \eta; \tau^{-1})}_{\tau^{-1}(j), j}
=
\eta^{-1} (\beta^{(\underline{h}; \tau)}_{-\tau(j), -j})^{-1} \eta.
\]
Identifying $i \in \theta^{(\sigma^{-1})}_k$ with the inverse of $i \in \theta_k^{(\sigma)}$, the associated component of the final expression agrees with that of the clockwise composition. It follows that the diagram commutes on morphisms, completing the proof.

\item By the discussion preceding the theorem, the statement is non-trivial only when $(\underline{g};\sigma) \in \pi_0((\G\wr\Sigma_N)^{\tor} \git \G\wr\Sigma_N)_{-1}$. In this case, it follows from the previous part and the fact that $(X^N)^{(\underline{g};\sigma)}\git\Lambda^R_{\widehat{\G\wr\Sigma_N}}(\underline{g}; \sigma)$ and $d^{\refl}_{(\underline{g}; \sigma)}(X)$ are quotients of $(X^N)^{(\underline{g};\sigma)}\git\Lambda_{\G\wr\Sigma_N}(\underline{g}; \sigma)$ and $d_{(\underline{g}; \sigma)}(X)$, respectively, by $(\iota_{\omega}, \Theta_{\omega})$ for any $\omega \in C_{\widehat{\G \wr \Sigma_N}}^R(\underline{g};\sigma) \setminus C_{\G \wr \Sigma_N}(\underline{g};\sigma)$. The map $f_{(\underline{g};\sigma)}$ then descends to the desired $B \Z_2$-graded equivalence.\qedhere
\end{enumerate}
\end{proof}

\subsubsection{The power operation  $\{\mathbb{P}^R_N\}_{N \geq 0}$} \label{QRPowerconstr}

Define
\[
\mathbb{P}^R_{ (\underline{g}; \sigma)}:
\QR^{\bullet}_{\G}(X) \rightarrow {^{\pi}}K^{\bullet }_{\Lambda^R_{\widehat{\G\wr\Sigma_N}}(\underline{g}; \sigma)}((X^N)^{(\underline{g}; \sigma)})
\]
as the composition \eqref{pgs2r}. In view of Proposition \ref{prop:QRGlobQuot}, the codomain of $\mathbb{P}^R_{ (\underline{g}; \sigma)}$ is a factor of $\QR^{\bullet}_{\G \wr \Sigma_N}(X^N)$ and we can view $\mathbb{P}^R_{ (\underline{g}; \sigma)}$ as a map $\QR^{\bullet}_{\G}(X) \rightarrow \QR^{\bullet}_{\G \wr \Sigma_N}(X^N)$.

\begin{Def}
For each integer $N \geq 0$, let
\begin{equation*}
\mathbb{P}^R_N :=\prod_{(\underline{g};\sigma)\in\pi_0((\G\wr\Sigma_N)^{\tor} \git_R \widehat{\G\wr\Sigma_N})}\mathbb{P}^R_{(\underline{g}; \sigma)}:
\QR^\bullet_{\G}(X)\rightarrow \QR^\bullet_{\G\wr\Sigma_N}(X^N).
\end{equation*}
\end{Def}

\begin{Thm}
\label{thm:powOpQR}
The operations $\{\mathbb{P}^R_N\}_{N\in \Z_{\geq 0}}$ have the following properties:
\begin{enumerate}[wide,labelwidth=!, labelindent=0pt,label=(\roman*)]
\item $\mathbb{P}^R_0(x)=1$ and $\mathbb{P}^R_1=\id$.

\item Let $x\in \QR^\bullet_{\G}(X)$, $(\underline{g}; \sigma)\in \G\wr\Sigma_N$ and $(\underline{h}; \tau)\in \G\wr\Sigma_M$. The external product of two operations satisfies
$$\mathbb{P}^R_{(\underline{g}; \sigma)}(x)\boxtimes \mathbb{P}^R_{(\underline{h}; \tau)}(x)=\Res^{\Lambda^{R}_{ \widehat{\G\wr\Sigma_{M+N}}}(\underline{g},\underline{h}; \sigma\tau)}_{\Lambda^{R}_{\widehat{\G\wr\Sigma_N}}(\underline{g};\sigma)\times_{\mathsf{O}_2}\Lambda^{R}_{\widehat{\G\wr\Sigma_M}}(\underline{h};\tau)} \mathbb{P}^R_{(\underline{g},\underline{h};\sigma\tau)}(x).$$

\item Let $(\underline{\underline{h}; \tau})\in (\widehat{\G\wr \Sigma_M})^N$ and $\sigma\in\Sigma_N$. The composition of two operations satisfies
\[
\mathbb{P}^R_{((\underline{\underline{h}; \tau}); \sigma)}(\mathbb{P}^R_M(x))=\Res^{\Lambda^{R}_{\widehat{\G\wr\Sigma_{MN}}}(\underline{h}; (\underline{\tau},\sigma))}_{\Lambda^{R}_{\widehat{(\G\wr\Sigma_M)\wr\Sigma_N}}((\underline{\underline{h};\tau}); \sigma)}\mathbb{P}^R_{(\underline{\underline{h}};(\underline{\tau}, \sigma))}(x).
\]

\item The operations preserve external products: if $\underline{(g,h)}\in (\G\times \Hh)^N$
and $\sigma\in \Sigma_N$, then
$$\mathbb{P}^R_{(\underline{(g, h)};\sigma)}(x\boxtimes y)=\Res^{\Lambda^{R}_{\widehat{\G\wr\Sigma_N}}(\underline{g}; \sigma)\times_{\mathsf{O}_2}\Lambda^{R}_{\widehat{\Hh\wr\Sigma_N}}(\underline{h}; \sigma)}
_{\Lambda^{R}_{\widehat{(\G\times \Hh)\wr\Sigma_n}}(\underline{(g, h)};\sigma)}\mathbb{P}^R_{(\underline{g};\sigma)}(x)\boxtimes \mathbb{P}^R_{(\underline{h}; \sigma)}(y).\qedhere$$ \label{main1p}
\end{enumerate}
\label{thm:QRPower}
\end{Thm}

\begin{proof}
The proof is a direct modification of the proof of \cite[Theorem 4.12]{huan2018b}.
\end{proof}

\begin{Rem}
\begin{enumerate}[wide,labelwidth=!, labelindent=0pt,label=(\roman*)]
\item Motivated by \cite[Definition 4.3]{ganter2006}, we regard Theorem \ref{thm:QRPower} as the statement that $\{\mathbb{P}^R_N\}_N$ define a Real $H_{\infty}$-structure on $\QR$.
\item The power operations $\{\mathbb{P}^R_N\}_N$ and those of quasi-elliptic cohomology $\{\mathbb{P}_N\}_N$ \cite[\S4]{huan2018b} are compatible in the sense that the following diagram commutes:
\begin{equation}
\label{eq:EllPowOpCompat}
\begin{tikzpicture}[baseline= (a).base]
\node[scale=1.0] (a) at (0,0){
\begin{tikzcd}[column sep=4.5em,row sep=1.5em]
\QR^{\bullet}_{\G}(X) \arrow{r}[above]{\mathbb{P}^R_N} \arrow{d}[left]{c}& \QR^{\bullet}_{\G \wr\Sigma_N}(X^N) \arrow{d}[right]{c} \\
\Q^{\bullet}_{\G}(X) \arrow{r}[below]{\mathbb{P}_N} & \Q^{\bullet}_{\G \wr\Sigma_N}(X^N).
\end{tikzcd}
};
\end{tikzpicture}
\end{equation}

\item The power operation $\{\mathbb{P}^R_N\}_{N \geq 0}$ extends uniquely to one for Tate $KR$-theory
\[ P^{\str,R}_N: 
\QR^{\bullet}_\G(X)\otimes_{\Z[q^{\pm 1}]}\Z\Lser{q}\rightarrow
\QR^{\bullet}_{\G\wr\Sigma_N}(X^N)\otimes_{\Z[q^{\pm 1}]}\Z\Lser{q}.
\]
This is a Real lift of the stringy power operation $P^{\str}_N$ constructed in \cite[Definition 5.10]{ganter2007}. The obvious analogue of diagram \eqref{eq:EllPowOpCompat} commutes. The operations $\{P^{\str,R}_N\}_{N \geq 0}$ are elliptic in the sense of \cite{ando2001}.
\end{enumerate}
\end{Rem}

\bibliographystyle{amsalpha}
\bibliography{QEllRbib}



\end{document}